\documentclass[10pt,leqno, final]{amsart}
\usepackage[colorlinks,linkcolor=blue,citecolor=blue,urlcolor=blue]{hyperref}
\usepackage{amsfonts,amssymb}
\usepackage{color}
\usepackage{graphicx}
\usepackage[notref,notcite]{showkeys}
\usepackage{mathtools}
\usepackage{extarrows}
\usepackage{algorithmic}   
\usepackage{caption}
\usepackage{subcaption}
\usepackage{tikz}

\allowdisplaybreaks

\usepackage{todonotes}

\newtheorem{theorem}{Theorem}[section]
\newtheorem{lemma}[theorem]{Lemma}
\newtheorem{corollary}[theorem]{Corollary}
\newtheorem{proposition}[theorem]{Proposition}
\newtheorem{assumption}[theorem]{Assumption}

\theoremstyle{definition}
\newtheorem{definition}[theorem]{Definition}
\newtheorem{example}[theorem]{Example}
\newtheorem*{example*}{Example}
\newtheorem{remark}[theorem]{Remark}
\newtheorem{algorithm}[theorem]{Algorithm}

\numberwithin{equation}{section}

\usepackage{my_latex_commands}

\renewcommand{\dist}{\overline{\operatorname{dist}}}
\newcommand{\abst}{\operatorname{dist}}
\newcommand{\tol}{\texttt{\textup{tol}}}

\begin{document}

\title[Adaptive Time Stepping for Rate-Independent Systems]{An Adaptive Time Stepping Scheme for 
Rate-Independent Systems with Non-Convex Energy}\thanks{This research was supported by the German Research Foundation (DFG) under grant 
number~ME~3281/9-1 within the priority program Non-smooth and Complementarity-based
Distributed Parameter Systems: Simulation and Hierarchical Optimization (SPP~1962).}

\author{Merlin Andreia} \address{Technische Universit\"at Dortmund, Fakult\"at f\"ur
  Mathematik, Lehrstuhl LSX, Vogelpothsweg 87, 44227 Dortmund, Germany}
\email{merlin.andreia@tu-dortmund.de}

\author{Christian Meyer} \address{Technische Universit\"at Dortmund, Fakult\"at f\"ur
  Mathematik, Lehrstuhl LSX, Vogelpothsweg 87, 44227 Dortmund, Germany}
\email{christian2.meyer@tu-dortmund.de}

\date{\today} 

\begin{abstract} 
    We investigate a local incremental stationary scheme for the numerical solution of rate-independent systems. 
    Such systems are characterized by a (possibly) non-convex energy and a dissipation potential, 
    which is positively homogeneous of degree one. Due to the non-convexity of the energy, 
    the system does in general not admit a time-continuous solution.
    In order to resolve these potential discontinuities, the algorithm produces a sequence of 
    state variables and physical time points as functions of a curve parameter.    
    The main novelty of our approach in comparison to existing methods is an adaptive choice 
    of the step size for the update of the curve parameter depending on a prescribed tolerance for the 
    residua in the energy-dissipation balance and in a complementarity relation concerning the so-called
    local stability condition. 
    It is proven that, for tolerance tending to zero, the piecewise affine approximations generated by
    the algorithm converge (weakly) to a so-called $\V$-parametrized balanced viscosity solution.
    Numerical experiments illustrate the theoretical findings and show that an adaptive choice of the step size 
    indeed pays off as they lead to a significant increase of the step size during sticking and in viscous jumps.
\end{abstract}

\maketitle


\section{Introduction}
This paper is concerned with the design and analysis of an adaptive time stepping scheme 
for rate-independent systems of the form 
\begin{equation}\label{eq:ris}
    0 \in \partial \RR(\dot z(t)) + D_z \II(t, z(z)) \quad \text{f.a.a.\ } t \in (0,T), 
    \quad z(0) = z_0,
\end{equation}
where $\RR: \XX \to [0,\infty)$ denotes an 1-homogeneous and convex dissipation functional, 
while $\II : \ZZ \to \R$ is a non-convex, but smooth energy functional. The precise assumptions 
on the data and the involved spaces will be given in Section~\ref{sec:assu} below.
Systems of the form \eqref{eq:ris} appear in various applications such as elastoplasticity, 
phase field models for damage evolution, or shape-memory alloys, see e.g.\ \cite{dalMaso, KRZ}
and the references therein.

It is well known that, caused by the interplay between the non-smooth, but convex dissipation and 
the smooth, but non-convex energy, there is in general no solution that fulfills \eqref{eq:ris} a.e.\ in time. 
We refer to the counterexample in \cite[Section~2.3]{sthomas}.
For this reason, several generalized notions of solutions have been introduced in 
the recent past, among them local solutions, global energetic solutions, and balanced viscosity (BV) solutions. 
We refer to \cite{MRRIS} for a comprehensive survey on this matter.
All these notions have in common that they allow for discontinuous solutions, even if the 
data, in particular the force driving the system, are smooth.
This is essential to guarantee the existence of solutions as the aforementioned counterexample illustrates.

In this paper, we focus on a solution concept entitled \emph{$\V$-parametrized BV solution}, which will be defined 
precisely in Section~\ref{sec:sol}.
This concept has two major advantages in comparison to other notions of solution.
First, it arises as limit in a vanishing viscosity approach and as such, solutions of this type 
will only provide jumps, if they are really necessary. Thus, solutions will be continuous in time as long as possible, 
which may be seen physically reasonable in many applications. 
Moreover, since the solution trajectory is parametrized by a curve parameter as the title indicates, 
$\V$-parametrized BV solutions enable a refined resolution of the jumps. 
As a result, a solution that was discontinuous as a function of time, 
becomes smooth in terms of the curve parameter.

Despite this gain of smoothness,  
the numerical approximation of solutions of this type is still a challenging issue.
This is mainly due to the non-uniqueness of $\V$-parametrized BV solutions, 
see e.g.\ the counterexample in \cite[Section~2.4]{MeySie20}. We underline that 
a lack of uniqueness is typical for rate-independent systems with non-convex energies 
and also concerns the other notions of solutions, see \cite{MRRIS, michael, sthomas} for various other counterexamples.
One can therefore not expect that an approximation will converge to a solution unless 
the energy is uniformly convex (at least locally around a given solution trajectory), which guarantees 
uniqueness of the solution (at least locally).
Nevertheless, in case of a uniform refinement of the interval for the curve parameter, 
discretization schemes have been developed that provide subsequences which converge weakly
to $\V$-parametrized BV solutions. 
Such a scheme has first been introduced and analyzed 
in \cite{EM06} in the finite dimensional setting, where $\ZZ = \R^n$.
It is known as \emph{local incremental minimization scheme}. 
In \cite{knees, Neg, KneShc21, siev21}, 
this approach has been transferred to infinite dimensions along with several modifications thereof. 
A combination of such a time stepping scheme with a spatial discretization is investigated in \cite{FElocmin}.
The aforementioned strong convergence in the uniformly convex case including convergence rates is 
established in \cite{MeySie20}. Let us finally mention that incremental minimization schemes are also 
employed to approximate other types of solutions, such as global energetic solutions, 
we exemplarily refer to \cite{bmr, bartels} and the references therein.
 
All aforementioned contributions employ a uniform discretization of the interval for the curve parameter. 
While the discretization of the physical time interval is intrinsically adaptive in context of local 
incremental minimization schemes, the curve parameter was so far only updated by a fixed step size $\tau$ 
in each iteration. One may however essentially distinguish three different regimes characterizing a 
$\V$-parametrized BV solution, that is \emph{rate-independent slip}, \emph{viscous jump}, and 
\emph{sticking}. In the last two regimes, one component of the solution curve $s\mapsto (t(s), z(s))$ is constant 
($t$ in a viscous jump and $z$ during sticking), which calls for coarser step sizes for the curve parameter 
especially in this parts of the curve. On the other hand, the switching points between two different regimes should 
be resolved properly, which may require local refinement.
There is thus a clear intuition that an adaptive choice of the step size for the curve parameter 
makes sense in the context of local incremental minimization schemes.
Nevertheless, to the best of our knowledge, this has not been investigated so far. 
With this work, we aim to develop such an adaptive algorithm, which at the same time provides 
the same convergence properties as the existing methods with uniform step size, i.e., weak 
convergence of subsequences to $\V$-parametrized BV solutions, which is all one can expect
in case of a non-convex energy as explained above.

\subsection{Outline of the Paper} 
Let us give a brief overview over the paper. 
After introducing the notation and our standing assumptions on the data in \eqref{eq:ris} in the rest 
of this introduction, we will precisely define $\V$-parametrized BV solutions and motivate 
their definition in more details in Section~\ref{sec:sol}.
Afterwards, in Section~\ref{sec:algo}, we present the basic version of our adaptive algorithm.
Section~\ref{sec:prelim} is then devoted to several results known from the uniform version of 
the algorithm with constant step size. 
If a proof is entirely analogous to the uniform case, we just state the associated result from the existing 
literature, otherwise we give a proof in the appendix.
In Section~\ref{sec:convRn}, we then investigate the convergence of the algorithm in case $\ZZ = \R^n$. 
As indicated above, we show in Theorem~\ref{thm:convRn} that subsequence of iterates converge weakly to 
$\V$-parametrized BV solutions, if the tolerance in the step size control tends to zero.
Unfortunately, an analogous result for the infinite dimensional case where $\ZZ \neq \R^n$ can only be proven 
under the additional assumption that the meshes generated by the adaptive algorithm are nested. 
In Section~\ref{sec:convinf}, we give an example for an algorithm ensuring this assumption and 
adapt the convergence result to this case in Theorem~\ref{konvergenz}.
Finally, we illustrate our theoretical findings with two numerical tests in Section~\ref{sec:numerics} 
showing that the adaptive step size control behaves as desired and refines the step size close to 
switching points.

\subsection{Notation}
If a normed linear space $X$ is continuously embedded in another space $Y$, we write $X\embed Y$. 
Compact embeddings are denoted by $X\embed^c Y$.
The dual pairing between $X$ and its dual $X^*$ is denoted by $\dual{\cdot}{\cdot}_{X^*, X}$, where 
the subscript is sometimes neglected, if there is no risk for ambiguity.
Throughout the paper, $c$ and $C$ denote positive generic constants.
Given a Lebesgue measurable set $M\subset \R$, we denote its Lebesgue measure by $|M|$.

\subsection{Standing Assumptions}\label{sec:assu}
We now introduce the precise assumptions on the data in \eqref{eq:ris}.

\subsubsection*{Spaces}
The spaces underlying \eqref{eq:ris} are $\XX$, $\VV$, $\WW$, and $\ZZ$, 
where $\ZZ$, $\VV$ are Hilbert and $\XX$, $\WW$ Banach spaces such that 
\begin{align*}
	\ZZ\embed^c \WW \embed \VV\embed \XX
\end{align*}
and $\ZZ$ is dense in $\VV$. 
Furthermore, $\V\in\LL(\VV,\VV^*)$ is a linear, coercive, and self-adjoint operator, i.e., 
there is a constant $\gamma > 0$ such that 
\begin{equation*}
    \dual{\V v}{v}_{\VV^*,\VV} \geq \gamma \|v\|^2_\VV 
    \quad \forall\, v\in \VV.
\end{equation*}
We sometimes equip $\VV$ with the equivalent norm $\|.\|_\V:\VV\to\R$ defined by
$\|v\|_\V=\langle \V v,v\rangle_{\VV^*,\VV}^{\frac{1}{2}}$.
Clearly, $\V$ is invertible and by $\eta \ni \VV^* \mapsto \dual{\eta}{\V^{-1}\eta}_{\VV^*,\VV}^{\frac{1}{2}}\in \R$
we obtain a norm on $\VV^*$, which we denote by $\|\eta\|_{\V^{-1}}$.

\subsubsection*{Dissipation}
For the dissipation $\RR:\XX\to [0,\infty)$, we assume
\begin{align}
	&\RR \mbox{~is~proper,~convex,~and~lower~semicontinuous}, \label{R1}\\
	&\RR \text{ is positive 1-homogeneous, i.e., }  
	\RR(\lambda v)=\lambda\RR(v) ~\forall \,v\in\XX,~\lambda>0, \label{R2}\\
 	&\exists \,\rho, R >0, \mbox{~such~that}~\rho\,\|v\|_\XX\leq\RR(v)\leq R\,\|v\|_\VV~\forall \,v\in\VV.  \label{R3}
\end{align}
In all what follows, we will consider $\RR$ as a mapping from $\ZZ$ to $\R$. 
Accordingly, its Fenchel conjugate and its convex subdifferential are considered as a mapping and 
a subset, respectively, in $\ZZ^*$.

\subsubsection*{Initial data} The initial state $z_0$ satisfies $z_0 \in \ZZ$ and 
$D_z \II(0,z_0) \in \VV^*$.

\subsubsection*{Energy}
The energy functional $\II:[0,T]\times\ZZ\to\R$ is of the form
\begin{equation}\label{eq:energydef}
	\II(t,z)=\frac{1}{2}\langle Az,z\rangle_{\ZZ^*,\ZZ} +\FF(z)-f(t,z),
\end{equation}
where $A\in\LL(\ZZ,\ZZ^*)$ is coercive and self-adjoint such that there is a constant $\alpha > 0$ so that 
\begin{equation*}
    \langle Az,z\rangle_{\ZZ^*,\ZZ}\geq\alpha\|z\|_\ZZ^2 \quad \forall \, z \in \ZZ.
\end{equation*}
The functional $\FF:\ZZ\to\R$ fulfills
\begin{align}
	&\FF\in C^2(\ZZ;\R),~\FF\geq 0 \label{F1}\\
	&D_z\FF\in C^1(\ZZ,\VV^*)~\mbox{and}~\|D_z^2\FF(z)v\|_{\VV^*}\leq C(1+\|z\|_\ZZ^q)\|v\|_\ZZ~\forall\, z,v\in\ZZ, \label{F2}
\end{align}
with constants $C>0$ and $q\geq1$. Moreover, we require the derivate $D_z\FF$ to be continuous in the 
following sense
\begin{equation}\label{eq:DFFweakcont}
	z_k\rightharpoonup z~\mbox{in}~\ZZ\implies 
	D_z\FF(z_k)\rightharpoonup D_z\FF(z)~\mbox{in}~\ZZ^* ~\mbox{for} ~k\to\infty.
\end{equation}
The time dependent part $f:[0,T]\times \ZZ\to \R$ is supposed to satisfy the following properties 
\begin{align}
	&f\in C^1([0,T]\times\ZZ;\R),\label{f1} \\
	&f(t,z)\leq c(\|z\|_\ZZ+1),~|\partial_tf(t,z)|\leq \mu(\|z\|_\ZZ+1) \label{f2} \\
	&|\langle D_zf(t_1,z_1)-D_zf(t_2,z_2),v\rangle_{\VV^*,\VV} |\leq \nu(|t_1-t_2|+\|z_1-z_2\|_\WW)\|v\|_\VV \label{f3}
\end{align}
for all $t,t_1,t_2\in[0,T],~z,z_1,z_2\in\ZZ$, and $v\in\VV$ with constants  $c,\mu,\nu>0$.
Moreover, we require that
\begin{align}
	&f(t,.) ~\mbox{is weakly continuous for all}~t\in[0,T], \label{eq:fweakinz}\\
	&\partial_t f(t_k,z_k)\to \partial_t f(t,z)~\mbox{holds for all sequences}~t_k\to t~\mbox{in}~\R ~\mbox{and}~z_k\rightharpoonup z~\mbox{in}~\ZZ~\mbox{for}~k\to\infty. \label{dtfkonv}
\end{align}

We end this section with several useful results that immediately follow from 
from the above assumptions on the energy.

\begin{lemma}\label{lem:fFFcont}
    The functionals $\FF$ and $f$ are weakly continuous in the following sense:
    \begin{align*}
	    z_k\rightharpoonup z \mbox{ in }\ZZ & \implies \FF(z_k)\to\FF(z),\\
        t_k\to t~\mbox{in}~\R,~z_k\rightharpoonup z~\mbox{in}~\ZZ  & \implies f(t_k,z_k)\to f(t,z),
    \end{align*}
    as $k\to \infty$.
\end{lemma}

\begin{proof}
    In view of the required regularity of $\FF$, we can apply the mean value theorem, 
    which yields the existence of $\zeta_k\in(0,1)$ such that 
    \begin{align*}
	    \|\FF(z_k)-\FF(z)\|_\ZZ
	    & \leq |\langle D_z\FF(z),z_k-z\rangle_{\VV^*,\VV}|+\frac{1}{2}|D_z^2\FF(z+\zeta_k(z_k-z))(z_k-z)^2|\\
	    &\leq \|D_z\FF(z)\|_{\VV^*}\|z_k-z\|_\VV
	    +\frac{C}{2}(1+\|z+\zeta_k(z_k-z)\|_\ZZ^q)\|z_k-z\|_\ZZ\|z_k-z\|_\VV \\
	    &\to 0,~k\to\infty, 
    \end{align*}
    where we exploited \eqref{F2} and the weak convergence of $\{z_k\}$ in $\ZZ$, which gives the boundedness of 
    $\{z_k\}$ in $\ZZ$ and, due to the compactness of $\ZZ\hookrightarrow \VV$ by assumption, 
    the strong convergence $z_k \to z$ in $\VV$.
    
    For the second statement, we argue similarly. Again the mean value theorem implies 
    the existence of an intermediate value $\zeta_k$ such that 
    $f(t_k,z_k) = \partial_t f(\zeta_k,z_k)(t_k-t)+f(t,z_k)$. Then, \eqref{eq:fweakinz} and \eqref{dtfkonv} 
    immediately imply the assertion.
\end{proof}

\begin{lemma}
    The energy functional $\II$ satisfies $\II\in C^1([0,T]\times\ZZ;\R)$ and 
    \begin{align}
	    &t_k\to t \mbox{ in } \R,~z_k\rightharpoonup z \mbox{ in }\ZZ
	    \implies \II(t,z)\leq\liminf_{k\to\infty}\II(t_k,z_k), \label{iuhs}\\ 
	    &t_k\to t~\mbox{in}~\R,~ z_k\rightharpoonup z ~\mbox{in}~\ZZ
	    \implies D_z\II(t_k,z_k)\rightharpoonup D_z\II(t,z)~\mbox{in}~\ZZ^*~\mbox{for}~k\to\infty. \label{ikonv}
\end{align}
\end{lemma}

\begin{proof}
    The differentiability of $\II$ follows immediately from the required regularity of $f$ and $\FF$.
    The lower semicontinuity in \eqref{iuhs} is an immediate consequence of Lemma~\ref{lem:fFFcont} and
    the coercivity of $A$.
    To prove \eqref{ikonv}, observe that \eqref{f3} and the compactness of $\ZZ\embed \WW$ imply that 
    $D_z f(t_k, z_k)$ converge strongly in $\VV^*$ and thus even more in $\ZZ^*$. 
    Together with \eqref{eq:DFFweakcont} and the linearity of $A$, this gives the assertion.
\end{proof}

Moreover, exploiting the coercivity of $A$ together with (\ref{F1}) and (\ref{f2}) gives the lower bound estimate
\begin{equation}
	\II(t,z)\geq \frac{\alpha}{2}\|z\|_\ZZ^2-c(\|z\|_\ZZ+1)\geq \|z\|_\ZZ-\tilde{c} \quad \forall\, t\in [0,T],~z\in\ZZ, \label{ibesch}
\end{equation}
where $\tilde{c}=\frac{(c+1)^2}{2\alpha}+c$. Hence, combining this and (\ref{f2}) results in
\begin{align*}
	|\partial_t\II(t,z)|\leq \mu(\|z\|_\ZZ+1)\leq \mu(\II(t,z)+\beta),
\end{align*}
where $\beta=\tilde{c}+1$. Now integrating over $(s,t)$ and $(t,s)$, respectively, and using Gronwall's lemma leads to

\begin{lemma}\label{lem:gronwall}
    For all $z\in\ZZ$ and  and all $s,t\in[0,T]$, there holds
    \begin{align}
	    \II(t,z)+\beta & \leq(\II(s,z)+\beta)\exp(\mu|t-s|), \label{gr1}\\ 
	    |\partial_t \II(t,z)| & \leq \mu (\II(s,z)+\beta)\exp(\mu|t-s|).\label{gr2}
    \end{align}
\end{lemma}

\begin{example}\label{ex:bspsystem}
    Let us give an example for a rate-independent system fulfilling the above assumptions. 
    For this purpose, let $\Omega\subset \R^d$, $d \in \N$, be a bounded Lipschitz-domain in the sense of 
    \cite[Def.~1.2.1.2]{Grisvard11}. Then we choose for the spaces 
    \begin{equation*}
        \ZZ := H^1_0(\Omega), \quad \WW = \VV := L^2(\Omega), \quad \XX := L^1(\Omega), 
    \end{equation*}
    for the dissipation functional
    \begin{equation*}
        \RR(z) := \|z\|_{L^1(\Omega)},
    \end{equation*}
    and for the energy
    \begin{equation}\label{eq:loadex}
        A := -\laplace, \quad \FF(z) := \int_\Omega g(z)(x) \,\d x,
        \quad f(t,z) := \dual{\ell(t)}{z}_{\VV^*,\VV},
    \end{equation}
    where $g: L^p(\Omega) \to L^2(\Omega)$, $p < 2d/(d-2)$,  
    is a Nemyzkii operator that is twice continuously Fr\'echet-differentiable
    and $\ell : [0,T] \to \VV^*$ is continuously Fr\'echet-differentiable, too. Then, it is easily seen by means 
    of Sobolev embeddings that the example fits into the above setting,
    provided that $\frac{\d}{\d t}\ell$ is Lipschitz continuous and $g$ is non-negative and satisfies 
    \begin{equation}\label{eq:gsecond}
        \| g''(z) v \|_{L^2(\Omega)} \leq C(1 + \|z\|_{L^p(\Omega)}^q) \|v\|_{L^p(\Omega)}.
    \end{equation}
    We will come back to an example of this type in our numerical experiments in Section~\ref{sec:numerics} below.
\end{example}


\section{Solution Concepts in Brief}\label{sec:sol}

Before we present our adaptive time stepping scheme, let us introduce the notion 
of solution of \eqref{eq:ris} underlying this scheme and give a brief overview of the various solution 
concepts for \eqref{eq:ris} without claiming to be exhaustive. For a more detailed survey of 
the various notions of solutions to rate independent systems, 
we refer to \cite{MRRIS} and the references therein.

The most natural solution concept for \eqref{eq:ris} is probably the notion of a 
\emph{differential solution}. Here, one seeks for a function $z \in W^{1,1}(0,T;\ZZ)$ with $z(0) = z_0$ 
satisfying the differential inclusion in \eqref{eq:ris} almost everywhere. It however turns out that, 
if the energy is not uniformly convex, the existence of such a solution can in general not be expected, 
as e.g.\ the counterexample in \cite[Example~1.8.1]{MRRIS} illustrates. 
For this reason, Alexander Mielke and his co-authors came up with the meanwhile classical concept of 
\emph{global energetic solutions}, see \cite{MT99, MTL02, MT04}. 
This concept involves a global stability condition together with an energy balance identity, 
and existence of solutions of this type can be shown under rather mild assumptions on dissipation and energy, 
in particular without assuming the energy functional to be convex, we refer to \cite[Thm.~2.1.6]{MRRIS}.
However, caused by the global stability condition, 
global energetic solutions have the tendency to jump as early as possible to global minimizers in the 
energy landscape and in this way may cross energy barriers, which can be seen unphysical in applications.

A solution concept that overcomes this drawback is the concept of 
\emph{parametrized  balanced viscosity (BV) solutions}, 
which also forms the basis for our adaptive time stepping scheme.
To the best of our knowledge, it was first introduced in \cite{EM06}, but has by now been analyzed by 
various authors in multiple aspects, we only refer to \cite{MRS12} and the reference therein. 
As the denomination indicates, we focus on BV solutions in parametrized form, where physical time and 
state variable are given as functions of a parameter $s$. 
For our particular setting, the precise definition of a parametrized BV solution reads as follows:

\begin{definition}[$\V$-parametrized BV solution] \label{param}
	We call a tuple $(\hat t, \hat z)$ $\V$-parametrized BV solution of the rate-independent system \eqref{eq:ris},
	if there exists $S\geq T$ such that the following is fulfilled:
	\begin{itemize}
		\item Regularity:
			\begin{align}
				\hat t\in W^{1,\infty}(0,S),\quad \hat z\in W^{1,\infty}(0,S;\VV)\cap L^\infty(0,S;\ZZ), \label{reg}
			\end{align}
		\item Initial and end time condition:
			\begin{align}
				\hat z(0)=z_0,~\hat t(0)=0,~\hat t(S)=T, \label{start}
			\end{align}
		\item Complementary condition:
		 	\begin{subequations}
				\begin{align}
					& \hat t'(s)\geq 0,~\hat t'(s)+\|\hat z'(s)\|_\V\leq 1  & & \text{f.a.a.}~s\in(0,S),\label{eq:komp1}\\
					& \hat t'(s)\, \dist_{\VV^*}\{-D_z\II(\hat t(s),\hat z(s)),\partial\RR(0)\}=0
					& & \text{f.a.a.}~s\in(0,S), \label{komp}
				\end{align}		 	
		 	\end{subequations}
			where $\dist_{\VV^*}\{\,\cdot\,,\partial\RR(0)\} : \ZZ^* \to [0,\infty]$ is defined by
			\begin{equation}\label{eq:distdef}
			    \dist_{\VV^*}\{\eta,\partial\RR(0)\}
			    := \inf\{\|w\|_{\V^{-1}}: w \in \VV^* \cap (\partial\RR(0) - \eta)\}
			\end{equation}
			(with the usual convention that $\inf \emptyset = \infty$),
		\item Energy equality:
			\begin{equation}
				\begin{aligned}
					\II(\hat t(s),\hat z(s))+\int_0^s \RR(\hat z'(\sigma)) + \|\hat z'(\sigma)\|_\V 
					\,\dist_{\VV^*}\{-D_z\II(\hat t(\sigma),\hat z(\sigma)),\partial\RR(0)\} \, d\sigma 
					\qquad\qquad &\\[-1ex]
					= \II(0,z_0)+\int_0^s \partial_t\II(\hat t(\sigma),\hat z(\sigma))\hat t'(\sigma)d\sigma 
					\quad \forall s\in[0,S]. & \label{energie}
				\end{aligned}
			\end{equation}
	\end{itemize}
	We call a $\V$-parametrized solution non-degenerate, if there is a constant $\delta > 0$ such that 
	$\hat t'(s)+\|\hat z'(s)\|_\V \geq \delta$ a.e.\ in $(0,S)$. If $\delta =1$, then the solution is called 
	normalized $\V$-parametrized BV solution.
\end{definition}

Some words concerning this definition are in order. The interpretation of Definition~\ref{param} is probably best 
understood, if one introduces the function $\lambda : [0,S] \to [0, \infty]$ by 
\begin{equation}\label{eq:lambdadef}
\begin{aligned}
    \lambda(s) := 
    \begin{cases}
        \frac{1}{\|\hat z'(s)\|_{\V}}\dist_{\VV^*}\{-D_z\II(\hat t(s),\hat z(s)),\partial\RR(0)\}, &
        \text{if } \hat z'(s)\neq 0, \\ 
         0, & \text{else} 
    \end{cases}        
\end{aligned}
\end{equation}
and applies the chain rule to 
\eqref{energie} to obtain the pointwise relation 
\begin{equation}\label{eq:viscouseqlimit}
    0 \in \partial\RR(\hat z'(s)) + \lambda(s) \hat z'(s) + D_z\II(\hat t(s),\hat z(s)),
\end{equation}
cf.\ \cite{MRS12}, 
which is \eqref{eq:ris} enriched with the additional viscous term $\lambda(s) \hat z'(s)$.
We observe that, thanks to the complementarity relation in \eqref{komp}, this additional term only appears, 
if $\hat t'(s) = 0$, i.e., if the physical time stands still, which corresponds to a discontinuous jump of the system. 
In case of a jump, the above equation thus describes the 
viscous transition of the state variable $\hat z$ through the complement of the set of local stability
$\{(t, z)\in \R\times \ZZ : - D_z \II(t, z) \in \partial \RR(0)\}$.
Of course, one may re-parametrize the curve $s\mapsto (\hat t(s), \hat z(s))$ such that 
a parametrized BV solution is intrinsically non-unique. Here, we have chosen a parametrization by arc length
w.r.t.\ the $\V$-norm, another parametrization that is often used involves the 
so-called vanishing viscosity contact potential. One may also define BV solutions in the physical time regime 
without parametrization, leading to a modified energy balance.
For more details on (parametrized) BV solutions, we refer to \cite{MRS12} and the references therein.

\begin{remark}[{\cite[Lemma~5.6]{KRZ}}, {\cite[Lemma~2.4.6]{michael}}]\label{rem:energyid}
    It is to be noted that, provided a chain rule can be applied, an energy inequality is sufficient 
    for \eqref{energie} to hold. To be more precise, if
	$(\hat t,\hat z)\in W^{1,\infty}(0,S)\times W^{1,1}(0,S;\ZZ)$ 
	with $\dist_{\VV^*}\{-D_z\II(\hat t(s),\hat z(s)),\partial\RR(0)\}< \infty$ f.a.a.\ $s\in(0,S)$, 
	then the inequality
    \begin{equation}\label{eq:enerineq}
	\begin{aligned}
		0\leq&~\II(\hat t(s_2),\hat z(s_2))-\II(\hat t(s_{1}),\hat z(s_{1}))
		-\int_{s_{1}}^{s_{2}} \partial_t\II(\hat t(\sigma),\hat z(\sigma))\hat t'(\sigma)d\sigma \\
 		&+\int_{s_{1}}^{s_{2}}\RR(\hat z'(\sigma))
 		+\|\hat z'(\sigma)\|_\V\, \dist_{\VV^*}\{-D_z\II(\hat t(\sigma),\hat z(\sigma)),\partial\RR(0)\}d\sigma 
	\end{aligned}
    \end{equation}
	is automatically fulfilled for all $s_1,s_2\in[0,S]$ with $s_1\leq s_2$.
	The required chain rule for the energy $\II$ is applicable here thanks to our 
	standing assumptions, cf., e.g., \cite[Lemma~A.2.5]{michael}.	
\end{remark}

While the existence of parametrized BV solutions is frequently established by 
a vanishing viscosity approach (as the name indicates), it is also possible to prove their existence by means of 
tailored time stepping schemes and a limit analysis for time step size tending to zero.
The corresponding schemes are known as \emph{local incremental minimization} or 
\emph{local stationary scheme}.
They have been analyzed in \cite{EM06} for the finite dimensional case
and in \cite{knees, FElocmin, siev21, KneShc21} for the infinite dimensional setting. 
Moreover, as the numerical experiments in \cite{FElocmin, siev21} demonstrate, 
these schemes can also be realized in practice and used for numerical computations.
So far however, all these schemes employ a uniform step size for the curve parameter $s$. 
In the following, we present a method that chooses an adaptive step size, but still can be shown 
to converge to a $\V$-parametrized BV solution in the sense of Definition~\ref{param}.


\section{An Adaptive Incremental Stationary Scheme}\label{sec:algo}

For the construction of our adaptive algorithm, we need the following
\begin{definition}
    By $I_{\tau_k}:\ZZ\to\R$, we denote the indicator functional of $B_\V(0,\tau_k)$, i.e.,
    \begin{equation*}
        I_{\tau_k}(z)=\begin{cases} 0,~&\mbox{if}~ \|z\|_\V\leq\tau_k\\ \infty, &\mbox{else}.\end{cases}
    \end{equation*}        
\end{definition}

With $I_{\tau_k}$ at hand, the algorithm then reads as follows:

\begin{algorithm}[Adaptive Incremental Stationary Scheme]\label{alg:adaptlocmin}
\ 
\begin{algorithmic}[1]
	\STATE\label{it:initial} Set $t_0=0,~s_0=0,~k=1$ and choose $\tol>0,~\tau_k>0$.
	\WHILE {$t_{k-1}<T$}\label{it:abbruch}
        	\STATE\label{it:stat} Calculate a solution $z_k$ of
        		\begin{equation}
		 0\in\partial(\RR+I_{\tau_k})(z_k-z_{k-1})+D_z\II(t_{k-1},z_k), \label{alg1}
		\end{equation}
		which, moreover, fulfills
		\begin{equation}
		\II(t_{k-1},z_k)+\RR(z_k-z_{k-1})\leq\II(t_{k-1},z_{k-1}). \label{alg2}
		\end{equation}
		\STATE\label{it:update} Time update: 
        \begin{equation*}
            t_k=t_{k-1}+\tau_k-\|z_k-z_{k-1}\|_\V, \quad s_k=\sum_{i=1}^{k}\tau_i.
        \end{equation*}        		
        \STATE \label{it:4} Define for $s\in [s_{k-1},s_k)$ the affine and constant interpolants 
        \begin{equation*}
        \begin{aligned}
            \hat{z}(s) &=z_{k-1}+\frac{s-s_{k-1}}{\tau_k}(z_k-z_{k-1}), & &
            \hat{t}(s)=t_{k-1}+\frac{s-s_{k-1}}{\tau_k}(t_k-t_{k-1}), \\
            \bar{z}(s) &=z_k, & &
            \underline{t}(s)=t_{k-1}.        
        \end{aligned}
        \end{equation*}
	\STATE\label{it:5} Compute the integrals
        \begin{align}
            &I_1^k =\int_{s_{k-1}}^{s_k}
            \hat{t}'(s)\dist_{\VV^*}\{-D_z\II(\hat{t}(s),\hat{z}(s)),\partial\RR(0)\}ds\label{i1}\\
            &I_2^k =\int_{s_{k-1}}^{s_k} 
            \begin{aligned}[t]
                \Big[ & \langle D_z\II(\hat{t}(s),\hat{z}(s))-D_z\II(\underline{t}(s),\bar{z}(s)),
                \hat{z}'(s)\rangle_{\ZZ^*,\ZZ}\\[-0.5ex]
                & + \|\hat{z}'(s)\|_\V 
                \begin{aligned}[t]
                    \Big( & \dist_{\VV^*}\{-D_z\II(\hat{t}(s),\hat{z}(s)),\partial\RR(0)\} \\[-0.5ex]
                    & -\dist_{\VV^*}\{-D_z\II(\underline{t}(s),\bar{z}(s)),\partial\RR(0)\} \Big)\Big] ds .
                \end{aligned}                                
            \end{aligned}\label{i2}
		\end{align}
	\IF{$I_1^k < \tol$ and $I_2^k < \tol$}\label{it:refine} 
		\STATE Set $k=k+1$.
	    	\IF {$I_1^k<\frac{\tol}{2}$ and $I_2^k<\frac{\tol}{2}$}
                		\STATE\label{it:10} Set $\tau_{k+1}=2\,\tau_k$.
        		\ENDIF
	\ELSE
        		\STATE\label{it:reduc} Set $\tau_k=\frac{\tau_k}{2}$ and \textbf{go to} Step~\ref{it:stat}.
	\ENDIF\label{it:13}
	
\ENDWHILE
\end{algorithmic}
\end{algorithm}

It is to be noted that the above algorithm coincides with the local stationary scheme from 
\cite{michael, siev21} except for the adaptive choice of the step size $\tau_k$ in 
Steps~\ref{it:5}--\ref{it:13}. As mentioned before, in contrast to Algorithm~\ref{alg:adaptlocmin}, 
the step size in \cite{michael, siev21} is fixed throughout the whole iteration.

Let us first show that each iteration of the above algorithm is well posed.

\begin{lemma}
    Given $t_k \geq 0$ and $z_{k-1}\in \ZZ$, there exists at least one $z_k\in \ZZ$ fulfilling 
    \eqref{alg1} and \eqref{alg2}.
\end{lemma}

\begin{proof}
    Owing to the differentiability of $\II$ and the convexity of $\RR$, \eqref{alg1} is nothing else than the 
    necessary optimality condition of 
    \begin{equation}\label{min}
    	z_k\in \mbox{argmin}\{\II(t_{k-1},z)+\RR(z-z_{k-1}): \; z\in\ZZ, \,\|z-z_{k-1}\|_\V\leq\tau\}.  
    \end{equation}
    Now, as $\ZZ$ is reflexive, $\II(t_{k-1}, \cdot)$ is weakly lower semicontinuous by 
    \eqref{iuhs} and the same holds for $\RR$, since it is convex and continuous, 
    and $\II$ is radially unbounded by \eqref{ibesch}, the direct method gives the existence 
    of a solution to \eqref{min} and thus to \eqref{alg1}. As it is a global minimum and $z_{k-1}$ is feasible, 
    \eqref{alg2} follows immediately.
\end{proof}

\begin{remark}
    The original incremental minimization scheme from \cite{EM06} is based on \eqref{min} instead of 
    \eqref{alg1}--\eqref{alg2}. However, any optimization algorithm can hardly be guaranteed to 
    solve a non-convex problem up to (global) optimality. What can be shown instead is convergence to 
    stationary points, at least under suitable assumptions. From an algorithmic point of view, it therefore 
    makes more sense to require the stationarity condition in \eqref{alg1} instead of solving \eqref{min}. 
    The second condition in \eqref{alg2} is not restrictive in this context, provided that 
    one employs a descent method, e.g., via backtracking or other line search methods. 
\end{remark}

\begin{remark}\label{rem:time}
    It is to be noted that, due to the indicator functional $I_{\tau_k}$ involved in \eqref{alg1}, 
    there always holds $\|z_k - z_{k-1}\|_\V \leq \tau_k$ such that $t_k \geq t_{k-1}$ by the time update
    in Step~\ref{it:update}, i.e., the algorithm does not go backwards in time.
\end{remark}

As we will see below, the integrals in Step~\ref{it:5} are well defined, see 
Lemma~\ref{lem:I1I2} below. 
Therefore, each iteration of the algorithm is well defined, too. 
Moreover, the overall iteration is well posed in the sense that the end time $T$ is reached after 
finitely many iterations, as we will see in Proposition~\ref{endtime} below. 
Let us shortly comment on the update of the step size in Steps~\ref{it:5}--\ref{it:13}.
As we shall see in the next section, 
the integrals $I_1^k$, $I_2^k$ calculated in Step~\ref{it:5} measure the residuum that arises, 
if one inserts the affine interpolants from Step~\ref{it:4} in the complementarity relation in \eqref{komp} 
and in the energy identity in \eqref{energie}, see~Lemma~\ref{fehler} below.
The aim of the algorithm is thus to drive this residuum below a given tolerance $\tol$. 
However, given a sequence $\{\tol\}_{n\in\N}$ tending to zero with asscoiated affine interpolants 
$\{(\hat t_n, \hat z_n)\}_{n\in \N}$ resulting from Algorithm~\ref{alg:adaptlocmin}, there is in general no hope
that $(\hat t_n, \hat z_n)$ will converge (strongly) to a $\V$-parametrized BV solution, 
although the residuum tends to zero. This is due to the non-uniqueness of $\V$-parametrized BV solutions and
is even the case 
for a uniform refinement of the step size, as demonstrated by the counterexample in \cite[Section~2.4]{MeySie20}.
Nevertheless, for a uniform step size, one can show that subsequences converge weakly to 
a $\V$-parametrized BV solution and, as we will see in the following, the same holds true for our 
adaptive algorithm. 


\section{Preliminaries and Known Results}\label{sec:prelim}

For the final proof of our convergence results in Sections~\ref{sec:convRn} and \ref{sec:convinf}, 
we need several auxiliary results that are established in the following. 
We point out that most of the proofs are very similar to the case with a uniform step size 
as discussed in \cite{knees, FElocmin}. Therefore, we omit a proof, when it is entirely analogous, 
or postpone it to the appendix.

Moreover, it is important to note that all proofs of this section do \emph{neither} depend on the 
choice of the step size \emph{nor} on the tolerance, i.e., if we choose an update rule for the steps size, 
that differs from the one in Steps~\ref{it:refine}--\ref{it:13}, the results of this section still apply. 
In fact, an inspection of the proofs in this section shows that 
the arguments are only based on the construction of $z_k$ and $t_k$ 
in Step~\ref{it:stat} and \ref{it:update}, see also Remark~\ref{rem:tau_egal} below.

We start with a reformulation of \eqref{alg1}, which follows from classical results of convex analysis.

\begin{lemma}[{\cite[Proposition~2.2]{knees}, \cite[Lemma~3.1]{FElocmin}, \cite[Lemma~3.2.1]{michael}}]
\label{lem:stat}
    Given $\tau_k> 0$, $t_{k-1} \geq 0$, and $z_k\in \ZZ$, the differential inclusion in \eqref{alg1} 
    is equivalent to the existence of a Lagrange multiplier $\lambda_k\geq0$ such that
	\begin{align}
		&\lambda_k(\|z_k-z_{k-1}\|_\V-\tau_k)=0 \label{eig1} \\
		&\tau_k \dist_{\VV^*}\{-D_z\II(t_{k-1},z_k),\partial\RR(0)\}=\lambda_k\|z_k-z_{k-1}\|_\V^2  \label{eig2} \\
		&\RR(z_k-z_{k-1})+\tau_k\dist_{\VV^*}\{-D_z\II(t_{k-1},z_k),\partial\RR(0)\}
		=\langle -D_z\II(t_{k-1},z_k), z_k-z_{k-1}\rangle_{\ZZ^*,\ZZ}  \label{eig3} \\
		&\RR(v)\geq -\langle \lambda_k \V(z_k-z_{k-1})+D_z\II(t_{k-1},z_k),v \rangle_{\ZZ^*,\ZZ} 
		\quad \forall \, v\in \ZZ. \label{eig4}
	\end{align}
\end{lemma}

The next lemma has been proven in various publications for the case of a uniform step size.
Its proof is only based on the condition \eqref{alg2} and on the standing assumptions on
$\RR$ and $\II$, in particular Lemma~\ref{lem:gronwall},
and therefore readily carries over to an adaptive choice of $\tau_k$ as in Algorithm~\ref{alg:adaptlocmin}.

\begin{lemma}[{\cite[Theorem~2.1.5]{MRRIS}, \cite[Proposition~2.1]{knees}, \cite[Lemma~3.2]{FElocmin}, 
\cite[Lemma~3.2.4]{michael}}]
\label{irbeschr}
    For all $k\in \N$, the iterates $(t_k, z_k)$ of Algorithm~\ref{alg:adaptlocmin} satisfy
	\begin{align*}
		\II(t_k,z_k)+\sum_{i=1}^{k}\RR(z_i-z_{i-1})\leq (\beta+\II(0,z_0))\exp(\mu T).
	\end{align*}
\end{lemma}

If we combine the estimate \eqref{ibesch} with the above lemma and the non-negativity of $\RR$, we obtain
\begin{equation*}
    \|z_k\|_\ZZ\leq\II(t_k,z_k)+\tilde{c}\leq (\beta+\II(0,z_0))\exp(\mu T)+\tilde{c},
\end{equation*}
i.e., we have the following

\begin{corollary}\label{cor:zkbound}
	For every choice of step sizes $\tau_k>0$, the iterates of Algorithm~\ref{alg:adaptlocmin} satisfy
	\begin{align}
		\|z_k\|_\ZZ \leq C \quad \forall\, k\in\N \label{zbeschr}
	\end{align}
	with a constant $C>0$ independent of $\tau_k$.
\end{corollary} 

Using our assumptions \eqref{F2} and \eqref{f3} on the components $\FF$ and $f$ of the energy functional 
in combination with the lower bound of the dissipation functional from \eqref{R3} and Ehrling's lemma, 
one finds the following

\begin{lemma}[{\cite[Lemma 3.5]{FElocmin}, \cite[Lemma 3.2.10]{michael}}]
    Let $\epsilon > 0$ be arbitrary. Then, for all $r>0$, 
    there exists a constant $C_{\epsilon,r}>0$  such that
	\begin{align}
		|\langle D_z\FF(z_1)-D_z\FF(z_2),z_1-z_2\rangle_{\VV^*,\VV}|
		\leq \epsilon\, \|z_1-z_2\|_\ZZ^2+C_{\epsilon,r}\,\RR(z_1-z_2)\|z_1-z_2\|_\VV \label{estF}
	\end{align}
	holds true for all $z_1,z_2\in B_\ZZ(0,r)$. Moreover, there is a constant $c_\epsilon>0$ such that
    \begin{equation}\label{estf}
	\begin{aligned}
	 &|\langle D_zf(t_1,z_1)-D_zf(t_2,z_2),v\rangle_{\VV^*,\VV}| \\\
		& \qquad\qquad \leq \nu\, |t_1-t_2|\,\|v\|_\VV
		+\epsilon\,\|z_1-z_2\|_\ZZ\,\|v\|_\ZZ+c_\epsilon\,\RR(z_1-z_2)\|v\|_\VV 
	\end{aligned}    
    \end{equation}
	holds for all $t_1,t_2\in [0,T]$, $z_1,z_2, v\in\ZZ$.
\end{lemma}

\begin{lemma}\label{lem:sumz}
    There exists a constant $C>0$ such that, for every choice of the step size $\tau_k$, $k\in \N$, there holds 
	\begin{equation}
		\lambda_{k+1}\|z_{k+1}-z_k\|_\V + \sum_{i=0}^{k}\|z_{i+1}-z_{i}\|_\ZZ \leq C \label{sumz} 
	\end{equation}
	for all $k\geq 0$.
\end{lemma}

\begin{proof}
    The proof is similar to  the ones of \cite[Proposition~3.6]{FElocmin} and \cite[Proposition~2.3]{knees}, 
    but, since we need some parts of the proof for later reference and, in particular, the estimate of 
    $\sum_{i=0}^{k}\|z_{i+1}-z_{i}\|_\ZZ$ can be seen as the basis for the overall convergence analysis, 
    we present the arguments in detail for convenience of the reader.

	Let $k\in \N$ be arbitrary. We start with inserting $v=z_{k+1}-z_k$ in (\ref{eig4}) in order to obtain
	\begin{equation}
			\RR(z_{k+1}-z_k)\geq -\langle \lambda_k\V(z_k-z_{k-1}),z_{k+1}-z_k\rangle_{\VV^*,\VV}
			- \langle D_z\II(t_{k-1},z_k),z_{k+1}-z_k\rangle_{\ZZ^*,\ZZ} \label{w1}
	\end{equation}
	Rewriting the statements (\ref{eig2}) and (\ref{eig3}) for $k=k+1$ and inserting them into each other yields
	\begin{align}
		\RR(z_{k+1}-z_k)+\lambda_{k+1}\|z_{k+1}-z_k\|_\V^2
		=\langle -D_z\II(t_k,z_{k+1}),z_{k+1}-z_k\rangle_{\ZZ^*,\ZZ}. \label{w2}
	\end{align}
	Combining (\ref{w1}) and (\ref{w2}) leads to
	\begin{align*}
		0\geq&~ \lambda_{k+1}\|z_{k+1}-z_k\|_\V^2-\langle \lambda_k\V(z_k-z_{k-1}),z_{k+1}-z_k\rangle_{\VV^*,\VV}\\
		&+\langle D_z\II(t_k,z_{k+1})-D_z\II(t_{k-1},z_k),z_{k+1}-z_k\rangle_{\ZZ^*,\ZZ}.
	\end{align*}
	By inserting the definition of $\II$ from \eqref{eq:energydef} and using the coercivity of $A$ by assumption, 
	we obtain the inequality 
	\begin{equation} \label{w3}
		\begin{aligned}
            & \lambda_{k+1}\|z_{k+1}-z_k\|_\V^2-\lambda_k\|z_k-z_{k-1}\|_\V\|z_{k+1}-z_k\|_\V
            +\alpha \|z_{k+1}-z_k\|_\ZZ^2\\		
            & \qquad\qquad\leq 		
            \langle D_z\FF(z_k)-D_z\FF(z_{k+1}),z_{k+1}-z_k\rangle_{\VV^*,\VV} \\
            & \qquad\qquad\qquad\qquad + \langle D_zf(t_{k-1},z_k)-D_zf(t_k,z_{k+1}),z_{k+1}-z_k\rangle_{\VV^*,\VV}
		\end{aligned}
	\end{equation}
    Now we apply \eqref{estF} and \eqref{estf}, both with $\epsilon=\frac{\alpha}{4}$, to estimate the right hand side.
	Note at this point that, due to \eqref{zbeschr}, the iterates $(z_i)_{i\in\N}$ are bounded in $\ZZ$ 
	so that the requirements for (\ref{estF}) are satisfied. 
    This leads to the estimate
	\begin{equation}
		\begin{aligned}
			&\lambda_{k+1}\|z_{k+1}-z_k\|_\V^2-\lambda_k\|z_k-z_{k-1}\|_\V\|z_{k+1}-z_k\|_\V
			+\frac{\alpha}{2} \|z_{k+1}-z_k\|_\ZZ^2\\
			& \qquad \leq (C_{\alpha/4} + c_{\alpha/4})\RR(z_{k+1}-z_k)\|z_{k+1}-z_k\|_\VV
			+\nu\,(t_k-t_{k-1})\|z_{k+1}-z_k\|_\VV. \label{q1}
		\end{aligned}
	\end{equation}
    By dividing this by $\|z_{k+1}-z_k\|_\V$ and using the continuous embedding $\ZZ \embed \VV$, we end up with 
	\begin{equation*}
		\lambda_{k+1}\|z_{k+1}-z_k\|_\V-\lambda_k\|z_k-z_{k-1}\|_\V+c\|z_{k+1}-z_k\|_\ZZ
		\leq C(\RR(z_{k+1}-z_k)+t_k-t_{k-1}).
	\end{equation*}
    Now let us sum up this inequality with respect to $k$ in order to conclude
	\begin{equation}
		\lambda_{k+1}\|z_{k+1}-z_k\|_\V+c\sum_{i=1}^k\|z_{i+1}-z_i\|_\ZZ\leq \lambda_1\|z_1-z_0\|_\V+C\Big(t_k+\sum_{i=1}^k\RR(z_{i+1}-z_i)\Big). \label{w6}
	\end{equation}
	Next for estimating the term $\lambda_1\|z_1-z_0\|_\V$ we consider (\ref{w2}) for $k=0$, which reads
	\begin{equation*}
		\RR(z_{1}-z_0)+\lambda_{1}\|z_{1}-z_0\|_\V^2=\langle -D_z\II(0,z_{1}),z_{1}-z_0\rangle_{\ZZ^*,\ZZ},
	\end{equation*}
    which, in view of the non-negativity of $\RR$, in turn yields
	\begin{equation*}
		\langle D_z\II(0,z_1)-D_z\II(0,z_0),z_1-z_0\rangle_{\ZZ^*,\ZZ}+\lambda_{1}\|z_{1}-z_0\|_\V^2 
		\leq \langle -D_z\II(0,z_0),z_1-z_0\rangle_{\ZZ^*,\ZZ}.
	\end{equation*}
	For the first part on the left-hand side, we argue in the same way as above, i.e., 
	we split the derivative of the energy into its individuals parts, rearrange terms and use the coercivity of $A$, 
	as well as \eqref{estF} and \eqref{estf} with $\epsilon = \alpha/4$ as above in order to obtain
	\begin{equation}
        \frac{\alpha}{2}\|z_1-z_0\|_\ZZ^2+\lambda_1\|z_1-z_0\|_\V^2
		\leq C(\RR(z_1-z_0)+\|D_z\II(0,z_0)\|_{\VV^*})\|z_1-z_0\|_\VV. \label{q3}
	\end{equation}
	By dividing by $\|z_1-z_0\|_\V$ and exploiting the continuous embedding $\VV\hookrightarrow\ZZ$, 
	it follows that
	\begin{equation}\label{eq:lambda1}
		c\|z_1-z_0\|_\ZZ+\lambda_1\|z_1-z_0\|_\V\leq  C(\RR(z_1-z_0)+\|D_z\II(0,z_0)\|_{\VV^*}).
	\end{equation}
	Now inserting this into \eqref{w6} yields
	\begin{equation*}
		\lambda_{k+1}\|z_{k+1}-z_k\|_\V+c\sum_{i=0}^k\|z_{i+1}-z_i\|_\ZZ
		\leq C\Big(t_k+\sum_{i=0}^k\RR(z_{i+1}-z_i)+\|D_z\II(0,z_0)\|_{\VV^*}\Big)	.
	\end{equation*}
	From this, in combination with Lemma~\ref{irbeschr}, the lower bound (\ref{ibesch}), 
	the termination criterion of the algorithm, and our standing assumptions on the initial state, we deduce the claim.
	Note that the case $k=0$ is covered by \eqref{eq:lambda1} by the same arguments.
\end{proof}	
	
\begin{lemma}\label{lem:taumax}
    There exists a constant $\tau_{\max} > 0$, independent of the chosen tolerance, 
    such that $\max_{k\in \N} \tau_k \leq \tau_{\max}$.
\end{lemma}
	
\begin{proof}
    According to the termination criterion and the time update in Step~\ref{it:update} of the algorithm, all but the last 
    iterations fulfill
    \begin{equation}\label{eq:taumaxest}
        T\geq t_k = t_0+\sum_{i=1}^{k}\tau_i-\sum_{i=1}^{k}\|z_i-z_{i-1}\|_\V ,
    \end{equation}
    which, together with \eqref{sumz}, yields the assertion for all $k$ except for the last iteration (if there is one at all, 
    which we will see in fact in Proposition~\ref{endtime} below). 
    Since the last step size can just be twice as large as the second to last one by Step~\ref{it:10},
    the bound carries over to all step sizes.
\end{proof}

\begin{lemma}\label{lem:distbound}
    There is a constant $C>0$ such that, for all $k\in \N$, there holds
    \begin{align}
    	\dist_{\VV^*}\{-D_z\II(t_{k-1},z_{k}),\partial\RR(0)\}&\leq C, \label{distbeschr}\\
		\dist_{\VV^*}\{-D_z\II(\hat{t}(s),\hat{z}(s)),\partial\RR(0)\}&\leq C \quad \forall s\in[s_{k-1},s_{k}).\label{distbeschr2}
    \end{align}    
\end{lemma}

\begin{proof}
	For the first claim we make use of \eqref{eig1} and \eqref{eig2}, which gives
	\begin{equation*}
		\dist_{\VV^*}\{-D_z\II(t_{k-1},z_{k}),\partial\RR(0)\}=\lambda_{k}\|z_{k}-z_{k-1}\|_\V
	\end{equation*}
	such that \eqref{distbeschr} follows from \eqref{sumz}. 
	
	 To prove \eqref{distbeschr2}, we first note that the structure of $\II$ and the assumption in \eqref{f3} imply
    \begin{equation}\label{eq:distkk}
	\begin{aligned}
		&\dist_{\VV^*}\{-D_z\II(t_{k},z_{k}),\partial\RR(0)\}\\
		&\qquad =\dist_{\VV^*}\{-D_z\II(t_{k-1},z_{k})+D_zf(t_k,z_k)-D_zf(t_{k-1},z_k),\partial\RR(0)\}\\
		&\qquad \leq \dist_{\VV^*}\{-D_z\II(t_{k-1},z_{k}),\partial\RR(0)\}+\|D_zf(t_k,z_k)-D_zf(t_{k-1},z_k)\|_{\V^{-1}}\\
		&\qquad \leq \dist_{\VV^*}\{-D_z\II(t_{k-1},z_{k}),\partial\RR(0)\}+ C\tau_k.
	\end{aligned}    
    \end{equation}
	Hence, thanks to Lemma~\ref{lem:taumax}, 
	$\dist_{\VV^*}\{-D_z\II(t_{k},z_{k}),\partial\RR(0)\}$ is bounded for all $k\geq1$. 
    Due to $D_z\II(t_0,z_0) = D_z\II(0,z_0) \in \VV^*$ by our standing assumption, 
    the claim also holds for $k=0$.
    Next we rewrite the affine interpolants as 
    \begin{equation*}
        \hat{z}(s)=\lambda(s) z_{k+1}+(1-\lambda(s))z_k 
        \quad \text{and}  \quad \hat{t}(s)=\lambda(s) t_{k+1}+(1-\lambda(s))t_k,    
    \end{equation*}        
    where $\lambda$ is given by $\lambda(s)=\frac{s-s_k}{\tau_{k+1}}\in [0,1)$ for $s\in[s_k,s_{k+1})$.
    The convexity of the distance then gives
    \begin{equation*}
	\begin{aligned}
		&\dist_{\VV^*}\{-D_z\II(\hat{t}(s),\hat{z}(s)),\partial\RR(0)\} \\
		& \leq\lambda(s)\dist_{\VV^*}\{-D_z\II(t_{k+1},z_{k+1}),\partial\RR(0)\} 
		+ (1-\lambda(s))\dist_{\VV^*}\{-D_z\II(t_{k},z_{k}),\partial\RR(0)\}\\
		&\quad +\lambda(s) \|D_z\FF(z_{k+1})-D_z\FF(\hat{z}(s))\|_{\V^{-1}}
		+(1-\lambda(s))\|D_z\FF(z_k)-D_z\FF(\hat{z}(s))\|_{\V^{-1}}\\
		&\quad +\lambda(s) \|D_zf(t_{k+1},z_{k+1})-D_zf(\hat{t}(s),\hat{z}(s))\|_{\V^{-1}}
		+(1-\lambda(s))\|D_zf(t_k,z_k)-D_zf(\hat{t}(s),\hat{z}(s))\|_{\V^{-1}}.
	\end{aligned}    
    \end{equation*}
    While the first two addends on the right hand side are bounded by \eqref{eq:distkk}, 
	we use \eqref{F2} and \eqref{zbeschr} to estimate the third one by
	\begin{align*}
		\|D_z\FF(z_{k+1})-D_z\FF(\hat{z}(s))\|_{\V^{-1}} 
		&= \Big\|\int_0^1 D_z^2\FF(\hat{z}(s)+r(z_{k+1}-\hat{z}(s)))[z_{k+1}-\hat{z}(s)]dr\Big\|_{\V^{-1}}\\
		&\leq C \int_0^1 (1+\|\hat{z}(s)+r(z_{k+1}-\hat{z}(s))\|_\ZZ^q)\|z_{k+1}-\hat{z}(s)\|_\ZZ dr
		\leq C.
	\end{align*}
	Similarly, we deduce from \eqref{f3}, \eqref{zbeschr}, and $t_{k+1} - t_k \leq \tau_k \leq C$ for all $k\in \N$ that 
	\begin{align*}
		\|D_zf(t_{k+1},z_{k+1})-D_zf(\hat{t}(s),\hat{z}(s))\|_{\V^{-1}} 
		\leq C(|t_{k+1}-\hat{t}(s)|+\|z_{k+1}-\hat{z}(s)\|_\ZZ)
		\leq C.
	\end{align*}
	For the terms including $t_k$ and $z_k$ one can argue in the same way. 
	This finally gives the claim.
\end{proof}

We proceed with the complementarity relation in \eqref{eq:komp1}. 
As an immediate consequence of the time update in Step~\ref{it:update}, Remark~\ref{rem:time}, 
and the construction of the affine interpolants in Step~\ref{it:4} of the algorithm, we deduce the following

\begin{lemma}\label{lem:complaffine}
    The affine interpolants satisfy 
	\begin{equation}
		\hat{t}'(s)\geq 0, \quad \hat{t}'(s)+\|\hat{z}'(s)\|_\V=1  \label{beschr}
	\end{equation}
    for all $s \neq s_k$, $k\in \N$.
\end{lemma}

The next result essentially allows us to prove that the end time is reached 
and the algorithm will stop after finitely many iterations, see Proposition~\ref{endtime} below. 
Since the associated  proof is essentially analogous to
the one of \cite[Proposition~2.4]{knees}, we postpone it to Appendix~\ref{sec:I1I2}.

\begin{lemma}\label{lem:I1I2}
    Define the function $r: \R_+\setminus\{s_k\}_{k\in \N} \to \R$ by 
    \begin{equation}\label{eq:rdef}
        r(s) := \langle D_z\II(\hat{t}(s),\hat{z}(s))-D_z\II(\underline{t}(s),\bar{z}(s)),
        \hat{z}'(s)\rangle_{\ZZ^*,\ZZ} .
    \end{equation}        
    Then, for all $k\in \N$, there holds
    \begin{equation}
        r(s) \leq \bar c\,\tau_k \quad \forall  s\in(s_{k-1},s_k)\label{rabsch}
    \end{equation}
    and
	\begin{align}
        I_1^k &= \int_{s_{k-1}}^{s_k}\hat{t}'(s)\,\dist_{\VV^*}\{-D_z\II(\hat{t}(s),\hat{z}(s)),\partial\RR(0)\}ds 
		\leq \bar c\, \tau_k, \label{e1}\\
        I_2^k &= \int_{s_{k-1}}^{s_k}r(s)+ \|\hat{z}'(s)\|_\V
        \begin{aligned}[t]
    		\Big( & \dist_{\VV^*}\{-D_z\II(\hat{t}(s),\hat{z}(s)),\partial\RR(0)\} \\[-1ex]
	    	& -\dist_{\VV^*}\{-D_z\II(\underline{t}(s),\bar{z}(s)),\partial\RR(0)\}\Big) ds \leq \bar c\, \tau_k, 
        \end{aligned}\label{e2}
	\end{align}
	with a constant $\bar c >0$, independent of $\tau_k$ and the chosen tolerance.
\end{lemma}

Lemma~\ref{lem:complaffine} already provides a bound for $\hat z'$, but only in $\VV$.
The next lemma addresses a bound w.r.t.\ the $\ZZ$-norm.
Its proof combines arguments from \cite[Proposition 3.7 and Lemma~3.9]{siev21}
and is given in Appendix~\ref{sec:tzbeschr} for convenience of the reader.

\begin{lemma}\label{tzbeschr}
	There exists a constant $C>0$, independent of the tolerance and the choice of the step size, such that, 
	for all $k\in \N$, there holds $\|\hat{z}\|_{H^{1}(0,s_k;\ZZ)}\leq C$.
\end{lemma}

The next auxiliary result addresses the complementarity relation and the energy balance in 
\eqref{komp} and \eqref{energie}.
Its proof is only based on the stationarity conditions from Lemma~\ref{lem:stat} 
in combination with a chain rule for the energy.
It is therefore completely analogous to the case with uniform step size and hence omitted.

\begin{lemma}[{\cite[Proposition~2.4]{knees}, \cite[Lemma~3.8]{FElocmin}}] \label{fehler}
    Let $k\in \N$ be arbitrary.  For every choice of the step sizes,     
    the affine and constant interpolants generated by Algorithm~\ref{alg:adaptlocmin} fulfill
	\begin{align}
		\hat{t}'(s)\, \dist_{\VV^*}\{-D_z\II(\underline{t}(s),\bar{z}(s)),\partial\RR(0)\}=0 
		\quad \text{a.e.\ in } (0, s_k).\label{kompl}
	\end{align}
    and, for all $0\leq a \leq b \leq s_k$, the following discrete energy identity holds true
	\begin{equation}
		\begin{aligned}
			\II(\hat{t}(b),\hat{z}(b))+\int_{a}^{b}\RR(\hat{z}'(\sigma))
			+ \dist_{\VV^*}\{-D_z\II(\underline{t}(\sigma),\bar{z}(\sigma)),\partial\RR(0)\}d\sigma  & \\
			-\II(\hat{t}(a),\hat{z}(a))-\int_{a}^{b} 
			\partial_t\II(\hat{t}(\sigma),\hat{z}(\sigma))\,\hat{t}'(\sigma)d\sigma 
			&= \int_{a}^{b} r(\sigma)d\sigma, \label{diskenergy}
		\end{aligned}
	\end{equation}
	where $r$ is as defined in \eqref{eq:rdef}.
\end{lemma}

\begin{lemma}\label{lem:I2pos}
    For all $k\in \N$, there holds $I^k_{2} \geq 0$.
\end{lemma}

\begin{proof}
    Using $\hat t' + \|\hat z'\|_{\V} = 1$ a.e.\ by Lemma~\ref{lem:complaffine} and the discrete energy identity from 
    \eqref{diskenergy}, we find
    \begin{equation*}
		\begin{aligned}
            I^k_{2}
            &= \int_{s_{k-1}}^{s_k}r(s)+ \|\hat{z}'(s)\|_\V
            \begin{aligned}[t]
    		    \Big( & \dist_{\VV^*}\{-D_z\II(\hat{t}(s),\hat{z}(s)),\partial\RR(0)\} \\[-1ex]
	    	    & -\dist_{\VV^*}\{-D_z\II(\underline{t}(s),\bar{z}(s)),\partial\RR(0)\}\Big) ds 
            \end{aligned}\\
		    &= \II(\hat{t}(s_k),\hat{z}(s_k))+\int_{s_{k-1}}^{s_k}\RR(\hat{z}'(s))
			+ \|\hat{z}'(s)\|_\V\, \dist_{\VV^*}\{-D_z\II(\hat{t}(s),\hat{z}(s)),\partial\RR(0)\} ds   \\
			&\quad -\II(\hat{t}(s_{k-1}),\hat{z}(s_{k-1}))-\int_{s_{k-1}}^{s_k} 
			\partial_t\II(\hat{t}(s),\hat{z}(s))\,\hat{t}'(s)d s \\
			&\geq 0,
		\end{aligned}
	\end{equation*}
	where the last inequality follows from Remark~\ref{rem:energyid}.
\end{proof}

Lemma~\ref{fehler} demonstrates 
that the complementarity relation \eqref{komp} and the energy identity \eqref{energie} are 
not fulfilled by the affine interpolants (as expected). The mismatch between \eqref{komp} and its discrete counterpart 
\eqref{kompl} arises, since the constant instead of the affine interpolant appears in the distance in \eqref{kompl}.
This gives rise to our first residuum in \eqref{i1}, which measures the defect, if we insert $(\hat t, \hat z)$ in 
the complementarity condition instead of $(\underline{t},\bar{z})$. 
Similarly, the ``wrong'' interpolant appears in the distance in \eqref{diskenergy}, too. 
Note in this context that the missing  $\V$-norm of $\hat z'$ in front of the distance in comparison to 
\eqref{energie} can be inserted here, since, by \eqref{beschr} and \eqref{kompl}, 
\begin{equation*}
    \dist_{\VV^*}\{-D_z\II(\underline{t}(s),\bar{z}(s)),\partial\RR(0)\}
    = \|\hat z'(s)\|_{\V} \,\dist_{\VV^*}\{-D_z\II(\underline{t}(s),\bar{z}(s)),\partial\RR(0)\}
    \quad \text{a.e.\ in } (0,s_k).
\end{equation*}
However, the discrete energy identity in form of \eqref{diskenergy} without this norm is better suited 
for the convergence analysis in the upcoming sections.
Moreover, an additional term involving $r$ emerges on the right hand side of the discrete energy identity. 
Comparing \eqref{energie} with \eqref{diskenergy} thus results in the second residuum in \eqref{i2}, 
cf.\ also \eqref{e2}.

\begin{remark}\label{rem:tau_egal}
    As indicated at the beginning of this section, the aforementioned results are completely 
    independent of the choice of the step size, since their proofs only employ the stationarity conditions 
    from Lemma~\ref{lem:stat}, the decay condition in \eqref{alg2}, the time update formula from Step~\ref{it:update},
    and our standing assumptions on energy and dissipation.
    This observation will be of major importance in Section~\ref{sec:convinf}.
\end{remark}

As it will turn out, the trouble maker in the convergence analysis is not the mismatch induced by
$r$, since we already have a fairly rigorous estimate thereof in form of \eqref{rabsch}.
The delicate issue is to pass to the limit with the distance terms. If however the involved spaces 
and their respective norms happen to be equivalent, the analysis simplifies significantly and one obtains 
much sharper estimates. Moreover, an additional assumption on the algorithm 
leading to a modification of the update of the step size becomes 
superfluous in this case (see Assumption~\ref{assu:nested} below).
For this reason, we will treat this case separately in the next section.


\section{Convergence Analysis in $\R^n$}\label{sec:convRn}

As indicated above, we restrict ourselves to the case, where $\ZZ = \VV$.
We denote this space by $\VV$ and its norm by $\| \cdot \|$ in the following.
A natural example for this setting is of course the finite dimensional case, i.e., $\VV = \R^n$, 
but infinite dimensional settings are conceivable as well. Consider for instance a measure space $(\Omega, \AA, \mu)$ 
with a finite measure $\mu$
and take $\VV = L^2(\mu)$ and $\XX = L^1(\mu)$. 
Then, $\RR$ and $f$ can be chosen as in Example~\ref{ex:bspsystem} and 
$A$ is set to be a weighted $L^2(\mu)$-norm squared. If $\FF$ is chosen as 
$\FF(z) := \int_\Omega g(T z) d \mu$, where $T \in L(L^2(\mu), L^p(\mu))$, $p>2$, is a compact operator 
(e.g., a convolution with a smoothing kernel or another integral operator) and 
$g: L^p(\mu) \to L^2(\mu)$ is a suitable smooth Nemyzkii operator, then the assumptions 
in \eqref{F1}--\eqref{eq:DFFweakcont} can be fulfilled so that the example fits into our general setting. 
Nevertheless, the most relevant example is probably $\VV = \R^n$, which also motivates the title of this section.

So far, we do not know whether Algorithm~\ref{alg:adaptlocmin} is well posed in the sense that the 
termination criterion in Step~\ref{it:abbruch} is fulfilled for $k$ sufficiently large and the end time is reached.
This will be shown next under the additional assumption that 
\begin{equation}\label{eq:initialtau}
     \frac{\tol}{2\,\bar c} \leq \tau_1.
\end{equation}
Of course, since we drive $\tol$ to zero anyway, this assumption is not restrictive at all. 
Therefore, we tacitly take it for granted for the rest of this section without mentioning it every time.

\begin{proposition}\label{endtime}
	Algorithm~\ref{alg:adaptlocmin} terminates after a finite number of iterations, i.e., there exists $N_{\tol}\in\N$, 
	depending on the chosen tolerance, so that $t_{N_{\tol}}\geq T$. 
\end{proposition}

\begin{proof}
	Due to Steps~\ref{it:refine} and \ref{it:reduc} and \eqref{e1}, \eqref{e2} 
	and thanks to the assumption in \eqref{eq:initialtau}, 	it holds for each $k\in\N$ that
	\begin{align}
		\tau_k \geq \frac{\tol}{2\,\bar c} =:\tau_{\min}, \label{taumin}
	\end{align} 
    where $\bar c > 0$ is the constant from Lemma~\ref{lem:I1I2}, which is independent of the tolerance.
	On account of the time update in Step~\ref{it:update} and estimate \eqref{sumz}, we thus obtain 
	\begin{equation}
			t_k=t_0+\sum_{i=1}^{k}\tau_i-\sum_{i=1}^{k}\|z_i-z_{i-1}\|_\V 
			\geq k\, \tau_{\min}-\sum_{i=1}^{k}\|z_i-z_{i-1}\|_\V 
			\geq k\, \tau_{\min}-C ~\rightarrow \infty, \label{T}
	\end{equation}
    as $k\to \infty$.
\end{proof}

The above result allows us to introduce the maximal value for the curve parameter by setting 
\begin{equation}\label{eq:defS}
    S := \min\{ s \in [0, \infty) : \hat t(s) = T\}.
\end{equation}

\begin{remark}\label{rem:endtime}
    It is to be noted that the proof of Proposition~\ref{endtime} does not require that $\ZZ$ and $\VV$ coincide. 
    The assertion of Proposition~\ref{endtime} thus also holds for $\ZZ\neq \VV$.
\end{remark}

In the case that all spaces are equal, we can substantially sharpen the estimates from Lemma~\ref{lem:I1I2}, 
as the next lemma shows.

\begin{lemma}\label{lem:Itau2}
    For all $k = 1, ..., N_{\tol}$, there holds
    \begin{equation*}
        I_1^k \leq C\, \tau_k^2 \quad \text{and} \quad 
        I_2^k \leq C\, \tau_k^2
    \end{equation*}
    with a constant $C>0$ independent of $\tau_k$ and the chosen tolerance.
\end{lemma}

\begin{proof}
    First note that, thanks to Remark~\ref{rem:time}, the construction of the affine and constant interpolants, 
    and the time update in Step~\ref{it:update}, 
	\begin{equation*}
		\|\hat{z}(s)-\bar{z}(s)\|\leq C\, \|z_k-z_{k-1}\|_\V \leq C\,\tau_k 
		\quad \text{and} \quad |\hat{t}(s)-\underline{t}(s)|\leq |t_k-t_{k-1}|\leq \tau_k,
	\end{equation*}
	holds true for all $s\in [s_{k-1}, s_k)$, whereby it follows
    \begin{equation}\label{eq:Vstarest}
	\begin{aligned}
 		& \|D_z\II(\hat{t}(s),\hat{z}(s))-D_z\II(\underline{t}(s),\bar{z}(s))\|_{\VV^*} \\
		&\qquad \leq \|A(\hat{z}(s)-\bar{z}(s))\|_{\VV^*} 
       \begin{aligned}[t]
		& +\|D_z\FF(\hat{z}(s))-D_z\FF(\bar{z}(s))\|_{\VV^*}\\
		& +\|D_zf(\hat{t}(s),\hat{z}(s))-D_zf(\underline{t}(s),\bar{z}(s))\|_{\VV^*}
		\leq C\, \tau_k,       
        \end{aligned}       		
	\end{aligned}    
    \end{equation}
	where we used \eqref{f3} and \eqref{F2}, the latter in combination with the boundedness of the 
	iterates by Corollary~\ref{cor:zkbound}.
   	If $\ZZ = \VV$, then $\dist_{\VV^*}\{\eta, \partial\RR(0)\}$ equals the usual distance of $\eta$ to the 
   	bounded set $\partial\RR(0) \subset \VV^*$ and is always finite. We denote this distance therefore just by 
   	$\abst$. Exploiting the Lipschitz-continuity of the distance then results in
    \begin{equation}\label{eq:distlip}
    \begin{aligned}
		|\abst\{-D_z\II(\hat{t}(s),\hat{z}(s)),\partial\RR(0)\}-\abst\{-D_z\II(\underline{t}(s),\bar{z}(s)),\partial\RR(0)\}| &\\
		 \leq L\, \|D_z\II(\hat{t}(s),\hat{z}(s))-D_z\II(\underline{t}(s),\bar{z}(s))\|_{\VV^*} 
		& \leq C\, \tau_k.  
    \end{aligned}
    \end{equation}
	With this estimate and \eqref{i1}, \eqref{beschr}, and \eqref{kompl} at hand, we derive
	\begin{align*}
		I_1^k   
		&=\int_{s_{k-1}}^{s_k}\hat{t}'(s)\big( \abst\{-D_z\II(\hat{t}(s),\hat{z}(s)),\partial\RR(0)\}
		- \abst\{-D_z\II(\underline{t}(s),\bar{z}(s)),\partial\RR(0)\}\big)ds\\
		&\leq \int_{s_{k-1}}^{s_k}\big|\abst\{-D_z\II(\hat{t}(s),\hat{z}(s)),\partial\RR(0)\}
		- \abst\{-D_z\II(\underline{t}(s),\bar{z}(s)),\partial\RR(0)\}\big|ds
 		\leq C\,\tau_k^2.
	\end{align*}
    Concerning $I_2^k$, we use \eqref{rabsch} in combination with \eqref{eq:distlip} and \eqref{kompl} to obtain
	\begin{align*}
		I_2^k &= \int_{s_{k-1}}^{s_k} r(s) 
		+ \|\hat{z}'(s)\|_\V\big(\abst\{-D_z\II(\hat{t}(s),\hat{z}(s)),\partial\RR(0)\}
		- \abst\{-D_z\II(\underline{t}(s),\bar{z}(s)),\partial\RR(0)\}\big)ds\\
		&\leq C\, \tau_k^2,
	\end{align*}
	which completes the proof.
\end{proof}

These improved a priori estimates allow us to give a sharper estimate of $N_{\tol}$, the maximum number of 
steps needed to reach the final time $T$. Arguing as in the proof of Proposition~\ref{endtime}
with Lemma~\ref{lem:Itau2} instead of \eqref{e1} and \eqref{e2}, we find that, 
for each iteration, 
\begin{align}
	\tau_k \geq c\, \sqrt{\tol} =:\tau_{min}. \label{taumin2}
\end{align} 
with a constant $c>0$ independent of the chosen tolerance. 
Analogously to \eqref{T}, we then find for the maximum number of iterations
\begin{equation}
    N_{\tol}\leq \bigg\lceil\frac{T+C}{\tau_{min}}\bigg\rceil=\bigg\lceil\frac{T+C}{c\sqrt{\tol}}\bigg\rceil , \label{Nbeschr}
\end{equation}
where $C>0$ is the constant from Lemma~\ref{lem:sumz}.
With this estimate at hand, we are now in the position to prove the weak* convergence of a subsequence 
of affine interpolants generated by Algorithm~\ref{alg:adaptlocmin}
to a  $\V$-parametrized BV solution, provided the tolerance tends to zero. 
For this purpose, let $\{\tol_n\}_{n\in\N}$ be an arbitrary positive sequence converging monotonously to zero. 
For each $n\in\N$, we denote the iterates of the algorithm from Step~\ref{it:stat} and \ref{it:update} 
by $(t_k^n,z_k^n)$ with associated step size $\tau_k^n$, $k=0,\dots, N_{\tol_n}$, 
and the corresponding affine and constant interpolants by $\hat{z}_n$, $\hat{t}_n$, $\bar{z}_n$, and $\underline{t}_n$. 
Moreover, we abbreviate $N_n := N_{\tol_n}$.
Finally, the maximum value for the curve parameter from \eqref{eq:defS} associated with $\tol_n$
is denoted by $S_n$ and we define 
\begin{equation}\label{eq:deftildeS}
    \tilde{S}:=\sup_{n\in\N} S_n.
\end{equation}
Due to \eqref{sumz}, we obtain
\begin{equation}\label{eq:Snbound}
	S_n \leq s^n_{N_{n}} =\sum_{i=1}^{N_{n}}\tau_i^n
	= \sum_{i=1}^{N_{n}}t_i^n-t_{i-1}^n+\|z_i^n-z_{i-1}^n\|_\V
	= t_{N_{n}}^n+ \sum_{i=1}^{N_{n}}\|z_i^n-z_{i-1}^n\|_\V \leq C,
\end{equation}
which implies the boundedness of the sequence $\{S_n\}_{n\in\N}$ and therefore $\tilde{S}$ is finite. 
Finally, we constantly extend $\hat{z}_n$, $\hat{t}_n$, $\bar{z}_n$, and $\underline{t}_n$ from 
$[0,s^n_{N_n}]$ to the interval $[0,\tilde{S}]$ and denote these constant continuations by the same symbolds.

\begin{theorem}\label{thm:convRn}
	Assume that  $\ZZ = \VV$. Then every sequence $\{\tol_n\}_{n\in\N}\subset \R^+$
	converging to zero admits a subsequence, denoted by the same symbol to ease notation, such that 
	\begin{equation}\label{eq:convRn}
        S_n\to S \mbox{~in~}\R, \quad 
        \hat{t}_n\rightharpoonup^* \hat{t} \mbox{~in~} W^{1,\infty}(0,S;\R), \quad
        \hat{z}_n\rightharpoonup^* \hat{z} \mbox{~in~} W^{1,\infty}(0,S;\VV),
	\end{equation}
	and every weak limit in the above sense is a $\V$-parametrized BV solution of the rate-independent system.
\end{theorem}

\begin{proof}
    To some extend, the proof follows the lines of \cite[Theorem~2.5]{knees} and \cite[Theorem~3.9]{FElocmin}, 
    but, since we have to modify the arguments at some distinct places, we present the proof in detail.

    First, due to \eqref{eq:Snbound} and \eqref{beschr}, which hold for every $n\in \N$, we observe that $\{S_n\}$, 
    $\{\hat t_n\}$, and $\{\hat z_n\}$ are bounded in $\R$, $W^{1,\infty}(0,\tilde S)$, and $W^{1,\infty}(0,\tilde S;\VV)$, 
    respectively, and the existence of a weakly* converging subsequence as in \eqref{eq:convRn} follows immediately.
    Note that $S\leq \tilde S$ by definition of $\tilde S$.
    
    \vspace*{1ex}\noindent
    \textsl{(i) Initial and end time condition}
    
    To show that every weak* accumulation point is a $\V$-parametrized BV solution,  
    assume that $\{S_n, \hat t_n, \hat z_n\}_{n\in \N}$ is a sequence generated by Algorithm~\ref{alg:adaptlocmin} 
    associated with $\tol_n$ such that \eqref{eq:convRn} holds.
    Then, by compact resp.\ continuous embedding, we know that 
    \begin{equation}\label{eq:weakptwise}
        \hat t_n \to \hat t\;\text{ in } C([0,\tilde S]), 
        \quad \hat z_n(s) \weak \hat z(s) \; \text{ in }\VV \;\;\forall\, s\in [0,\tilde S]
    \end{equation}
    and consequently,
	\begin{equation}
		z_0=\hat{z}_n(0)\rightharpoonup \hat{z}(0) \; \text{ in }\VV, \quad
		0=\hat{t}_n(0)\to\hat{t}(0), \quad T=\hat{t}_n(S_n)\to\hat{t}(S) \quad \text{ for } n\to \infty,
	\end{equation}
	so that the initial and end time condition in \eqref{start} are fulfilled. 
    
    \vspace*{1ex}\noindent
    \textsl{(ii) Complementarity relation}

    First, thanks to Lemma~\ref{lem:complaffine}, for every $n\in \N$ and almost every $s\in (0,\tilde S)$, 
    there holds $0 \leq \hat t_n'(s)$ and $\hat t_n'(s) + \|\hat z_n'(s)\|_\V \leq 1$. 
    Note that this also holds f.a.a.\ $s \geq s^n_{N_n}$ because of the constant continuation we have chosen. 
    To prove that this complementarity relation transfers to the weak limit, let us define the set
	\begin{equation}\label{eq:defM}
		M:=\{(\tau,\xi)\in L^2(0,S;\R)\times L^2(0,S;\VV):\tau(s)\geq 0, ~\tau(s)+\|\xi(s)\|_\V\leq1
		\text{ a.e.\ in } (0,\tilde S)\}.	
	\end{equation}
	Clearly $M$ is convex and closed and therefore weakly closed such that  $(\hat{t}_n',\hat{z}_n')\in M$ 
	for all $n\in\N$ yields $(\hat{t}',\hat{z}')\in M$, i.e., \eqref{eq:komp1} for the weak limit.

    To shorten the notation, we use the following abbreviations in the rest of the proof:
    \begin{equation*}
        \abst(s) := \abst\{-D_z\II(\hat{t}(s),\hat{z}(s)),\partial\RR(0)\}, \;
        \abst_n(s) := \abst\{-D_z\II(\hat{t}_n(s),\hat{z}_n(s)),\partial\RR(0)\}, \; s\in [0,S].
    \end{equation*}
    
    To prove the second complementarity relation in \eqref{komp}, we first note that
    the weak* convergence of $\hat t_n'$ yields
	\begin{align*}
		0&\leq \int_0^S\hat{t}'(\sigma) \,\abst(\sigma)\,d\sigma \\
 		&= \lim_{n\to\infty} \int_0^S\hat{t}_n'(\sigma)\,\abst(\sigma)\,d\sigma \\
		&\leq \limsup_{n\to\infty}\int_0^S\hat{t}_n'(\sigma)
		\big(\abst(\sigma) - \abst_n(\sigma)\big)d\sigma 
		+\limsup_{n\to\infty}\int_0^S\hat{t}_n'(\sigma)\,\abst_n(\sigma)\, d\sigma.
	\end{align*}
    Now, thanks to the pointwise (weak) convergence in \eqref{eq:weakptwise}, 
    the weak continuity of $D_z\II$ by \eqref{ikonv}, the weak lower semicontinuity of the distance implies
    \begin{equation}\label{eq:DzIwlsc}
    \begin{aligned}
        \liminf_{n\to\infty} \abst_n(s) 
        & = \liminf_{n\to\infty} \abst\{-D_z\II(\hat{t}_n(s),\hat{z}_n(s)),\partial\RR(0)\}\\
        &\geq \abst\{-D_z\II(\hat{t}(s),\hat{z}(s)),\partial\RR(0)\} = \abst(s)  
        \quad \forall\, s\in [0,S],
    \end{aligned}
    \end{equation}
    which, due to $0 \leq \hat t'_n(s) \leq 1$ a.e.\ in $(0,S)$, in turn gives
    \begin{equation}
    	\limsup_{n\to\infty}\hat{t}_n'(s)
    	\big(\abst(s) 	- \abst_n(s)\big)\leq 0
    \end{equation}
    f.a.a.\ $s \in (0,S)$. Thus Fatou's lemma implies
	\begin{align*}
		\limsup_{n\to\infty}\int_0^S\hat{t}_n'(\sigma)
		\big(\abst(\sigma) - \abst_n(\sigma)\big)d\sigma 
		\leq \int_0^S\limsup_{n\to\infty}\big(\hat{t}_n'\big(\sigma)
		\big(\abst(\sigma) - \abst_n(\sigma)\big)d\sigma \big)d\sigma 
		\leq 0.
	\end{align*}
	For the second integral we obtain by means of our refinement criterion in Step~\ref{it:refine} that
	\begin{align*}
		\limsup_{n\to\infty}\int_0^S\hat{t}_n'(\sigma)\,\abst_n(\sigma)\,d\sigma 
		\leq \limsup_{n\to\infty}\sum_{k=1}^{N_{n}} I_{1,n}^k
		\leq  \limsup_{n\to\infty} N_{n}\tol_n 
		\leq \lim_{n\to\infty} \bigg\lceil\frac{(T+C)}{c\sqrt{\tol_{n}}}\bigg\rceil \tol_n = 0,
	\end{align*}
	where we used (\ref{Nbeschr}) for the last estimate. Note that the integrand on the left hand side vanishes 
	a.e.\ in $(s^n_{N_n}, S)$ due to constant continuation, which justifies the first inequality.	
	Hence, we have proven
	\begin{equation*}
		\int_0^S\hat{t}'(\sigma) \abst\{-D_z\II(\hat{t}(\sigma),\hat{z}(\sigma)),\partial\RR(0)\} d\sigma = 0.
	\end{equation*}
	and the non-negativity of the integrand yields that the weak limit $(\hat{t},\hat{z})$ indeed fulfills \eqref{komp}.

    \vspace*{1ex}\noindent
    \textsl{(iii) Energy identity}
    
	It remains to verify the energy identity. For this purpose, let $s\in (0,S)$ be arbitrary.
	First of all, the pointwise (weak) convergence in \eqref{eq:weakptwise} 
	together with the weak lower semicontinuity of $\II$ by \eqref{iuhs} imply
	\begin{equation}\label{eq:IIlsc}
	     \II(\hat{t}(s),\hat{z}(s)) \leq \liminf_{n\to\infty} \II(\hat{t}_n(s),\hat{z}_n(s)).
	\end{equation}
	For the energy at initial time, we have
	\begin{equation}\label{eq:IIinitial}
	    \II(\hat{t}_n(0),\hat{z}_n(0)) = \II(0, z_0) = \II(\hat{t}(0),\hat{z}(0)) \quad \forall\, n\in \N.
	\end{equation}
	Furthermore, the convexity and continuity of $\RR$ by assumption together with the weak* convergence of $\hat z'$ 
	in $L^\infty(0,S;\VV)$ yields
	\begin{equation}\label{eq:RRlsc}
		\int_0^s \RR(\hat{z}'(\sigma))d\sigma \leq \liminf_{n\to\infty} \int_0^s \RR(\hat{z}_n'(\sigma))d\sigma.
	\end{equation}
	Moreover, the pointwise (weak) convergence in \eqref{eq:weakptwise} and the assumption in \eqref{dtfkonv} yield
	\begin{equation*}
	    \partial_t f(\hat t_n(\sigma), \hat z_n(\sigma)) \to \partial_t f(\hat t(\sigma), \hat z(\sigma))
	    \quad \forall\, \sigma\in [0,S],
	\end{equation*}
	and, since this expression is bounded by \eqref{f2} and \eqref{zbeschr}, 
	Lebesgue's dominated convergence theorem implies
	that it converges strongly in $L^2(0,S)$. Together with the weak convergence of $\hat t_n'$ in $L^2(0,T)$, it follows
	\begin{equation}\label{eq:dtIIconv}
		\int_0^s\partial_t\II(\hat{t}_n(\sigma),\hat{z}_n(\sigma))\hat{t}_n'(\sigma)d\sigma 
		\to \int_0^s \partial_t\II(\hat{t}(\sigma),\hat{z}(\sigma))\hat{t}'(\sigma)d\sigma.
	\end{equation}
	In addition, as $\|\hat z'(\sigma)\|_\V \leq 1$ a.e.\ in $(0,S)$ by \eqref{eq:komp1}, 
    Fatou's lemma together with \eqref{eq:DzIwlsc} gives
	\begin{align*}
		\int_0^s\|\hat{z}'(\sigma)\|_\V \,\abst(s)\,d\sigma 
		\leq\int_0^s \abst(s)\, d\sigma 
		\leq\liminf_{n\to\infty} \int_0^s \abst_n(s)\,d\sigma.
	\end{align*}
	Since $s < S$ and $S_n\to S$, there holds $s < S_n \leq s^n_{N_n}$ for sufficiently large $n\in \N$ and thus, 
	for those $n$, the discrete energy identity from \eqref{diskenergy} holds with $a=0$ and $b=s$.
	Combining the above convergence results with the discrete energy identity leads to
	\begin{align*}
		& \II(\hat{t}(s),\hat{z}(s))-\II(\hat{t}(0),\hat{z}(0))
		-\int_{0}^{s} \partial_t\II(\hat{t}(\sigma),\hat{z}(\sigma))\hat{t}'(\sigma)d\sigma 
	    +\int_{0}^{s}\RR(\hat{z}'(\sigma)) +\|\hat{z}'(\sigma)\|_\V \,\abst(\sigma)\,d\sigma  \\
		& \leq \liminf_{n\to \infty} \II(\hat{t}_n(s),\hat{z}_n(s))
		- \lim_{n\to\infty} \Big(\II(\hat{t}_n(0),\hat{z}_n(0))
		+ \int_{0}^{s} \partial_t\II(\hat{t}_n(\sigma),\hat{z}_n(\sigma))\hat{t}_n'(\sigma)d\sigma\Big) \\
		& \qquad +\liminf_{n\to\infty}\int_0^s \RR(\hat{z}_n'(\sigma))d\sigma
		+\liminf_{n\to\infty} \int_0^s \abst_n(\sigma)\,d\sigma \\
        & \leq \limsup_{n\to\infty} \Big(
        \begin{aligned}[t]
            &\II(\hat{t}_n(s),\hat{z}_n(s)) - \II(\hat{t}_n(0),\hat{z}_n(0))
		    + \int_{0}^{s} \partial_t\II(\hat{t}_n(\sigma),\hat{z}_n(\sigma))\hat{t}_n'(\sigma)d\sigma \\
		    & + \int_0^{s} \RR(\hat{z}_n'(\sigma)) + \abst_n(\sigma)\,d\sigma \Big),
        \end{aligned} \\
        & = \limsup_{n\to\infty} 
        \int_0^s r_n(\sigma) + \abst\{-D_z\II(\hat t_n(\sigma), \hat z_n(\sigma)), \partial \RR(0)\} 
         - \abst\{-D_z\II(\underline t_n(\sigma), \bar z_n(\sigma)), \partial \RR(0)\} d \sigma\\
		&= \limsup_{n\to\infty} \sum_{k=1}^{N_{n}} \big( I^k_{1,n} + I^k_{2,n} \big)
		\leq  \limsup_{n\to\infty} 2\cdot N_{n}\tol_n 
		\leq \limsup_{n\to\infty} \bigg\lceil\frac{2(T+C)}{c\sqrt{\tol_{n}}}\bigg\rceil \tol_n = 0,
	\end{align*}
    where we used the definition of the residua and their positivity by Lemma~\ref{lem:I2pos} along with 
    \eqref{beschr} and again \eqref{Nbeschr}.
    Thus, the limit $(\hat t, \hat z)$ satisfies the energy inequality in $(0,S)$, which, by Remark~\ref{rem:energyid}, 
    is equivalent to the energy identity \eqref{energie}.
    By continuity, the energy identity holds on the whole interval $[0,S]$.
	Consequently, $(\hat{t},\hat{z})$ is indeed a $\V$-parametrized BV solution, as claimed.
\end{proof}

We point out that one looses the normalization condition in \eqref{beschr} by passing to the limit
and it might happen that the weak limit is a degenerate $\V$-parametrized BV solution. 
Nevertheless, one may cut out sets of positive measure from $(0,S)$, where 
$\hat t'(s) + \|\hat z'(s)\|_\V = 0$, without changing the curve and convert the remaining curve
into a normalized solution by means of a suitable re-parametrization, see \cite[Lemma~A.4.3]{michael}
for details.


\section{Convergence for Systems in Infinite Dimensional Spaces}\label{sec:convinf}

In this section, we turn to the case $\ZZ \neq \VV$ such that Example~\ref{ex:bspsystem} is 
covered by the convergence analysis, too.
Unfortunately, as indicated above, we need an additional assumption on the algorithm to treat this case. 

\begin{assumption}\label{assu:nested}
    Let a sequence of tolerances $\{\tol_n\}_{n\in \N}$ tending monotonically to zero be given.
    Then we assume that, for each $n\in \N$, the update of the step size in Algorithm~\ref{alg:adaptlocmin} 
    is such that the following holds true:
    \begin{enumerate}
        \item\label{it:termination} For all $n\in \N$, the end time is reached, i.e., there is an index $N_n\in \N$ such that 
        $t^n_{N_n} \geq T$.
        \item\label{it:tolest} After each iteration, the estimates of Step~\ref{it:refine} are fulfilled, i.e., 
        $I^k_{1,n} < \tol$ and $I^k_{2,n} < \tol$ for all $n\in \N$ and all $k \leq N_n$.
    \end{enumerate}
    Thanks to \eqref{it:termination}, we can again construct $S_n$ and $\tilde S$ as in \eqref{eq:defS} 
    and \eqref{eq:deftildeS}, respectively.
    We then recursively define the family $\{\bar{\tau}_n\}_{n\in \N}$ of step size functions by 
    \begin{align*}
        &\bar\tau_1: [0,\tilde{S}]\to[0,\infty), & &
        \bar\tau_1(s) := 
        \begin{cases}
            \tau_{k+1}^1, &  s\in[s_{k}^1,s_{k+1}^1),~ k=0,\dots,N_1-1 , \\
            \tau_{\max}, & s \geq s^1_{N_1},
        \end{cases}  \\
        & \bar\tau_n: [0,\tilde{S}]\to[0,\infty), \, n>1, & &
        \bar\tau_n(s) := 
        \begin{cases}
            \tau_{k+1}^n, &  s\in[s_{k}^n,s_{n+1}^1),~ k=0,\dots,N_n-1 , \\
            \displaystyle{\min_{1\leq j \leq n-1} \bar\tau_{j}(s)} , & s \geq s^n_{N_n},
        \end{cases} 
    \end{align*}  
    where $\tau_{\max}$ is the constant from Lemma~\ref{lem:taumax}.
    Given the step size functions, we additionally require that
    \begin{enumerate}
        \setcounter{enumi}{2}
        \item\label{it:nested} The sequence of step size functions is pointwise monotonically decreasing, i.e., 
        $\bar\tau_{n+1}(s) \leq \bar\tau_{n}(s)$ for all $s\in [0,\tilde S]$ and all $n\in \N$.
    \end{enumerate}
\end{assumption}

The rather intricate definition of $\bar\tau_n$ is due to the fact that all iterates ``live'' on their own interval 
$[0,s^n_{N_n}]$ and, potentially, we need to extend them to the fixed interval $[0, \tilde S]$.
Note that the construction of the extension 
guarantees the monotony of $\{\bar\tau_n\}_{n\in \N}$ in $[s^n_{N_n}, \tilde S]$.

While Assumption~\ref{assu:nested}\eqref{it:tolest} and \eqref{it:termination} are fulfilled by the 
basic version of our adaptive algorithm, cf.~Proposition~\ref{endtime} and Remark~\ref{rem:endtime},
the third assumption on the step size function need not be satisfied in general.
For this reason, the update rule for the step size from Algorithm~\ref{alg:adaptlocmin} 
has to be modified, if $\ZZ \neq \VV$, in order to ensure Assumption~\ref{assu:nested}\eqref{it:nested}.
We again underline that the results of Section~\ref{sec:prelim} remain valid for any other choice of the 
step size, as already noted in Remark~\ref{rem:tau_egal}, so we still have them at our disposal, no matter 
how the step size is chosen.

In the following, we present one possible modification of our basic algorithm that fulfills 
Assumption~\ref{assu:nested}. The principle idea is to use nested grids, where 
the set of grid points $\{s_k^{n}\}_{0 = 1, ..., N_{n}}$ 
for the tolerance $\tol_{n}$ contains all grid points $s_k^j$, $k=0, ..., N_j$, 
corresponding to all previous tolerances $\tol_j$, $j < n$. 
This is realized as follows:

\begin{algorithm}[Adaptive Scheme Producing Nested Grids]\label{alg:nested}
\ 
\begin{algorithmic}[1]
    \STATE Choose $\tol_1 > 0$ and $\tau_1 > 0$, set $n=1$.
    \LOOP
        \STATE Set $t_0^n=0$, $s_0^n=0$, $k=1$, $\sigma^n_1 = \tau_1$.
        \WHILE {$t_{k-1}^n<T$} \label{it:innerabbruch}
           \STATE Set $\tau^n_k = \sigma^n_k$\label{it:sigmagleichtau}
           \IF{$n>1$}\label{it:n1} 
                \STATE Define 
                \begin{equation}\label{eq:bars}
                    \bar s^n_{k} := 
                    \min_{1\leq j\leq  n-1}\; \min_{1 \leq \ell\leq  N_j}\{s^{j}_\ell : s^{j}_\ell > s^n_{k-1}\}
                \end{equation}  
                (with the convention $\min \emptyset = \infty$)                      
                \IF {$s^n_{k-1} + \sigma^n_k > \bar s^n_k$}
                    \STATE\label{it:taumod} Set $\tau^n_k = \bar s^n_k - s^n_{k-1}$
                \ENDIF
            \ENDIF
            \STATE\label{it:solvestat} Perform Steps~\ref{it:stat}--\ref{it:5} from Algorithm~\ref{alg:adaptlocmin} 
            to compute the constant and affine interpolants and the residua $I_{1,n}^k$ and $I_{2,n}^k$
            \IF{$I_{1,n}^k < \tol_n$ and $I_{2,n}^k < \tol_n$}\label{it:lesstol}
                \STATE Set $k=k+1$.
                \IF {$I_{1,n}^k<\frac{\tol_n}{2}$ and $I_{2,n}^k<\frac{\tol_n}{2}$}
                    \STATE\label{it:taudouble}  Set $\sigma^n_{k+1}=2\,\sigma^n_k$.
                \ENDIF
	        \ELSE
        		\STATE Set\label{it:tauhalf} $\sigma^n_k=\frac{1}{2} \,\tau^n_k$ and \textbf{go to} Step~\ref{it:sigmagleichtau}.
	        \ENDIF 
        \ENDWHILE
        \STATE Choose $\tol_{n+1} < \tol_n$, set $n = n+1$.
    \ENDLOOP
\end{algorithmic}
\end{algorithm}

It is to be noted that, as long as $s^{n-1}_{N_{n-1}} > s^n_{k-1}$, it suffices to consider $j=n-1$ 
in \eqref{eq:bars}, because the grids are nested. 
Although it is rather obvious that Algorithm~\ref{alg:nested} indeed satisfies Assumption~\ref{assu:nested}, 
it is fairly technical to prove. For this reason, we postpone the proof of the following 
proposition to Appendix~\ref{sec:nested}.

\begin{proposition}\label{prop:nested}
    Algorithm~\ref{alg:nested} satisfies Assumption~\ref{assu:nested}.
\end{proposition}

Now, we are in the position to turn to our main convergence result.
As in Section~\ref{sec:convRn}, we again consider a sequence $\{\tol_n\}_{n\in \N}$ tending 
monotonically to zero and run an adaptive algorithm with that sequence, 
that fulfills Assumption~\ref{assu:nested} (not necessarily Algorithm~\ref{alg:nested}).
We use the same notation for the associated sequence of interpolants etc.\ as in Section \ref{sec:convRn}, 
i.e., e.g., $\hat{z}_n$, $\hat{t}_n$, and so on.
Moreover, we again extend the affine and constant interpolants beyond $s^n_{N_n}$
by constant continuation and denoted these extensions by the same symbols.

\begin{theorem}\label{konvergenz}
    Let $\{\tol_n\}_{n\in \N}$ be a sequence converging monotonically to zero and assume that a step size 
    update is used such that Assumption~\ref{assu:nested} is fulfilled.
    Then there exists a subsequence such that 
	\begin{alignat}{3}
		S_{n_k} & \to S  & \quad & \text{in }\R \label{S} \\
		\hat{t}_{n_k} & \rightharpoonup^* \hat{t} & &\text{in } W^{1,\infty}(0,S;\R) \label{t} \\
		\hat{z}_{n_k} &\rightharpoonup^* \hat{z} & & \text{in } W^{1,\infty}(0,S;\VV)\cap H^1(0,S;\ZZ) \label{z} 
	\end{alignat}
	and the limit $(\hat{t},\hat{z})$ is a $\V$-parametrized BV solution of the rate-independent system. 
	Furthermore, every accumulation point of $\{(S_n,\hat{t}_n,\hat{z}_n)\}_{n\in \N}$ 
	with respect to the above convergence is a $\V$-parametrized BV solution.
\end{theorem}

\begin{proof}
    The proof is in principle similar to the one of Theorem~\ref{thm:convRn}, but, since 
    we do not have the tightened bounds from Lemma~\ref{lem:Itau2} at hand if $\ZZ\neq \VV$, 
    we need to substantially modify the derivation of the complementarity relation and the energy identity for 
    the weak limit. We again divide the proof in several steps.

    \vspace*{1ex}
    \noindent    
    \textsl{(i) Boundedness and existence of converging subsequences}
    
    Since Lemma~\ref{lem:sumz} does not depend on the choice of the step size, 
    \eqref{eq:Snbound} remains true for the modified algorithm, too, such that $\tilde S$ is finite.
    Owing to the constant continuation we have chosen, the bounds from 
    Lemmas~\ref{lem:complaffine} and \ref{tzbeschr} transfer to $(0,\tilde S)$ and hence, there is 
    a constant $C>0$ with
    \begin{equation*}
        S_n \leq \tilde S, \quad 
        \|\hat t'\|_{L^\infty(0,\tilde S)} \leq C, \quad
        \|\hat z'\|_{L^\infty(0,\tilde S;\VV)} \leq C, \quad
        \|\hat z\|_{H^1(0,\tilde S; \ZZ)} \leq C.
    \end{equation*}
    Consequently there is a subsequence such that \eqref{S}--\eqref{z} holds true.

    \vspace*{1ex}
    \noindent    
    \textsl{(ii) Initial and end time condition} 
    
    In order to show that every accumulation point in the sense of \eqref{S}--\eqref{z} is a
    $\V$-parametrized BV solution, let us consider an arbitrary sequence $\{(S_n,\hat{t}_n,\hat{z}_n)\}_{n\in \N}$ 
    such that 
    \begin{equation}\label{eq:convTF}
        S_{n} \to S   \;\text{ in }\R, \quad
		\hat{t}_{n} \rightharpoonup^* \hat{t} \;\text{ in } W^{1,\infty}(0,S;\R) , \quad 
		\hat{z}_{n} \rightharpoonup^* \hat{z}  \;\text{ in } W^{1,\infty}(0,S;\VV)\cap H^1(0,S;\ZZ) .
    \end{equation}
    First we note that the linearity and continuity of $C([0,S];\ZZ) \ni z \mapsto z(s) \in \ZZ$, $s\in [0,S]$, 
    along with the continuous embedding $H^1(0, S; \ZZ) \embed C([0,S];\ZZ)$ yields the following pointwise 
    convergence
    \begin{equation}
		\hat{z}_{n}(s)  \rightharpoonup \hat{z}(s) \quad \text{in }\ZZ \mbox{~for all } s\in [0,S]. \label{punkt}    
    \end{equation}
    In addition, weak* convergence in $W^{1,\infty}(0,S)$ implies that $\hat t_n$ converges uniformly in $[0,S]$.
    With these pointwise convergences at hand, we can argue as in part (i) of the proof of Theorem~\ref{thm:convRn} 
    to see that the weak limit satisfies the initial and end time condition.

    \vspace*{1ex}
    \noindent    
    \textsl{(iii) Complementarity relation} 

    Concerning the first complementarity relation in \eqref{eq:komp1}, we again make use of the set $M$ 
    as defined in \eqref{eq:defM}. Since its weak closedness is completely independent of the space $\ZZ$, 
    the weak* convergence in \eqref{eq:convTF} together with Lemma~\ref{lem:complaffine} 
    yields $(\hat t', \hat z') \in M$ so that the weak limit indeed satisfies \eqref{eq:komp1}.
    
    For convenience, we use the following abbreviations in the rest of the proof:
    \begin{equation*}
    \begin{aligned}
        \widehat{\abst}(s) &:= \dist_{\VV^*}\{-D_z\II(\hat{t}(s),\hat{z}(s)),\partial\RR(0)\},\\
        \widehat{\abst}_n(s) &:= \dist_{\VV^*}\{-D_z\II(\hat{t}_n(s),\hat{z}_n(s)),\partial\RR(0)\},\\
        \overline{\abst}_n(s) &:= \dist_{\VV^*}\{-D_z\II(\underline{t}_n(s),\bar{z}_n(s)),\partial\RR(0)\},
    \end{aligned}
    \end{equation*}
    where $s\in [0,S]$ and $\underline{t}_n$ and $\bar z_n$ again denote the constant interpolants.
    
    To prove \eqref{komp}, let $s\in(0,S)$ and $\delta > 0$ be arbitrary. 
    Since $s< S$ and $S_n \to S$, there holds $s < S_n \leq s^n_{N_n}$ for $n\in \N$ sufficiently large. 
    Thus, for all $n$ large enough, the results of Section~\ref{sec:prelim} hold true
    and we do not have to care about the behavior of the discrete solution and its constant continuation, 
    respectively, beyond $s^n_{N_n}$, which 
    will be frequently used in the following. Since we are interested in the passage to the limit 
    $n\to \infty$ anyway, we thus tacitly assume that $n$ is sufficiently large 
    such that $s < s^n_{N_n}$ for the rest of the proof. 
    By Assumption~\ref{assu:nested}\eqref{it:nested}, the step size function $\bar{\tau}_n$ 
    converges pointwise to a measurable function $\bar{\tau}: [0,S] \to [0,\infty)$. 
    For this limit function, we define the sets
	\begin{equation}\label{eq:defG}
	\begin{aligned}
		G_0 & :=\{\sigma\in(0,s):\bar{\tau}(\sigma)=0\}, \\
		G_{<\delta} & :=\{\sigma\in(0,s):0<\bar{\tau}(\sigma)<\delta\}, \\
		G_{\geq\delta} & :=\{\sigma\in(0,s):\bar{\tau}(\sigma)\geq\delta\}.	
	\end{aligned}
	\end{equation}
    Note that $G_0\cup G_{<\delta}\cup G_{\geq\delta}=(0,s)$. 
	From the weak*-convergence of $\hat{t}_n'$ we deduce
    \begin{equation}\label{eq:complsplit}
    \begin{aligned}
		0&\leq \int_{0}^s\hat{t}'(\sigma)\,\widehat{\abst}(\sigma)\, d\sigma\\
		&=\lim_{n\to\infty} \int_{G_{0}}\hat{t}_n'(\sigma)
		\,\widehat{\abst}(\sigma)\,d\sigma
		+\lim_{n\to\infty} \int_{G_{<\delta}}\hat{t}_n'(\sigma)
		\,\widehat{\abst}(\sigma)\,d\sigma
		+\lim_{n\to\infty} \int_{G_{\geq\delta}}\hat{t}_n'(\sigma)
		\,\widehat{\abst}(\sigma)\,d\sigma\\
		& =: I_0 + I_{<\delta} + I_{\geq\delta}.    
    \end{aligned}        
    \end{equation}
    Next we will pass to the limit with the three integrals separately and start with $I_0$.
	For this reason, let $\sigma\in G_0$ be arbitrary. For every $n\in\N$, 
	there exists an interval $[s_{k}^n,s_{k+1}^n)$ such that $\sigma\in[s_{k}^n,s_{k+1}^n)$, 
	where the respective $k$ of course depends on $n$, but we suppress this dependency to ease notation. 
	The pointwise convergence $\bar{\tau}_n(\sigma)\to 0$ by the definition of $G_0$ 
	implies that the associated step size $\tau_{k+1}^n=s_{k+1}^n-s_{k}^n$ equals zero in the limit $n\to\infty$, too. 
	The construction of the interpolants along with Lemma~\ref{lem:complaffine} therefore implies
	\begin{gather}
		\|\hat{z}_n(\sigma)-\bar{z}_n(\sigma)\|_\V\leq|\sigma-s_k^n| \,\|\hat{z}_n'(\sigma)\|_\V
		\leq \tau_{k+1}^n\to 0, \label{zintabsch} \\
		|\hat{t}_n(\sigma)-\underline{t}_n(\sigma)|\leq|\sigma-s_{k+1}^n| \,\hat{t}_n'(\sigma)\leq \tau_{k+1}^n \to0
	\end{gather}
	for $n\to\infty$. Now \eqref{punkt} 
	together with the compactness of the embedding $\ZZ\hookrightarrow\VV$ implies the strong convergence 
	$\hat{z}_n(\sigma)\to\hat{z}(\sigma)$ in $\VV$. Because of \eqref{zbeschr}, 
	the constant interpolant $\bar{z}_n(\sigma)$ is bounded in $\ZZ$, 
	which yields the existence of a weak limit in $\ZZ$ for a subsequence. 
	Again, the compactness $\ZZ\hookrightarrow\VV$ implies strong convergence in $\VV$ 
	and \eqref{zintabsch} gives that the weak limit of $\bar{z}_n(\sigma)$ equals $\hat{z}(\sigma)$.
    Since $\sigma \in G_0$ was arbitrary, we have proven that
	the constant and affine interpolants have the same pointwise weak limit in $G_0$. 
	For the time interpolants one can argue in exactly the same way to show that 
	$\underline{t}(\sigma)\to \hat t(\sigma)$ for all $\sigma\in G_0$. 
	Therefore, we deduce from \eqref{ikonv}
	\begin{equation*}
		D_z\II(\underline{t}_n(\sigma),\bar{z}_n(\sigma))
		\weak D_z\II(\hat{t}(\sigma),\hat{z}(\sigma)) \quad \text{in }\ZZ^* \quad \text{for all }\sigma\in G_0,
	\end{equation*}
	which together with the weak lower semicontinuity of $\dist$ by Lemma~\ref{distuhs} in the appendix gives
	\begin{equation}
	\begin{aligned}
	    \widehat{\abst}(\sigma) 
	    &= \dist_{\VV^*}\{-D_z\II(\hat{t}(\sigma),\hat{z}(\sigma)),\partial\RR(0)\} \\
		&\leq \liminf_{n\to\infty} \dist_{\VV^*}\{-D_z\II(\underline{t}_n(\sigma),\bar{z}_n(s)),\partial\RR(0)\} 
        = \liminf_{n\to\infty} \overline{\abst}_n(\sigma)
		\label{distG0}	
	\end{aligned}
	\end{equation}
    for all $\sigma\in G_0$. Due to $0 \leq \hat t_n'(\sigma) \leq 1$ a.e.\ in $(0,\tilde S)$ and \eqref{kompl}, 
    this in turn implies
	\begin{equation*}
        \limsup_{n\to\infty}\hat{t}_n'(\sigma)\,\widehat{\abst}(\sigma)
        = \limsup_{n\to\infty}\hat{t}_n'(\sigma)
        \big(\widehat{\abst}(\sigma) -  \overline{\abst}_n(\sigma)\big)\leq 0
        \quad \text{a.e.\ in } G_0.
	\end{equation*}
	Hence, Fatou's lemma implies for the first limit
	\begin{equation}\label{eq:estI0}
		I_0 \leq
		\int_{G_{0}} \limsup_{n\to\infty}\hat{t}_n'(\sigma)\,\widehat{\abst}(\sigma)\,d\sigma\leq 0.
	\end{equation}

    Next, we turn to the second integral. Similarly to \eqref{distG0}, the weak lower semicontinuity 
    of $\dist_{\VV^*}$ in combination with \eqref{eq:convTF} and \eqref{ikonv} yields  
	\begin{equation}
	\begin{aligned}
    	\widehat{\abst}(\sigma)
		&= \dist_{\VV^*}\{-D_z\II(\hat{t}(\sigma),\hat{z}(\sigma)),\partial\RR(0)\} \\
		&\leq \liminf_{n\to\infty} \dist_{\VV^*}\{-D_z\II(\hat{t}_n(\sigma),\hat{z}_n(\sigma)),\partial\RR(0)\}  
		=  \liminf_{n\to\infty} \widehat{\abst}_n(\sigma) \label{distGdelta}	
	\end{aligned}
	\end{equation}
	for all $\sigma\in G_{<\delta}\cup G_{\geq\delta}$. 
    Therefore, thanks to \eqref{beschr} and \eqref{distbeschr2}, the integrands in $I_{<\delta}$ are bounded 
    by a constant $C>0$ independent of $n$ and $\delta$ and we obtain
	\begin{equation}\label{eq:estIless}
		I_{< \delta} = \lim_{n\to\infty} \int_{G_{<\delta}}\hat{t}_n'(\sigma)\,\widehat{\abst}(\sigma)\,d\sigma
		\leq C\,|G_{<\delta}|.
	\end{equation}

    For the last limit we obtain 
    \begin{equation}\label{eq:estIgreater}
		I_{\geq \delta}
		\leq \limsup_{n\to\infty}\int_{G_{\geq\delta}}\hat{t}_n'(\sigma)
		\big(\widehat{\abst}(\sigma)-\widehat{\abst}_n(\sigma)\big) d\sigma 
		+\limsup_{n\to\infty}\int_{G_{\geq\delta}}\hat{t}_n'(\sigma)\,\widehat{\abst}_n(\sigma)\,d\sigma.
    \end{equation}
    Similarly to above, it follows from \eqref{distGdelta} and $\hat t_n' \in [0,1]$ a.e.\ in $(0,\tilde S)$ that, 
    for all $\sigma\in G_{\geq\delta}$,
	\begin{equation*}
		\limsup_{n\to\infty}\hat{t}_n'(\sigma)
		\big(\widehat{\abst}(\sigma)-\widehat{\abst}_n(\sigma)\big)\leq 0
	\end{equation*}
	so that, again by Fatou's lemma, we can estimate the first integral by
    \begin{equation}\label{eq:estIg1}
    \begin{aligned}
		\limsup_{n\to\infty}\int_{G_{\geq\delta}}\hat{t}_n'(\sigma)
		\big(\widehat{\abst}(\sigma)-\widehat{\abst}_n(\sigma)\big)d\sigma 
		\leq 0.
	\end{aligned}    
    \end{equation}
	To estimate the second integral, we introduce the set
	\begin{equation}\label{eq:defMk}
		M_\delta^n:=\big\{k\in \{0, ..., N_n-1\}:\tau_{k+1}^n=s_{k+1}^n-s_k^n\geq \delta\big\}, \quad n\in\N. 
	\end{equation}
	It describes the parts, where the algorithm uses step sizes larger than $\delta>0$. 
	Since $S_n$ is bounded by $\tilde{S}$, 
	$M_\delta^n$ has less than $\lceil\frac{\tilde{S}}{\delta}\rceil$ elements for every $n\in\N$.
	This together with our first residuum $I^k_{1,n}$ and the condition in Step~\ref{it:lesstol} allows us to estimate
    \begin{equation}\label{eq:estIg2}
    \begin{aligned}
		\limsup_{n\to\infty}\int_{G_{\geq\delta}}\hat{t}_n'(\sigma)\,\widehat{\abst}_n(\sigma) \,d\sigma 
		&\leq \limsup_{n\to\infty}\sum_{k\in M_{\delta}^{n}}\int_{s_{k}^n}^{s_{k+1}^n}\hat{t}_n'(\sigma)
		\,\widehat{\abst}_n(\sigma)\, d\sigma \\
		&\leq \limsup_{n\to\infty} \bigg\lceil\frac{\tilde{S}}{\delta}\bigg\rceil \tol_n =  0,    
    \end{aligned}        
    \end{equation}
	where we made use of the non-negativity of the above integrand 
	and the pointwise monotony of the meshsize functions $\bar{\tau}_n$, which leads to 
	$G_{\geq\delta}\subset \bigcup_{k\in M_{\delta}^{n}} [s_{k}^n,s_{k+1}^n]$ 
	for every $n\in\N$.

	Finally combining \eqref{eq:estI0}, and \eqref{eq:estIless}--\eqref{eq:estIg2} with \eqref{eq:complsplit} leads to 
	\begin{equation*}
		0 \leq \int_{0}^s\hat{t}'(\sigma)\, \widehat{\abst}(\sigma)\, d\sigma\leq C\, |G_{<\delta}|.
	\end{equation*}
	Since $\delta$ was arbitrary, this holds for every $\delta > 0$. From the definition of the set $G_{<\delta}$ in 
	\eqref{eq:defG}, it follows that $|G_{<\delta}|\to 0$ as $\delta \searrow 0$.
    Together with the non-negativity of the integrand and the arbitrariness of $s\in (0,S)$, this yields 
	\begin{align*}
		\hat{t}'(\sigma)\dist_{\VV^*}\{-D_z\II(\hat{t}(\sigma),\hat{z}(\sigma)),\partial\RR(0)\}=0
		\quad \text{a.e.\ in } (0,S),
  	\end{align*}
    which is the desired complementarity relation in \eqref{komp}.

    \vspace*{1ex}
    \noindent    
    \textsl{(iv) Energy identity} 
    
    To prove the energy equality, let again $s\in (0,S)$ and $\delta > 0$ be arbitrary and define the sets
    $G_0$, $G_{<\delta}$, and $G_{\geq \delta}$ as above.
    We again assume w.l.o.g.\ that $n\in \N$ is so large that $s< s^n_{N_n}$.
	We consider each part in the energy equality individually and, except for the integral involving the distance,
    all parts can be discussed in the same way as in the case $\VV = \ZZ$.
    To be more precise, in view of the uniform convergence of  $\hat t_n$, the pointwise weak convergence 
    in \eqref{punkt} (here in $\ZZ$ and not only in $\VV$), the weak convergence of 
    $\hat t_n'$ and $\hat z'_n$, and our standing assumptions on energy and dissipation 
    give that \eqref{eq:IIlsc}--\eqref{eq:dtIIconv} from the proof of Theorem~\ref{thm:convRn} also hold 
    in case $\VV\neq \ZZ$.
	 
	 Concerning the distance term, we exploit \eqref{distG0}, \eqref{distGdelta}, 
	 $\|\hat{z}'(\sigma)\|_\V\leq1$ by \eqref{eq:komp1}, and Fatou's lemma in order to obtain
	\begin{align*}
		\int_0^s\|\hat{z}'(\sigma)\|_\V \,\widehat{\abst}(\sigma)\,d\sigma 
		&\leq \int_{G_{0}}\widehat{\abst}(\sigma)\,d\sigma
		+\int_{G_{<\delta}\cup G_{\geq\delta}}\,	\widehat{\abst}(\sigma)\, d\sigma\\
		&\leq \int_{G_{0}}\liminf_{n\to\infty}  \overline{\abst}_n(\sigma) \,d\sigma 
		+\int_{G_{<\delta}\cup G_{\geq\delta}}	\liminf_{n\to\infty}  \widehat{\abst}_n(\sigma)\, d\sigma \\
		&\leq \liminf_{n\to\infty} 
		\Big(\int_{G_{0}} \overline{\abst}_n(\sigma) \,d\sigma 
	    + \int_{G_{<\delta}\cup G_{\geq\delta}}\widehat{\abst}_n(\sigma)\, d\sigma \Big).
	\end{align*}
    Combining this with \eqref{eq:IIlsc}--\eqref{eq:dtIIconv} leads to
	\begin{align*}
 		&\II(\hat{t}(s),\hat{z}(s))-\II(\hat{t}(0),\hat{z}(0))
 		-\int_{0}^{s} \partial_t\II(\hat{t}(\sigma),\hat{z}(\sigma))\hat{t}'(\sigma)d\sigma 
		+\int_{0}^{s}\RR(\hat{z}'(\sigma))+\|\hat{z}'(\sigma)\|_\V\, \widehat{\abst}(\sigma)\, d\sigma   \\
		&\leq \liminf_{n\to \infty}\Big(\II(\hat{t}_n(s),\hat{z}_n(s))-\II(\hat{t}_n(0),\hat{z}_n(0))
		-\int_{0}^{s} \partial_t\II(\hat{t}_n(\sigma),\hat{z}_n(\sigma))\hat{t}_n'(\sigma)d\sigma 
		+\int_0^s \RR(\hat{z}_n'(\sigma))d\sigma\Big)\\
		&\quad +\liminf_{n\to\infty} 
        \Big( \int_{G_{0}} \overline{\abst}_n(\sigma) \,d\sigma 
        + \int_{G_{<\delta}\cup G_{\geq\delta}} \widehat{\abst}_n(\sigma)\, d\sigma \Big)\\
		& \leq \limsup_{n\to\infty}
        \begin{aligned}[t]
    		\bigg( &\II(\hat{t}_n(s),\hat{z}_n(s))-\II(\hat{t}_n(0),\hat{z}_n(0))+\int_{0}^{s}\big(\RR(\hat{z}_n'(\sigma))
	    	+\overline{\abst}_n(\sigma) \big) d\sigma  \\
		    &-\int_{0}^{s} \partial_t\II(\hat{t}_n(\sigma),\hat{z}_n(\sigma))\hat{t}_n'(\sigma)d\sigma
		    +\int_{G_{<\delta}\cup G_{\geq\delta}} 
		    \big( \widehat{\abst}_n(\sigma) -\overline{\abst}_n(\sigma)  \big) d\sigma\bigg)
		\end{aligned}\\
		&= \limsup_{n\to\infty}\bigg(\int_0^s r_n(\sigma) d\sigma 
       +\int_{G_{<\delta}\cup G_{\geq\delta}} 
       \big( \widehat{\abst}_n(\sigma) -\overline{\abst}_n(\sigma)  \big) d\sigma\bigg)\\
		&\leq 
		\begin{aligned}[t]
	        \limsup_{n\to\infty}\int_{G_{0}} r_n(\sigma) d\sigma
		     & +\limsup_{n\to\infty}\int_{G_{<\delta}} 
  		    \big(r_n(\sigma) + \widehat{\abst}_n(\sigma) -\overline{\abst}_n(\sigma)  \big) d\sigma \\
		    & +\limsup_{n\to\infty}\int_{G_{\geq\delta}} 
		    \big(r_n(\sigma) + \widehat{\abst}_n(\sigma) -\overline{\abst}_n(\sigma)  \big) d\sigma
		    \end{aligned}\\
	    &=: J_0 + J_{<\delta} + J_{\geq \delta},
	\end{align*}
	where we made use of the discrete energy equality \eqref{diskenergy}.
	For the first limit, we obtain due to the definition of $G_0$, \eqref{rabsch}, and Fatou's lemma
	\begin{equation}
        J_0 \leq \int_{G_{0}} \limsup_{n\to\infty}r_n(\sigma) d\sigma \leq 0.
	\end{equation}
    Owing to \eqref{rabsch}, \eqref{distbeschr}, \eqref{distbeschr2}, and the constant continuation, 
    we see that the integrand in $J_{<\delta}$ is pointwise bounded from above by a constant $C>0$ independent of $n$.
	Thus, we obtain
	\begin{equation}
		J_{<\delta}\leq C\, |G_{<\delta}|.
	\end{equation}
	For the last limit, we use \eqref{beschr}, \eqref{kompl}, 
	and the definition of $M_\delta^n$ in \eqref{eq:defMk} in order to estimate
	\begin{align*}
        J_{\geq \delta}
		& = \limsup_{n\to\infty}
   		 \int_{ G_{\geq\delta}} 
        \Big[ r_n(\sigma) + \|\hat{z}_n'(\sigma)\|_\V
        \big( \widehat{\abst}_n(\sigma) -\overline{\abst}_n(\sigma)  \big)
	     + \hat{t}_n'(\sigma) \, \overline{\abst}_n(\sigma) \Big] d\sigma \\
		& \leq \limsup_{n\to\infty}\sum_{k\in M_{\delta}^{n}}
   		\int_{s_{k}^{n}}^{s_{k+1}^{n}}
        \Big[ r_n(\sigma) + \|\hat{z}_n'(\sigma)\|_\V
	     \big( \widehat{\abst}_n(\sigma) -\overline{\abst}_n(\sigma)  \big)
		 + \hat{t}_n'(\sigma) \,\overline{\abst}_n(\sigma) \Big] d\sigma \\
		 & = \limsup_{n\to\infty}\sum_{k\in M_{\delta}^{n}} I^k_{1,n} + I^k_{2,n}
		 \leq \limsup_{n\to\infty}2\bigg\lceil\frac{\tilde{S}}{\delta}\bigg\rceil \tol_n = 0,
	\end{align*}
   where the first inequality follows from the non-negativity of the integrand by Lemma~\ref{lem:I2pos}. 
    The last inequality follows analogously to \eqref{eq:estIg2} from the conditions on the residua
   in Step~\ref{it:lesstol} of the algorithm.
	
   All in all, we have proven that 
 	\begin{align*}
		&\II(\hat{t}(s),\hat{z}(s))-\II(\hat{t}(0),\hat{z}(0))
		-\int_{0}^{s} \partial_t\II(\hat{t}(\sigma),\hat{z}(\sigma))\hat{t}'(\sigma)d\sigma \\
		&\qquad +\int_{0}^{s}\RR(\hat{z}'(\sigma))
		+\|\hat{z}'(\sigma)\|_\V \,\dist_{\VV^*}\{-D_z\II(\hat{t}(\sigma),\hat{z}(\sigma)),\partial\RR(0)\}d\sigma
		\leq C\, |G_{<\delta}| \to 0 \quad\text{as } \delta \searrow 0
	\end{align*}
    such that $(\hat t, \hat z)$ fulfills an energy inequality in $(0,S)$. By continuity, this inequality 
    carries over to the whole interval $[0,S]$.
    From Remark~\ref{rem:energyid}, we know that this inequality is equivalent to the energy identity
    and, since $s\in [0,S]$ was arbitrary, we see that $(\hat t, \hat z)$ indeed satisfies \eqref{energie}, 
    which finally gives that the weak limit is a $\V$-parametrized BV solution.
\end{proof}

\begin{remark}
    Let us shortly explain why the additional Assumption~\ref{assu:nested}\eqref{it:nested} is needed.
    Unfortunately, we were not able to prove a result analogous to Lemma~\ref{lem:Itau2} in case of $\VV\neq \ZZ$. 
    In essence, the reason is that an estimate of the form \eqref{eq:Vstarest} does not hold, if the $\ZZ$-norm 
    and the $\VV$-norm are not equivalent. Even an estimate of the form $\|z_k - z_{k-1}\|_{\ZZ} \leq c \,\tau_k$
    would not suffice to prove \eqref{eq:Vstarest}, since $A$ maps $\ZZ$ to $\ZZ^*$ and not to $\VV^*$, which is 
    needed to estimate the distance term in \eqref{eq:distlip}.
    At this point, the distance-terms in $I^1_k$ and $I^2_k$ behave very differently compared to the $r$-term 
    from \eqref{eq:rdef}. In case of $r$, the critical term involving $A$ has the ``right'' sign, see \eqref{eq:rest}
    in Appendix~\ref{sec:I1I2}, so that the lack of regularity does not play a role here. 
    Note that, for this reason, we can only derive an estimate for $r$ and not for $|r|$, which luckily suffices for the 
    convergence analysis.
    To circumvent these regularity issues, we need the additional assumption of monotonically decreasing 
    step sizes. It essentially allows us to go without the sharpened estimate for the number of steps in 
    \eqref{Nbeschr} that is an immediate consequence of Lemma~\ref{lem:Itau2}.
\end{remark}


\section{Numerical Experiments}\label{sec:numerics}

In the following we test our adaptive algorithm by means of two examples.
The stationarity conditions in \eqref{alg1} are solved by means of a damped version of
the semi-smooth Newton method from \cite[Section~4]{FElocmin}.
For the evaluation of the residua in \eqref{i1} and \eqref{i2}, we use 
a composite Simpson's rule.

\subsection{A One Dimensonal Example}\label{sec:ex1d}

We start our investigations with an example, where $\VV = \ZZ = \R$. For the energy in \eqref{eq:energydef}, 
we choose 
\begin{equation*}
    \dual{Az}{z} := z^2, \quad 
    \FF \equiv 0, \quad
    f(t,z) := - \ell(t) \, z + 
    \begin{cases}
        4\,z + 8, ~& z\leq-2,\\
        4 - z^2,~&|z|<2,\\ 
        -4\,z + 8,~&z\geq2  
    \end{cases}
\end{equation*}
where $\ell$ is the external driving force, which is set to $\ell(t) := t+1$.
The end time is set to $T=5$.
It is easily seen that the standing assumptions in \eqref{F1}--\eqref{dtfkonv} are fulfilled with 
$c=6$, $\mu = 1$, and $\nu = 2$. By direct calculations one moreover verifies that 
\begin{equation}
	\hat{t}(s) := \begin{cases}s,~&s\in[0,2], \\2,~&s\in(2,10],\\(s-6)/2,~&s\in(10,16] \end{cases} 
	\quad \text{and}\quad 
	\hat{z}(s) := \begin{cases} -2,~&s\in[0,2],\\s-4,~&s\in(2,10],\\ (s+2)/2,~&s\in(10,16]\end{cases} \label{ex1s}
\end{equation}
is a $\V$-parametrized BV solution in the sense of Definition~\ref{param}. 
The solution is shown in Figure~\ref{fig:1Dexakt}.

\begin{figure}[h]
\centering
\begin{tikzpicture}[scale=0.8]
\draw[->] node[left]{\footnotesize 0} (0,0) -- (8.5,0) node[right] {$s$};
\draw[->] (0,-1.5) -- (0,5) node[left] {};
\draw[dashed,very thin] (1,-1.5) -- (1,5);
\draw[dashed,very thin] (5,-1.5) -- (5,5);
\draw[dashed,very thin] (8,-1.5) -- (8,5);
\draw[dashed,very thin] (8.5,1) -- (0,1) node[left]{\footnotesize 2};
\draw[dashed,very thin] (8.5,3) -- (0,3) node[left]{\footnotesize 6};
\draw[color=red, domain=0:1, line width=1pt] plot (\x,{\x}) node[right] {};
\draw[color=red, domain=1:5, line width=1pt] plot (\x,{1}) node[right] {};
\draw[color=red, domain=5:8, line width=1pt] plot (\x,{0.5*\x-1.5)}) node[right] {$\hat t$};
\draw[color=blue, domain=0:1, line width=1pt] plot (\x,{-1}) node[right] {};
\draw[color=blue, domain=1:5, line width=1pt] plot (\x,{\x-2}) node[right] {};
\draw[color=blue, domain=5:8, line width=1pt] plot (\x,{0.5*\x+0.5)}) node[right] {$\hat z$};
\draw (-0.3, -1) node{\footnotesize -2};
\draw (0.85, -0.2) node{\footnotesize 2};
\draw (4.75, -0.2) node{\footnotesize 10};
\draw (7.75, -0.2) node{\footnotesize 16};
\end{tikzpicture}
\caption{$\V$-parametrized BV solution from \eqref{ex1s}}\label{fig:1Dexakt}
\end{figure}
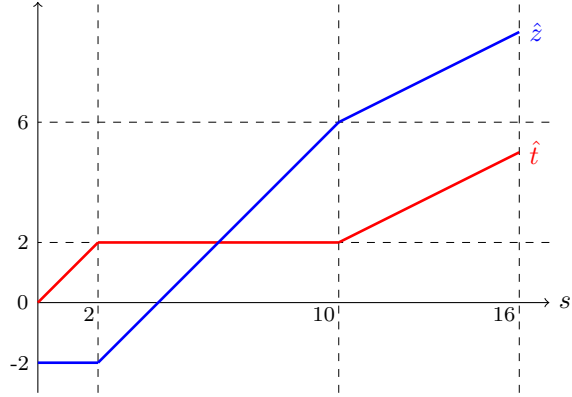
We observe that the state $\hat z$ is constant until $s = 2$. This \emph{sticking} behavior 
is typical for rate-independent processes and occurs, if the external loading is too small 
to change the system state such that the dissipation forces the system state to remain constant.
At $s = 2$, the system changes to a different regime. Now the physical time remains constant, 
while the system state starts to change and thus, we observe a discontinuous behavior of the system. 
In the majority of cases, the state is no more locally stable then, i.e., 
$-D_z\II(\hat t(s), \hat z(s)) \notin \partial\RR(0)$ and thus $\lambda$ from \eqref{eq:lambdadef} 
will be different from zero so that an additional viscous term in \eqref{eq:viscouseqlimit} appears. 
This gives rise to the notion of a \emph{viscous jump}. 
There may be discontinuities where the system state is still locally stable, but these situations are rather 
pathological as \cite[Example~2.3.5]{michael} demonstrates.
The viscous jump lasts until $s=10$ and afterwards, both $\hat t$ and $\hat z$ are both changing, 
a regime which is known as \emph{rate-independent slip}.

Figure~\ref{fig:1D} shows the result of the adaptive algorithm for two different values of the tolerance $\tol$.
\begin{figure}[h]
\begin{subfigure}[c]{0.49\textwidth}
    \centering
 	\includegraphics[width=0.9\textwidth]{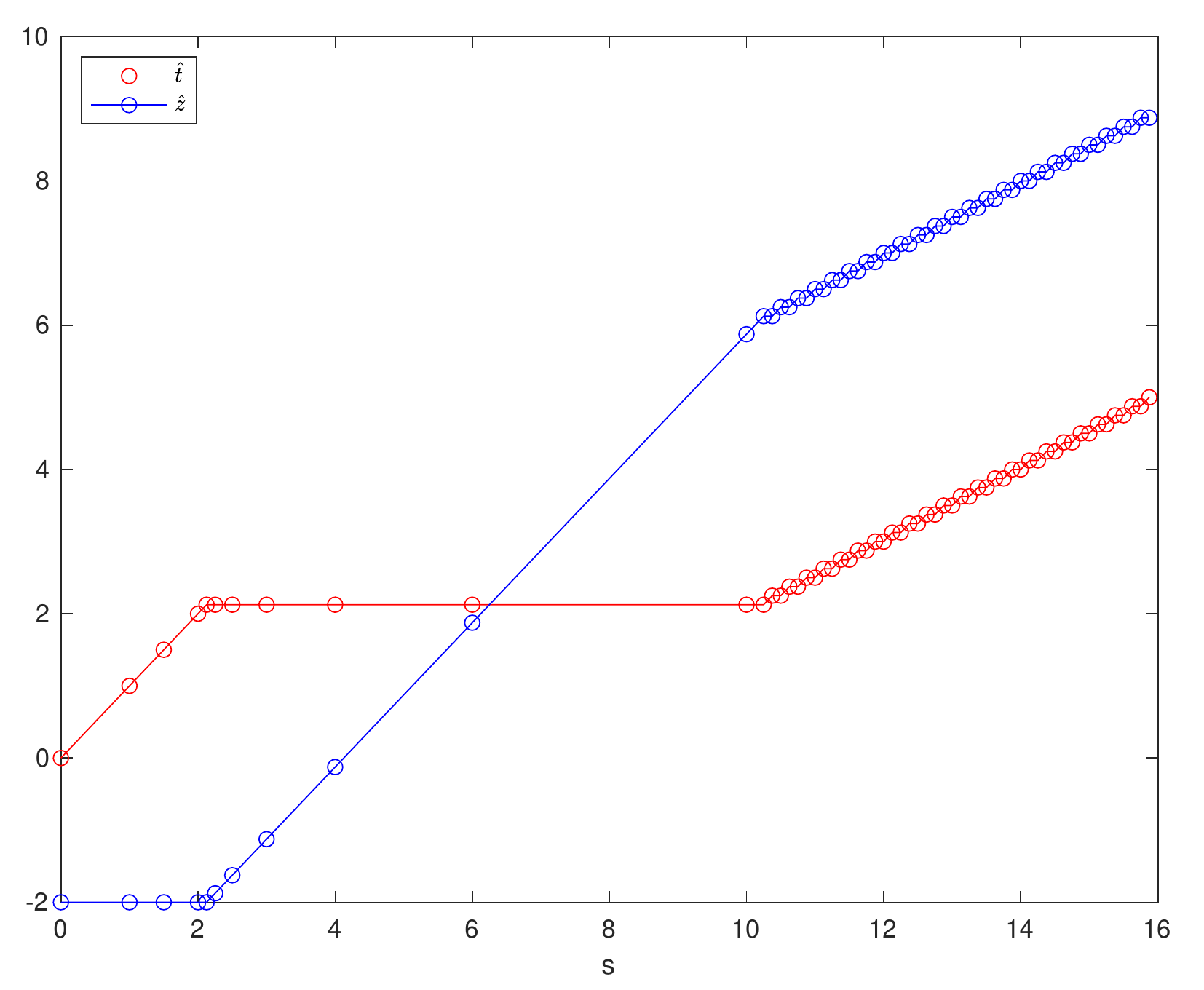}
    \subcaption{$\hat t$ and $\hat z$ for $\tol = 10^{-2}$}
\end{subfigure}
\begin{subfigure}[c]{0.49\textwidth}
    \centering
 	\includegraphics[width=0.9\textwidth]{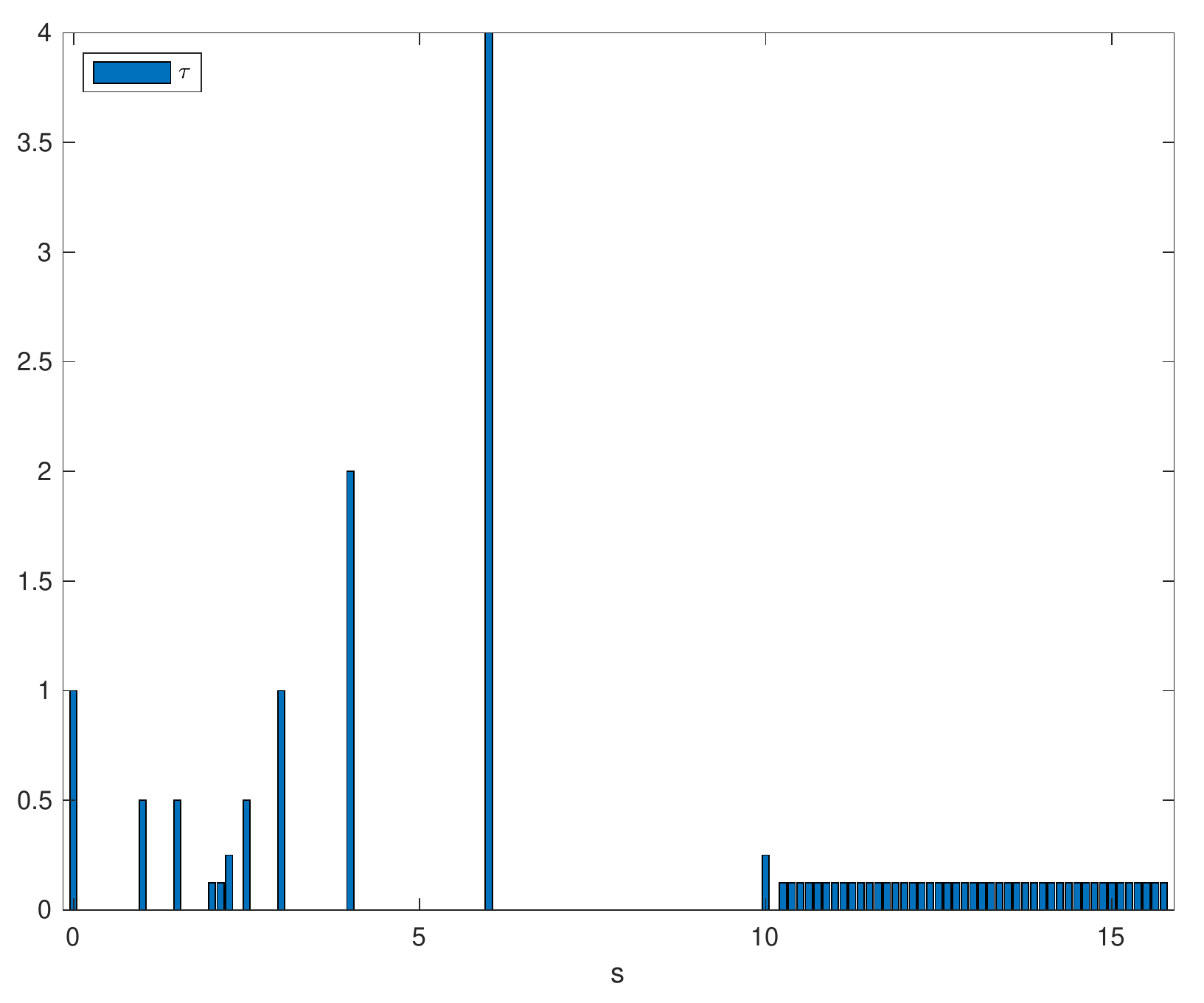}
    \subcaption{$\tau_k$ for $\tol = 10^{-2}$}
\end{subfigure}
\begin{subfigure}[c]{0.49\textwidth}
    \centering
 	\includegraphics[width=0.9\textwidth]{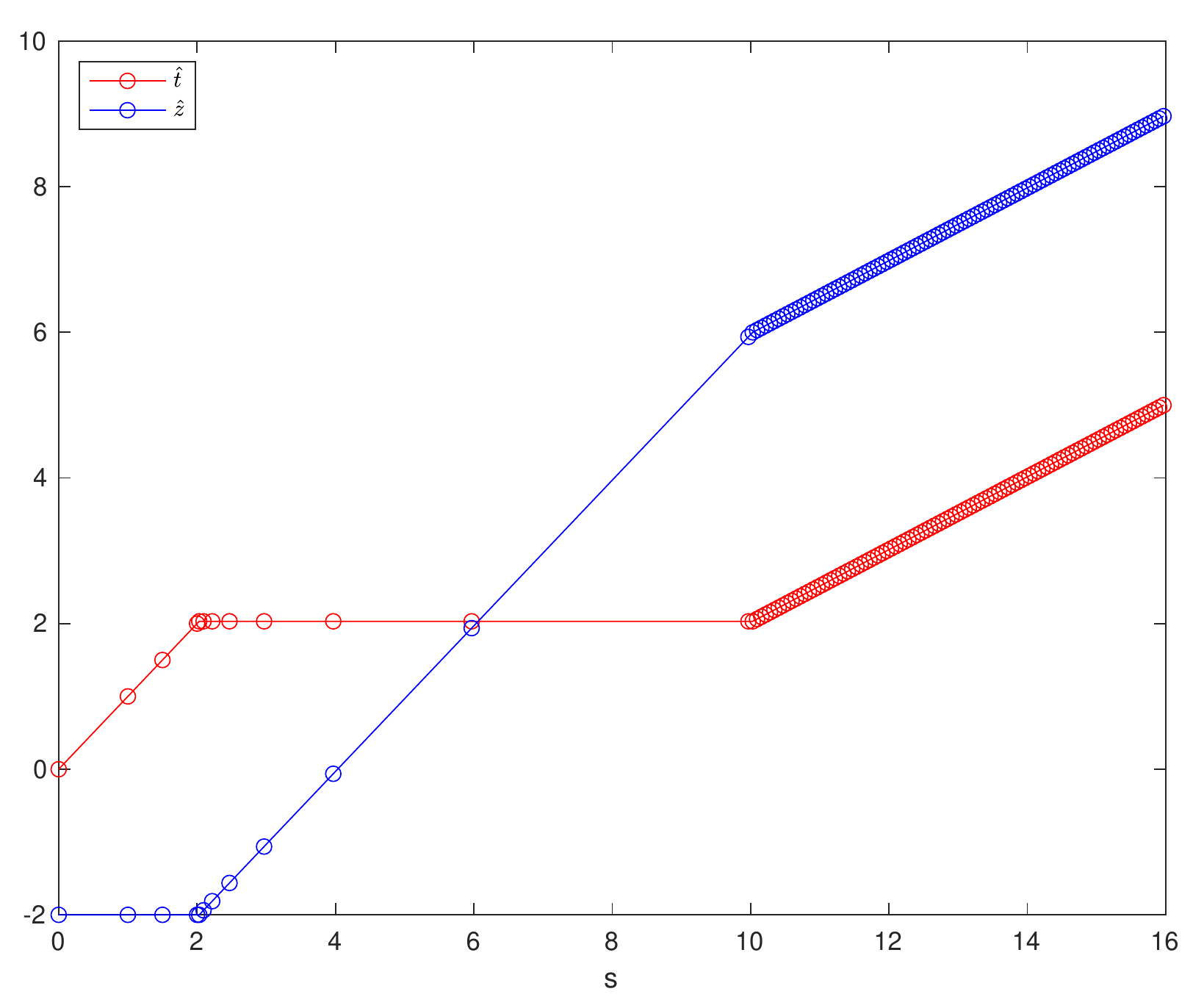}
    \subcaption{$\hat t$ and $\hat z$ for $\tol = 10^{-3}$}
\end{subfigure}
\begin{subfigure}[c]{0.49\textwidth}
    \centering
 	\includegraphics[width=0.9\textwidth]{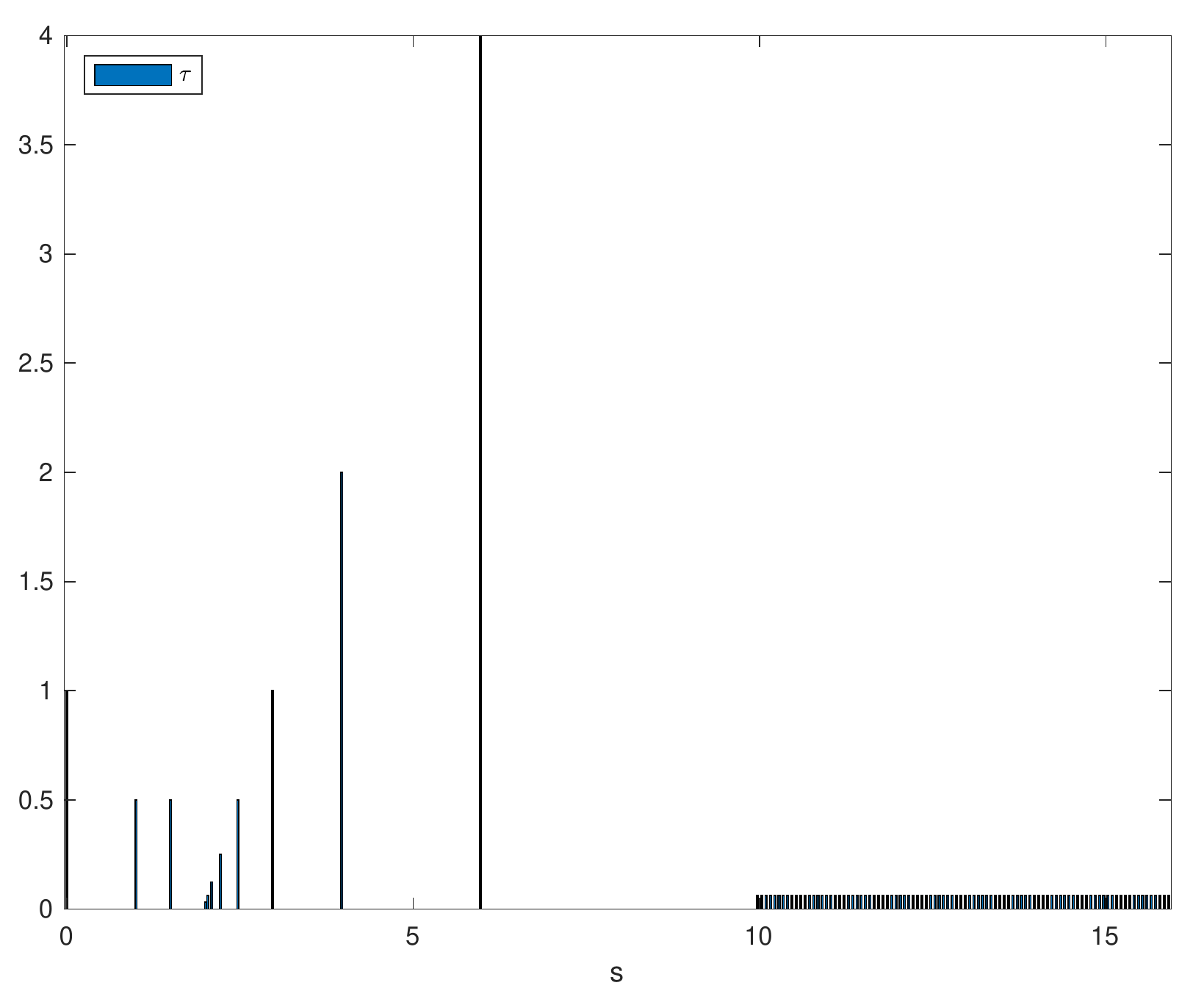}
    \subcaption{$\tau_k$ for $\tol = 10^{-3}$}
\end{subfigure}
\caption{Affine interpolants and step size for different tolerances}\label{fig:1D}
\end{figure}
We observe significantly larger step sizes during sticking and in a viscous jump accompanied with 
a refined solution of the break points at the transition between different regimes.
This is explained as follows:

\vspace*{1ex}
\noindent
\textsl{(i) Sticking:}\\
Assume that Step~\ref{it:stat} delivers $z_k = z_{k-1}$, i.e., the driving force in $t_{k-1}$ 
is so small that it is compensated by the dissipation and the system state does not change. 
Then, by Step~\ref{it:update}, the physical time is set to $t_k = t_{k-1} + \tau_k$ such that 
$\hat t'(s) = 1$ and $\hat z'(s) = 0$ for all $s\in [s_{k-1}, s_k)$. It thus holds
\begin{equation*}
    I^k_2 = 0 ,\quad 
    I^k_1 = \int_{s_k-1}^{s_k} \abst\{-D_z\II(\hat{t}(s),z_k),\partial\RR(0)\} ds .
\end{equation*}

Therefore, if the variations of the driving force $(s_{k-1}, s_k) \ni s \mapsto\ell(\hat t(s))$ are not too large, then 
the state $z_k$ remains locally stable on the whole interval, i.e., 
$-D_z\II(\hat t(s), z_k) \in \partial\RR(0)$ f.a.a.\ $s\in (s_{k-1}, s_k)$, and $I^k_1$ equals $0$, too.
The iteration is then accepted and the step size is doubled. 

\vspace*{1ex}
\noindent
\textsl{(ii) Viscous jump:}\\
Suppose that, in iteration $k\in \N$, the inequality constraint is active, i.e., 
$\|z_{k-1} - z_k\|_{\V} = \tau_k$. Then, by Step~\ref{it:update}, 
$\hat t'(s)=0$ and $\|\hat z'(s)\|_\V = 1$ for all $s\in [s_{k-1}, s_k)$ and hence
\begin{equation*}
    I^k_1 = 0 , \quad 
    I^k_2 = \int_{s_{k-1}}^{s_k}
    \big[\dual{D_z\II(\hat{t}(s),\hat{z}(s)}{\hat{z}'(s)} 
    +  \abst\{-D_z\II(\hat{t}(s),\hat{z}(s)),\partial\RR(0)\} + \RR(\hat z'(s)) \big] ds,
\end{equation*}
where we used \eqref{eig3} along with the positive homogeneity of $\RR$.
Now, the inf-convolution-formula gives for the distance
\begin{equation*}
    \abst\{\eta ,\partial\RR(0)\} 
    = \frac{1}{\tau_k} \,(\RR + I_{\tau_k})^*(\eta)
    = \sup_{\|v\|_{\V}\leq 1} \big( \dual{\eta}{v} - \RR(v)\big),
\end{equation*}
cf., e.g., \cite[Lemma~C.2]{FElocmin}.
In one dimension and $\V = 1$, it thus follows
\begin{equation*}
    \abst\{\eta ,\partial\RR(0)\} = \max\{ \eta - \RR(1),\,  - \eta - \RR(-1),\, 0 \} .
\end{equation*}
Therefore, $\hat z'(s) = \pm 1$ must only have the right sign and $I^k_2$ will vanish, too, 
such that the iteration is successful and the step size is doubled.
In multiple dimensions, this easy argument does of course not work, but, nonetheless, 
we frequently observe an increase of the step size in a viscous jump there, too, as the next example 
in Section~\ref{sec:sobolev} shows.

The aforementioned advantage of the adaptive scheme, i.e., 
the larger step size during \emph{sticking} and \emph{viscous jumps}, 
results in less iterations in order to approximate a solution 
with the same accuracy in comparison to an algorithm with fixed step size.
This is illustrated in Figure \ref{fig:steps}, where the difference between the numerical result 
(computed with fixed and adaptive step size) and the exact solution from \eqref{ex1s} measured in the 
$L^2$- and the $L^\infty$-norm is shown.
We observe that the qualitative behavior of the error is the same, but   
approximately only a quarter of iterations is needed in the adaptive case. 
\begin{figure}[h]
\begin{subfigure}[c]{0.49\textwidth}
    \centering
 	\includegraphics[width=\textwidth]{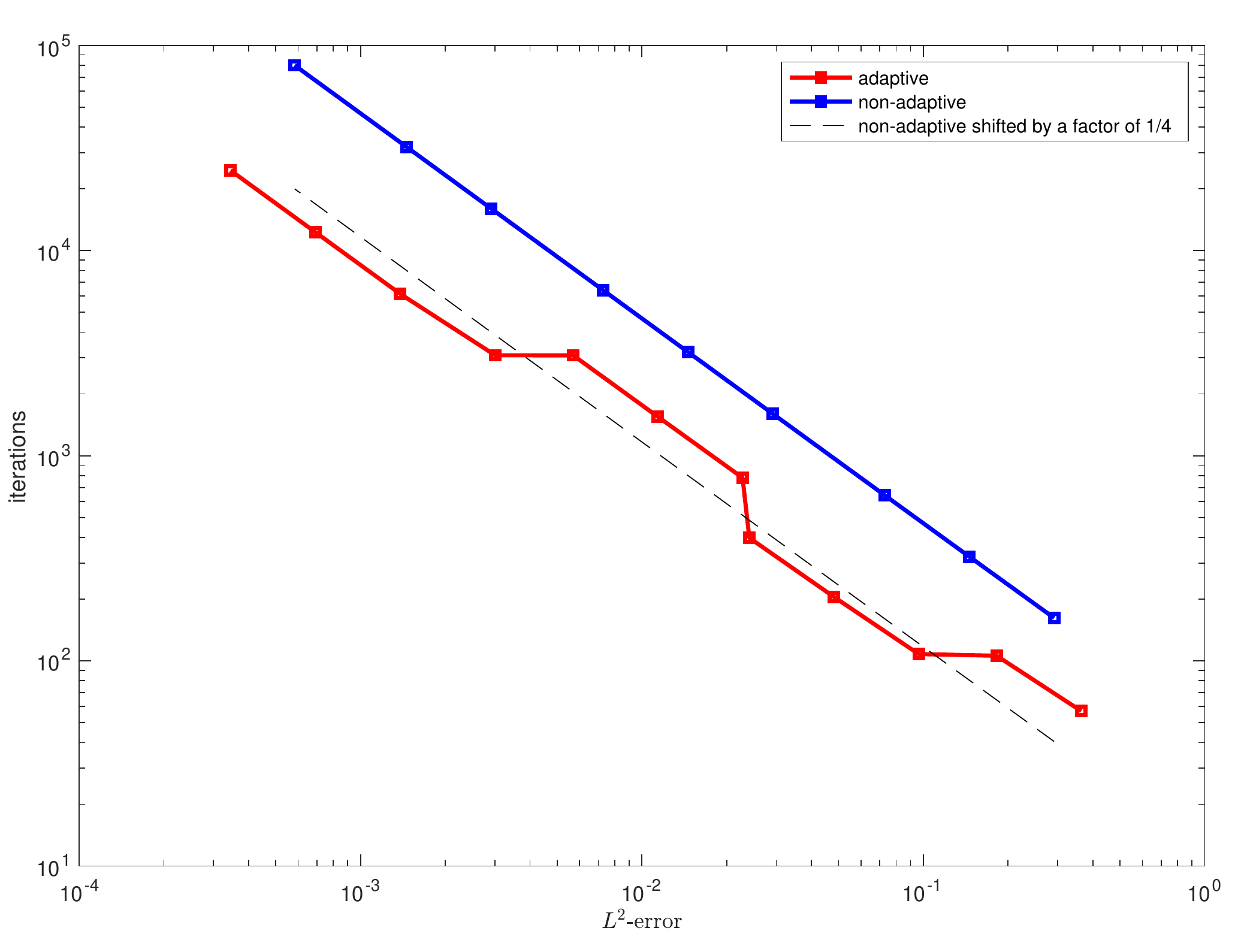}
\end{subfigure}
\begin{subfigure}[c]{0.49\textwidth}
    \centering
 	\includegraphics[width=\textwidth]{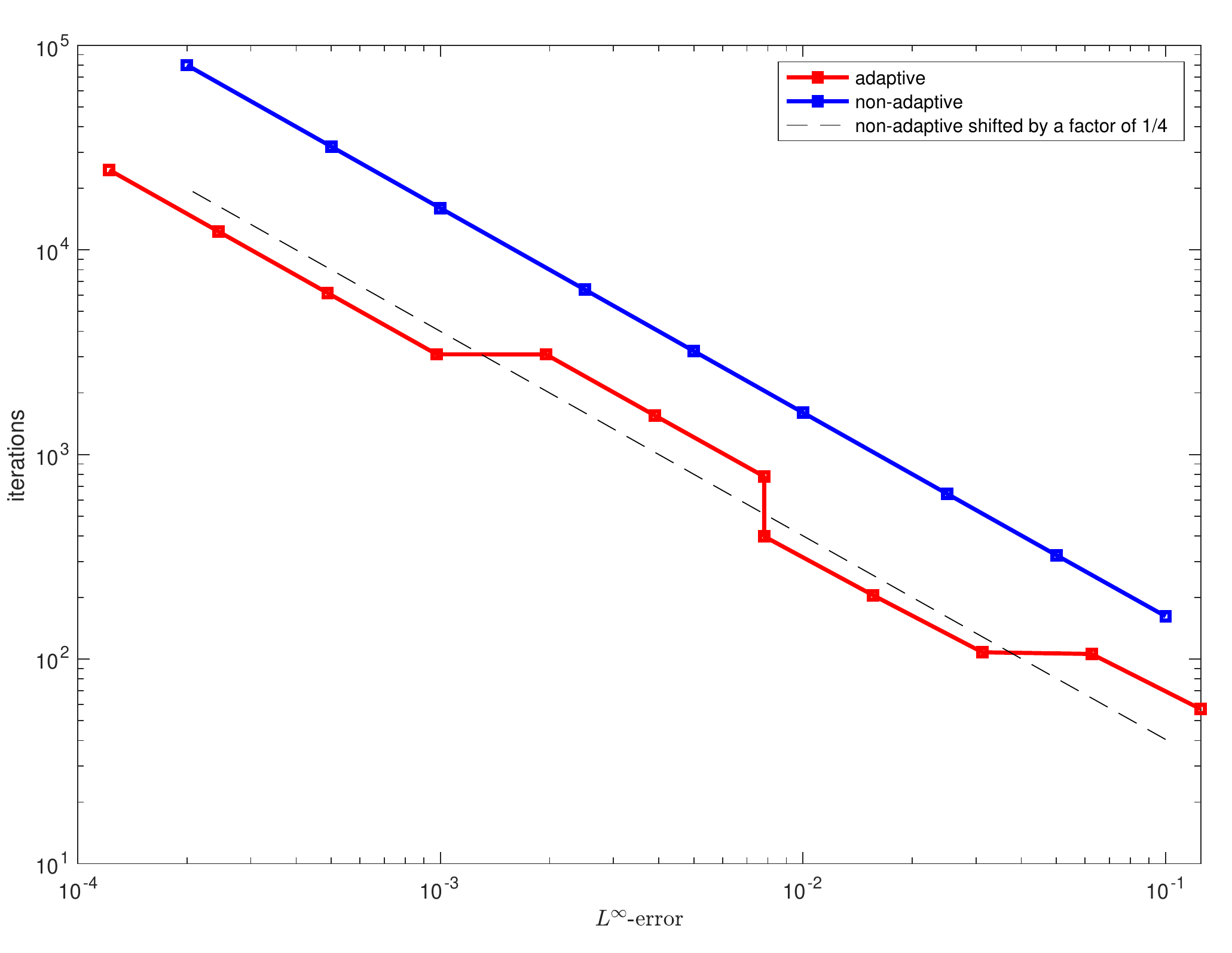}
\end{subfigure}
\caption{Comparison of the number of iterations between the adaptive and the non-adaptive scheme
in the ODE example}\label{fig:steps}
\end{figure}

However, it is to be noted that the error between ``the'' solution and its approximation is in general 
no meaningful quantity, since solutions are in general not unique and 
we can only guarantee that subsequences converge to $\V$-parametrized  BV solutions as seen in Theorem \ref{konvergenz}. It may therefore well happen that, for different tolerances,
different solutions are approximated as the next example shows.

\subsection{An Example in Sobolev Space}\label{sec:sobolev}

This example addresses a problem in function space, where $\VV\neq \ZZ$. 
To be more precise, we consider the setting from Example~\ref{ex:bspsystem} with 
\begin{equation*}
    T=1, \quad \Omega := (0,1)^2, \quad \V = \textup{id} : L^2(\Omega) \to L^2(\Omega).
\end{equation*}
Moreover, we set 
\begin{equation}\label{eq:doublewell}
    g : \R \to \R, \quad g(z) := 48 (1-z^2)^2
\end{equation}
(where the nonlinear function and the Nemyzki operator are denoted by the same symbol
with a slight abuse of notation). The associated Nemyzki operator is considered as a mapping
from $H^1_0(\Omega) \hookrightarrow L^8(\Omega)$ to $L^2(\Omega)$ and it is easily verified
that \eqref{eq:gsecond} is satisfied with $p = 8$ and $q = 2$. 
Furthermore, the external load from \eqref{eq:loadex} is set to
\begin{equation*}
    \langle \ell(t),z\rangle_{\VV^*,\VV}
    = - 200 \int_\Omega t^3 \, z(x)\,dx, \quad (t,z)\in [0,T]\times H_0^1(\Omega).
\end{equation*}
For the spatial discretization, we use standard linear finite elements with 1681 nodes along with the mass lumping 
approach for the $L^1$-norm as described in \cite[Section~4]{FElocmin}.

The result of Algorithm~\ref{alg:adaptlocmin} for the tolerances $\tol = 10^{-4}$, $\tol = 10^{-5}$ and $\tol = 10^{-6}$ is shown in 
Figure~\eqref{fig:pde}.
\begin{figure}[h]
\begin{subfigure}[c]{0.49\textwidth}
    \centering
 	\includegraphics[width=0.9\textwidth]{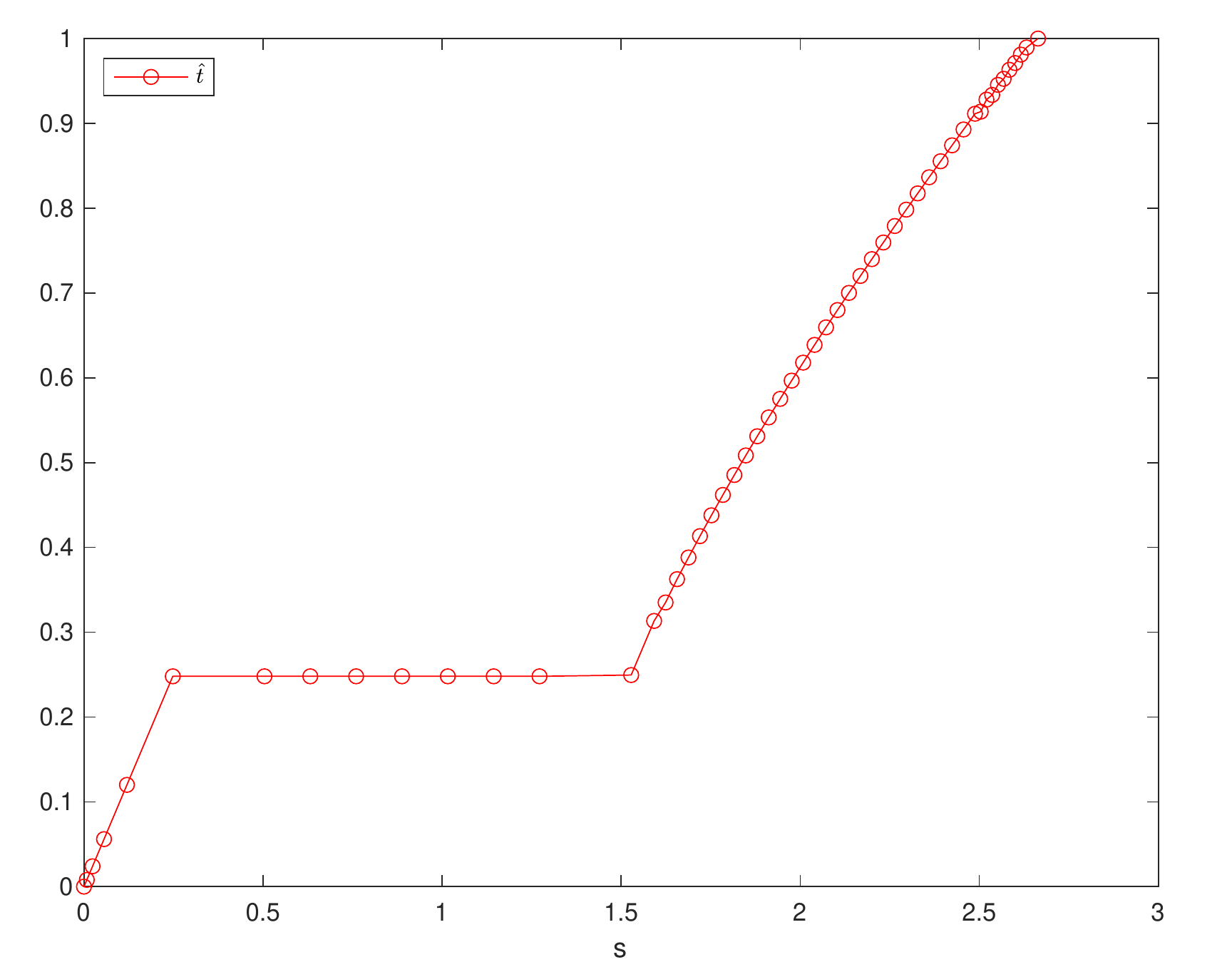}
    \subcaption{ $\hat t$ for $\tol = 10^{-4}$}
\end{subfigure}
\begin{subfigure}[c]{0.49\textwidth}
    \centering
 	\includegraphics[width=0.9\textwidth]{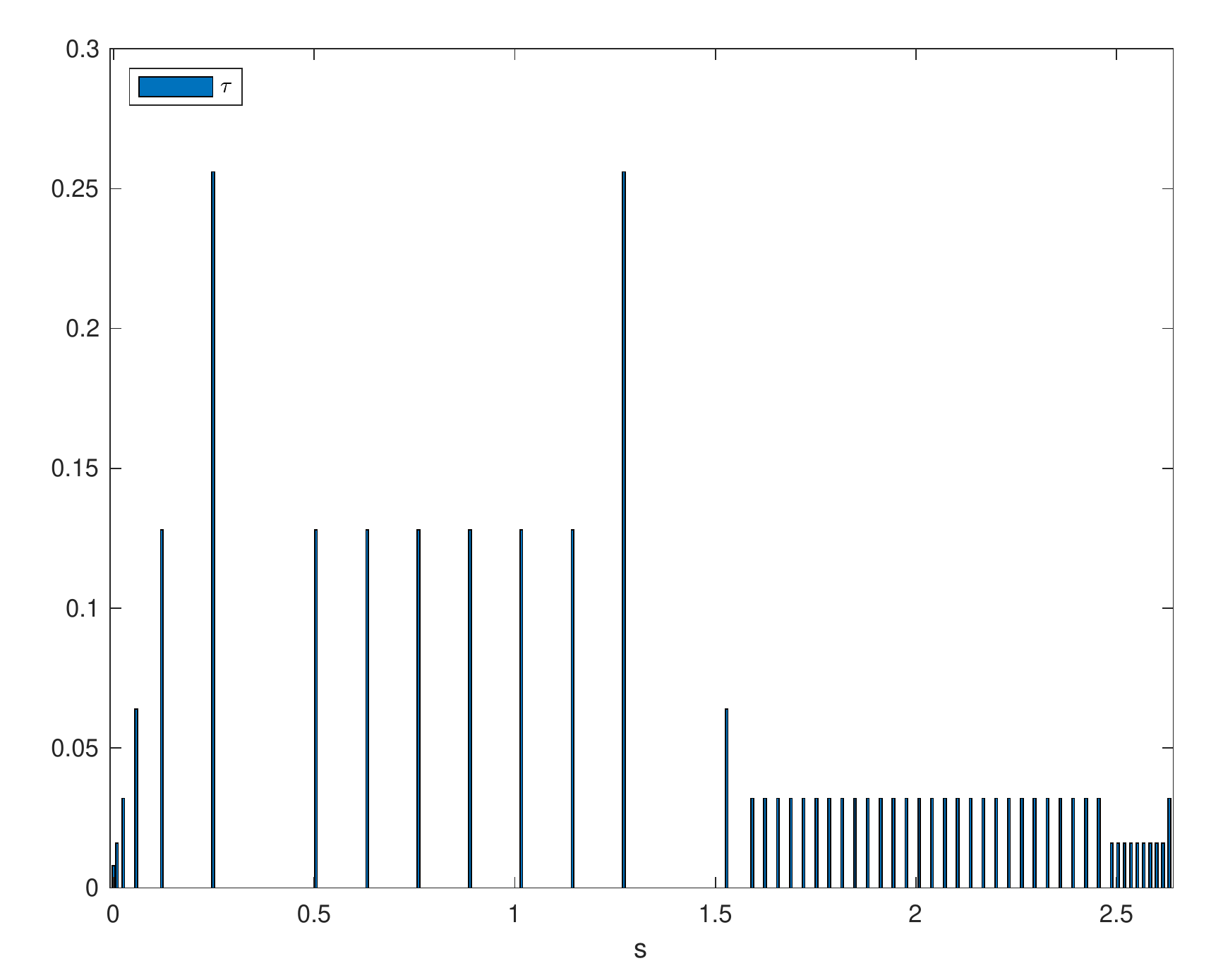}
    \subcaption{$\tau_k$ for $\tol = 10^{-4}$}
\end{subfigure}
\begin{subfigure}[c]{0.49\textwidth}
    \centering
 	\includegraphics[width=0.9\textwidth]{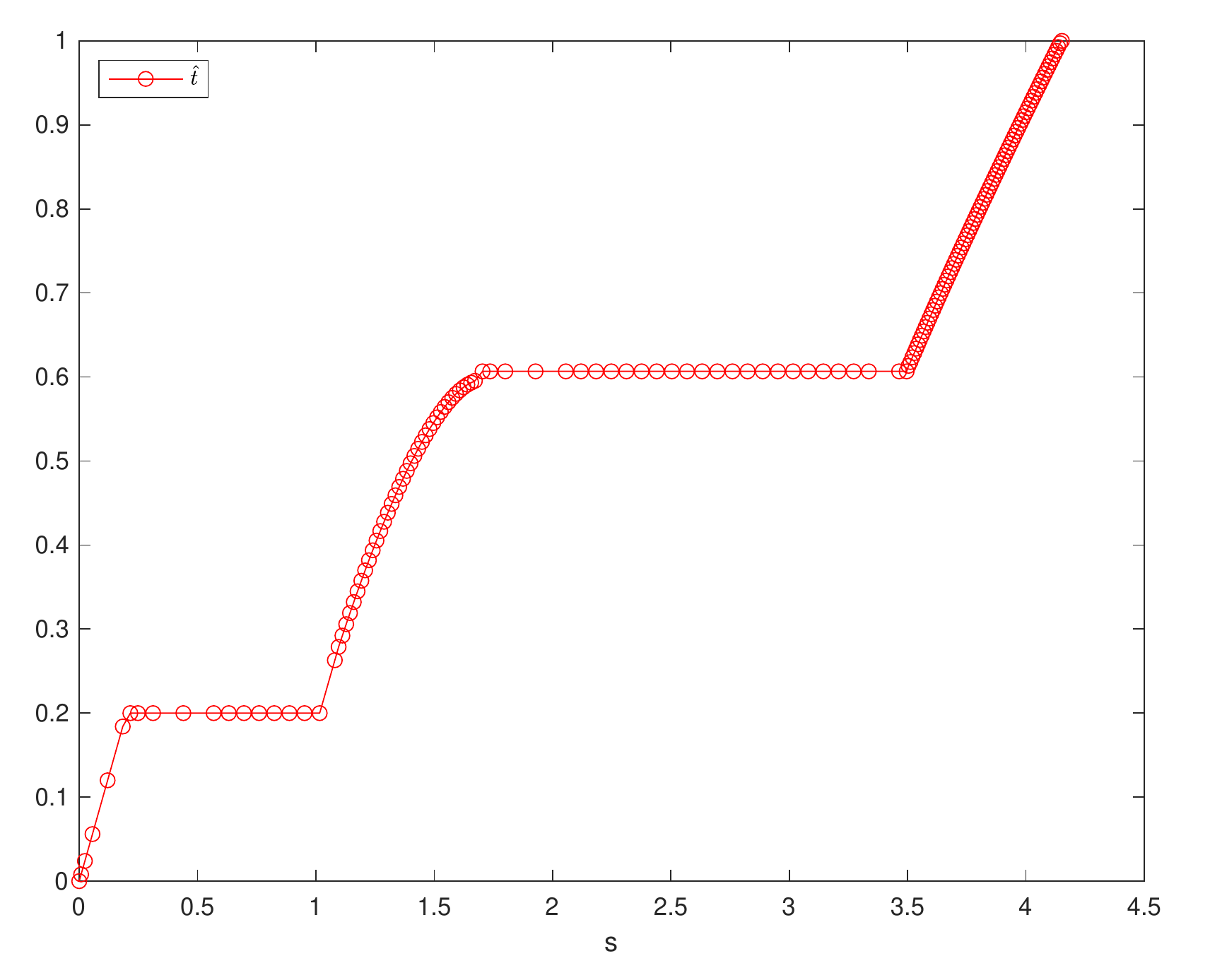}
    \subcaption{ $\hat t$ for $\tol = 10^{-5}$}
\end{subfigure}
\begin{subfigure}[c]{0.49\textwidth}
    \centering
 	\includegraphics[width=0.9\textwidth]{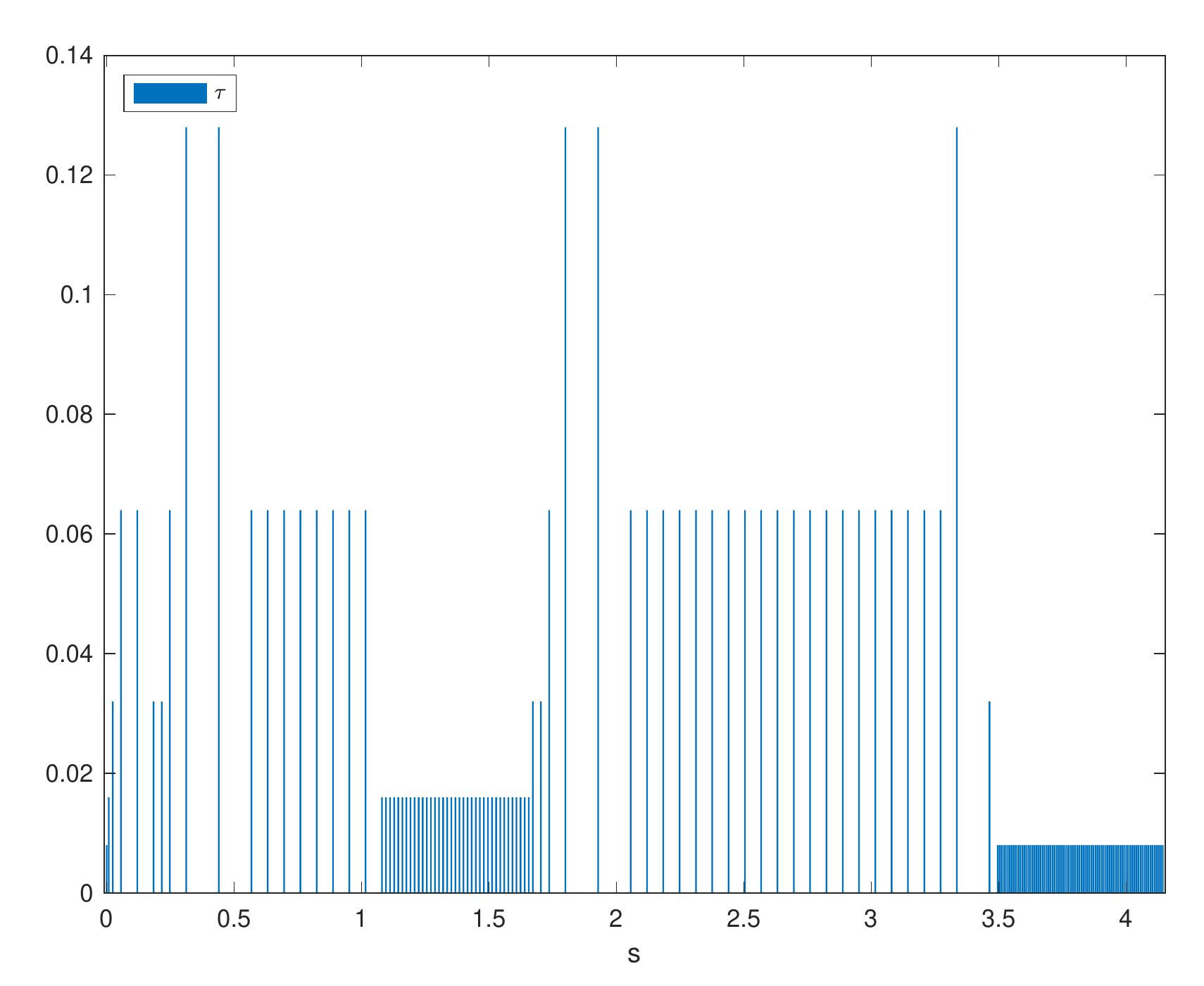}
    \subcaption{$\tau_k$ for $\tol = 10^{-5}$}
\end{subfigure}
\begin{subfigure}[c]{0.49\textwidth}
    \centering
 	\includegraphics[width=0.9\textwidth]{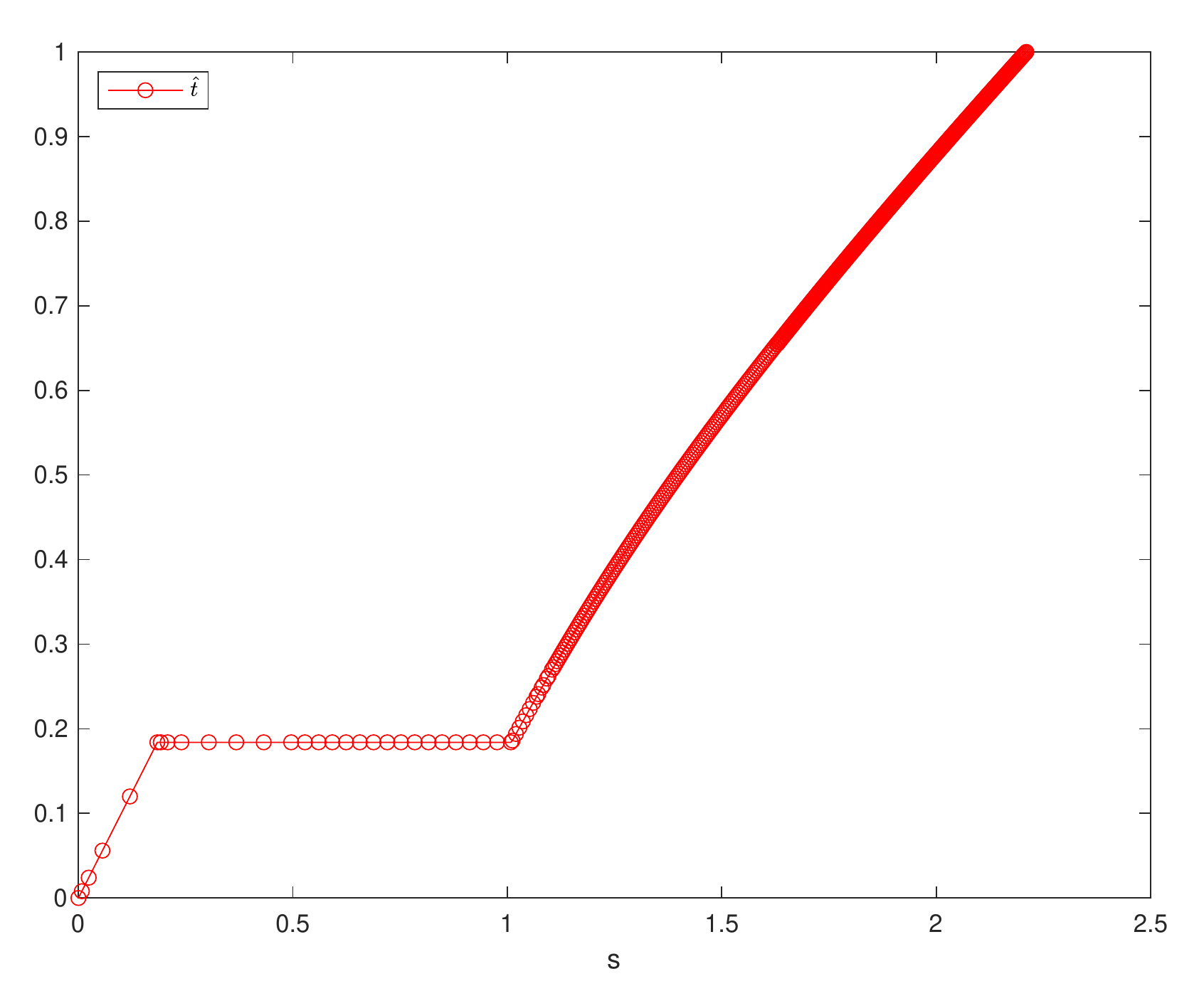}
    \subcaption{$\hat t$ for $\tol = 10^{-6}$}
\end{subfigure}
\begin{subfigure}[c]{0.49\textwidth}
    \centering
 	\includegraphics[width=0.9\textwidth]{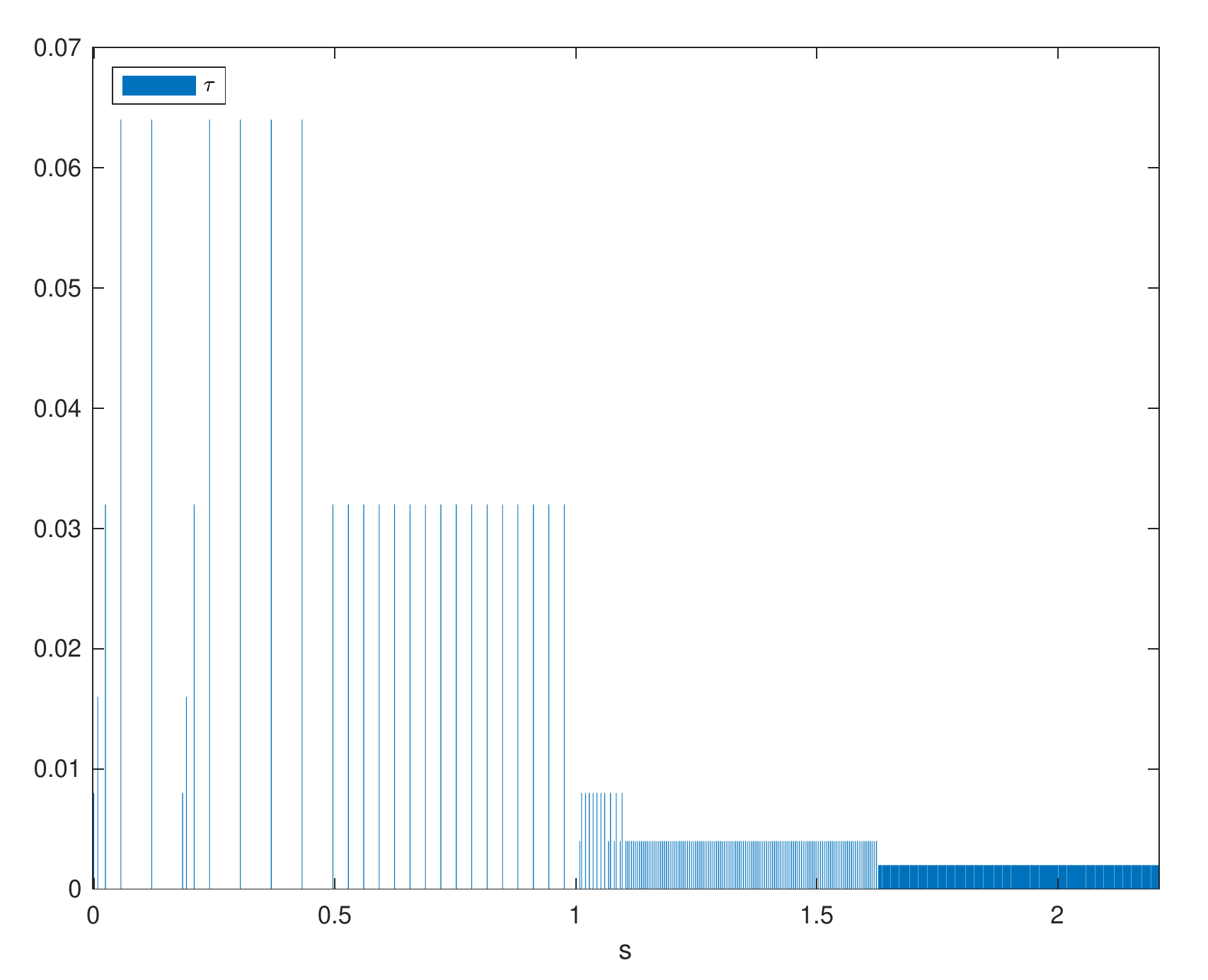}
    \subcaption{$\tau_k$ for $\tol = 10^{-6}$}
\end{subfigure}
\caption{Physical time and step size for $\tol = 10^{-4}$, $\tol = 10^{-5}$ and $\tol = 10^{-6}$}
\label{fig:pde}
\end{figure}
We first observe that, for $\tol = 10^{-5}$, seemingly another solution with an additional viscous jump 
is approximated than for the other tolerances. 
Note that, due to the nonconvexity of the double-well-type potential in \eqref{eq:doublewell}, 
there may well be multiple stationary points fulfilling \eqref{alg1} and \eqref{alg2} and 
it is just a matter of chance, which one is computed by the globalized semi-smooth Newton method.
Nevertheless, we again observe an increased step size during sticking and in viscous jumps.

Due to the ambiguity of numerical solutions and since an exact solution is unknown in this case, 
there is no error to evaluate and we investigate the residua in from \eqref{i1} and \eqref{i2} instead 
in order to compare the adaptive with the uniform scheme.
Figure \ref{fig:stepsresiduum} shows the local and the global residuum defined by
\begin{equation*}
    \max_{k = 1, ..., N_{\tol}} I_1^k + I^2_k 
    \quad \text{and} \quad 
    \sum_{k=1}^{N_{\tol}} I_1^k + I^2_k
\end{equation*}
for the adaptive and the uniform scheme. 
Note that both, $I_1^k$ and $I^2_k$, are non-negative by Lemma~\ref{lem:I2pos}.
Again, we observe the same qualitative behavior, but the adaptive scheme only needs approximately half of the steps 
to reach a residuum of the same size. The reduction of iterations is smaller compared to the previous
one-dimensional example, because the portion of the rate-independent slip regime is substantially larger 
compared to the example in Section~\ref{sec:ex1d}.
The kinks in the curves in Figure~\ref{fig:stepsresiduum} are due to the aforementioned ambiguity 
of numerical solutions with an additional viscous jump, which leads to 
an enlarged artificial time horizon and in this way increases the number of iterations without reducing the 
residuum.
\begin{figure}[h]
\begin{subfigure}[c]{0.49\textwidth}
    \centering
 	\includegraphics[width=\textwidth]{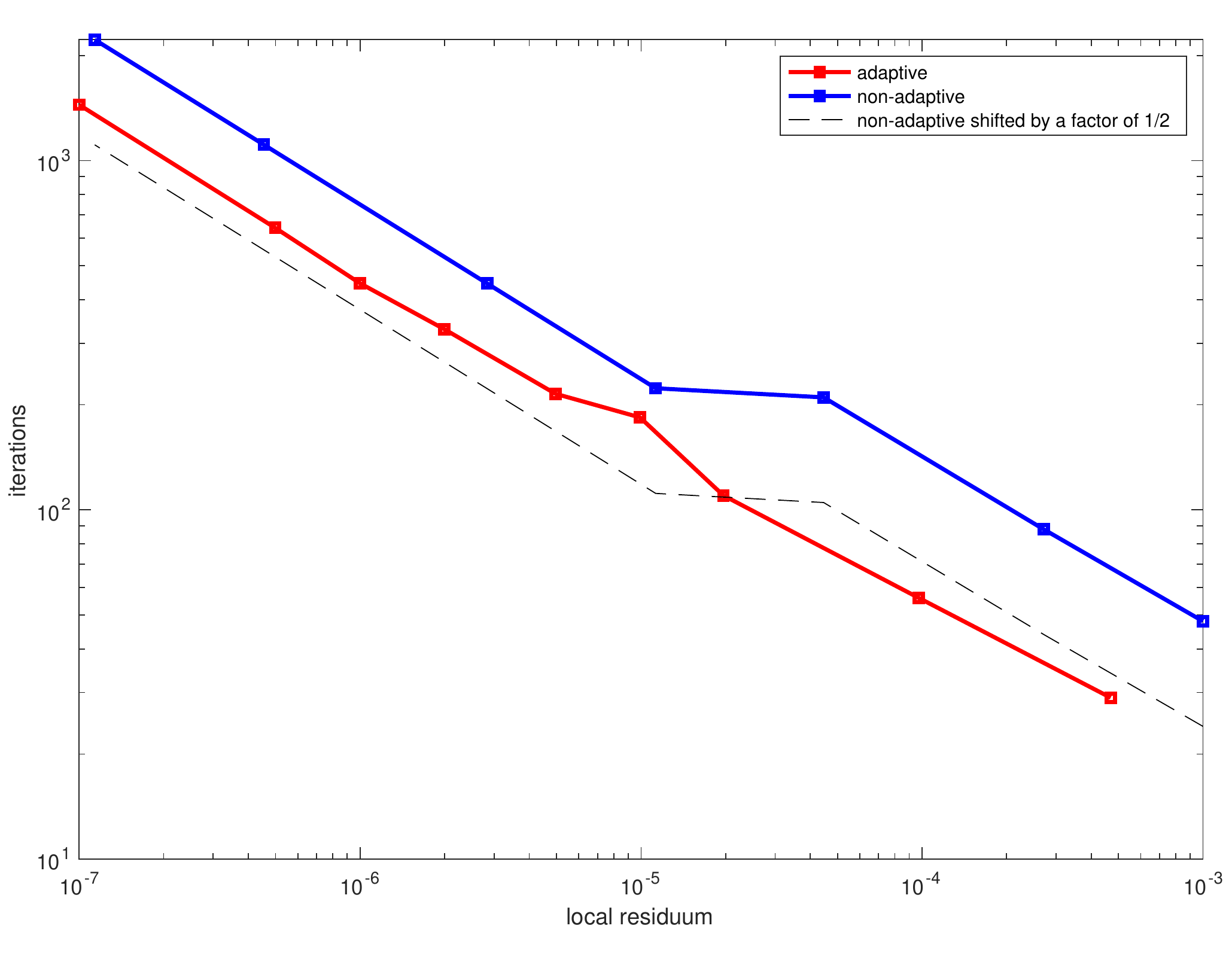}
\end{subfigure}
\begin{subfigure}[c]{0.49\textwidth}
    \centering
 	\includegraphics[width=\textwidth]{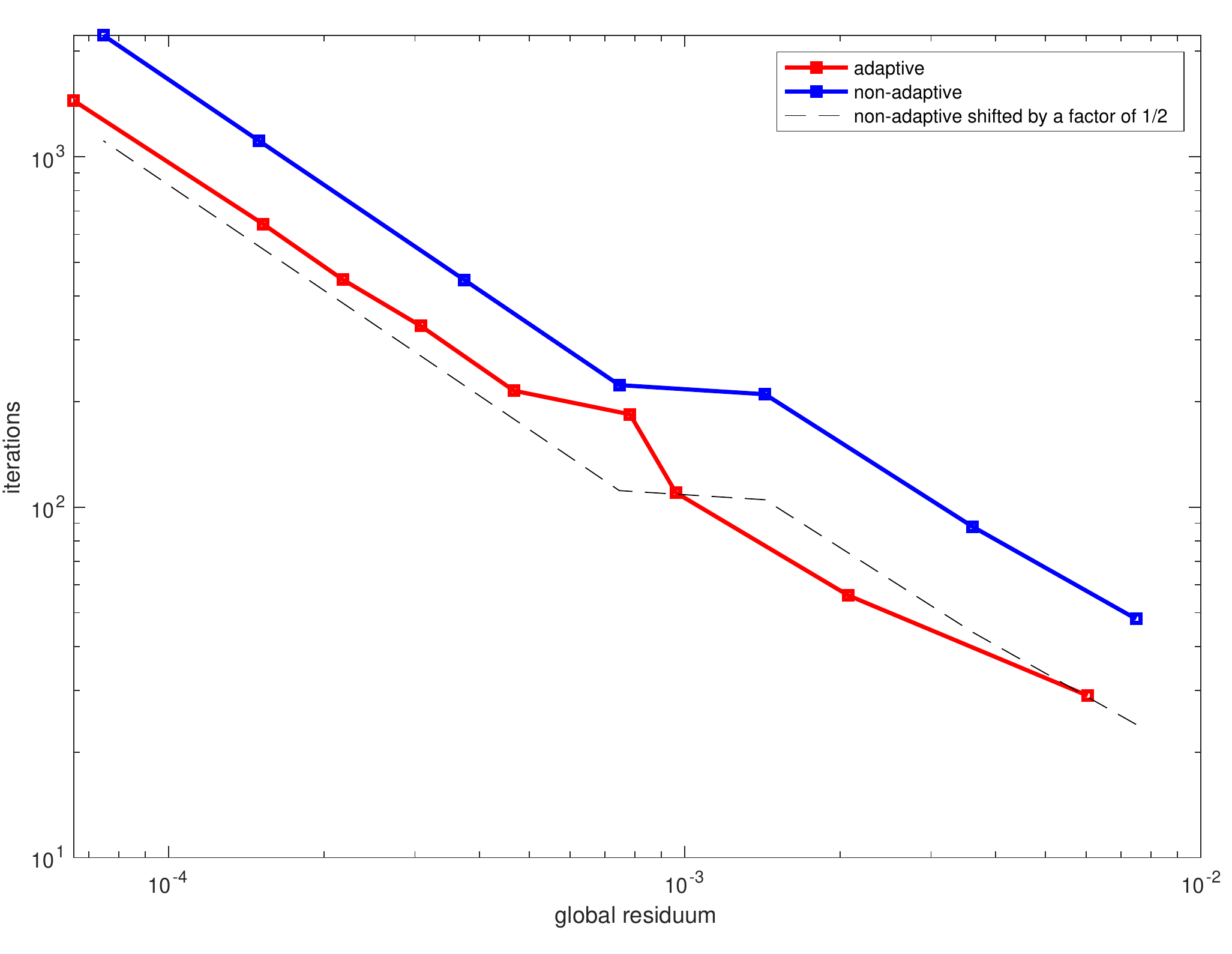}
\end{subfigure}
\caption{Comparison of number of iterations between adaptive and non-adaptive scheme
in the PDE example}\label{fig:stepsresiduum}
\end{figure}

\subsection*{Acknowledgement}
The authors are very greatful to Christian Kreuzer and Michael Sievers for several helpful discussions and advice.


\begin{appendix}

\section{Proof of Lemma~\ref{lem:I1I2}}\label{sec:I1I2}
    We start by proving \eqref{rabsch}. For this purpose, let $k\in \N$ and $s\in  (s_{k-1},s_k)$ be arbitrary.
	We first observe that the definition of the affine and constant interpolants
	imply for all $s\in (s_{k-1},s_k)$  that
	\begin{equation}
		\hat{z}'(s)=\frac{\hat{z}(s)-\bar{z}(s)}{s-s_k} \quad \text{and} \quad
		\hat{t}'(s)=\frac{\hat{t}(s)-\underline{t}(s)}{s-s_{k-1}}. \label{xy}
	\end{equation}
    With this at hand, we obtain
    \begin{equation}\label{eq:rest}
    \begin{aligned}
		r(s) &= \langle D_z\II(\hat{t}(s),\hat{z}(s)-D_z\II(\underline{t}(s),\bar{z}(s),\hat{z}'(s)\rangle_{\ZZ^*,\ZZ} \\
		&= -|s-s_k|\langle A\hat{z}'(s),\hat{z}'(s)\rangle_{\ZZ^*,\ZZ}
		+\Big\langle D_z\FF(\hat{z}(s))-D_z\FF(\bar{z}(s)),\frac{\hat{z}(s)-\bar{z}(s)}{s-s_k}\Big\rangle_{\VV^*,\VV}\\
		&\quad\; -\langle D_zf(\hat{t}(s),\hat{z}(s))-D_zf(\underline{t}(s),\bar{z}(s)),\hat{z}'(s)\rangle_{\VV^*,\VV}\\
		& \leq  - \alpha\,|s-s_k|\,\|\hat{z}'(s)\|_\ZZ^2
		+\frac{1}{|s-s_k|}\,|\langle D_z\FF(\bar{z}(s))-D_z\FF(\hat{z}(s)),\bar{z}(s)-\hat{z}(s)\rangle_{\VV^*,\VV}|\\
		&\quad\; +|\langle D_zf(\underline{t}(s),\bar{z}(s))-D_zf(\hat{t}(s),\hat{z}(s)),\hat{z}'(s)\rangle_{\VV^*,\VV}|.    
    \end{aligned}
    \end{equation}
	Thanks to \eqref{estF} with $\epsilon=\frac{\alpha}{4}$, 
	which is applicable here, since the interpolants are pointwise bounded on account of \eqref{zbeschr}, we conclude
    \begin{equation}\label{eq:Frest}
    \begin{aligned}
		& \frac{1}{|s-s_k|}|\langle D_z\FF(\bar{z}(s))-D_z\FF(\hat{z}(s)),\bar{z}(s)-\hat{z}(s)\rangle_{\VV^*,\VV}| \\
		& \qquad\qquad \leq \frac{\alpha}{4}\,|s-s_k|\,\|\hat{z}'(s)\|_\ZZ^2
		+ C_\alpha|s-s_k|\,\RR(\hat{z}'(s))\,\|\hat{z}'(s)\|_\VV.    
    \end{aligned}        
    \end{equation}
	Moreover, \eqref{estf} with $\epsilon=\frac{\alpha}{4}$ and \eqref{xy} yield
    \begin{equation}\label{eq:frest}
    \begin{aligned}
		& |\langle D_zf(\underline{t}(s),\bar{z}(s))-D_zf(\hat{t}(s),\hat{z}(s)),\hat{z}'(s)\rangle_{\VV^*,\VV}| \\
		&\qquad \qquad \leq\nu|s-s_{k-1}|\, \hat{t}'(s)\, \|\hat{z}'(s)\|_\VV
		+ c_\alpha |s-s_k|\, \RR(\hat{z}'(s))\|\hat{z}'(s)\|_\VV+\frac{\alpha}{4}\,|s-s_k|\, \|\hat{z}'(s)\|_\ZZ^2.
    \end{aligned}        
    \end{equation}
    Inserting \eqref{eq:Frest} and \eqref{eq:frest} in \eqref{eq:rest} and using the upper bound for the dissipation
    potential by assumption~\eqref{R3} then gives
	\begin{equation*}
		r(s)\leq C ( |s-s_k|\, \RR(\hat{z}'(s))\|\hat{z}'(s)\|_\V+|s-s_{k-1}|\,\hat{t}'(s)\|\hat{z}'(s)\|_\V) 
		\leq  C\, \tau_k\, \|\hat{z}'(s)\|_\V \big(\|\hat{z}'(s)\|_\V+\hat{t}'(s)\big),
	\end{equation*}
	which, in view of \eqref{beschr}, implies \eqref{rabsch}.
		
	Finally, \eqref{rabsch} together with the boundedness of the distance terms by \eqref{distbeschr} and 
	\eqref{distbeschr2} and the boundedness of $\hat{t}'(s)$ and $\|\hat z'(s)\|_{\V}$ by \eqref{beschr} 
	gives \eqref{e1} and \eqref{e2}.\hfill$\square$


\section{Proof of Lemma~\ref{tzbeschr}}\label{sec:tzbeschr}

	According to \eqref{zbeschr}, we already know that $\|\hat z\|_{L^\infty(0,s_k;\ZZ)} \leq C$.
	Hence, all we have to show is the boundedness of $	\hat{z}'_n$ in $L^2(0,s_k;\ZZ)$. 
	For this reason, we return to \eqref{q1} to conclude that, for every $i\in \N,$
	\begin{align*}
			&\lambda_{i+1}\|z_{i+1}-z_i\|_\V^2-\lambda_i\|z_i-z_{i-1}\|_\V\|z_{i+1}-z_i\|_\V
			+\frac{\alpha}{2} \|z_{i+1}-z_i\|_\ZZ^2\\
			& \qquad \leq (C_{\alpha/4} + c_{\alpha/4})\RR(z_{i+1}-z_i)\|z_{i+1}-z_i\|_\VV
			+\nu\,(t_i-t_{i-1})\|z_{i+1}-z_i\|_\VV\\
			& \qquad \leq C\, \tau_{i+1}\big(\RR(z_{i+1}-z_i)+t_i-t_{i-1}\big),
	\end{align*}
    where we exploited $\|z_{i+1}-z_i\|_\V\leq\tau_{i+1}$, cf.~Remark~\ref{rem:time}. 
    Now rearranging the terms and using \eqref{eig1} yields 
	\begin{align*}
		&\frac{\alpha}{2} \|z_{i+1}-z_i\|_\ZZ^2 \\
		&\quad \leq C\, \tau_{i+1}\big(\RR(z_{i+1}-z_i)+t_i-t_{i-1}\big)
		+\lambda_i\tau_{i+1}\|z_i-z_{i-1}\|_\V-\lambda_{i+1}\tau_{i+1}\|z_{i+1}-z_i\|_\V.
	\end{align*}
	From \eqref{q3} it follows by the same arguments
	\begin{equation*}
		\frac{\alpha}{2}\|z_1-z_0\|_\ZZ^2
		\leq C\,\tau_1\big(\RR(z_1-z_0)+\|D_z\II(0,z_0)\|_{\VV^*}\big)-\lambda_1\tau_1\|z_1-z_0\|_\V
	\end{equation*}
	Together with the estimate above, this implies
	\begin{align*}
		\|\hat{z}'\|_{L^2(0,s_k;\ZZ)}^2
		&\leq \sum_{i=0}^{k-1} \int_{s_i}^{s_{i+1}}\frac{\|z_{i+1}-z_i\|_\ZZ^2}{\tau_{i+1}^2}ds\\
		&=\sum_{i=0}^{k-1}\frac{\|z_{i+1}-z_i\|_\ZZ^2}{\tau_{i+1}} \\
		&\leq \frac{2}{\alpha}\Big( 
		\begin{aligned}[t]
		    & C\big(\RR(z_1-z_0)+\|D_z\II(0,z_0)\|_{\VV^*}\big)-\lambda_1\|z_1-z_0\|_\V \\
		    & +\sum_{i=1}^{k-1}C\big(\RR(z_{i+1}-z_i)+t_i-t_{i-1}\big)+\lambda_i\|z_i-z_{i-1}\|_\V
		    -\lambda_{i+1}\|z_{i+1}-z_i\|_\V\Big) 
		\end{aligned}\\
		&= \frac{2}{\alpha}\Big( C\,t_{k-1} -\lambda_k\|z_k-z_{k-1}\|_\V + C \,\|D_z\II(0,z_0)\|_{\VV^*}
		+ \sum_{i=0}^{k-1}C\,\RR(z_{i+1}-z_i) \Big) \\
		&\leq \frac{2}{\alpha}\,C \Big( T+\|D_z\II(0,z_0)\|_{\VV^*} + \sum_{i=0}^{k-1}\RR(z_{i+1}-z_i)\Big) 
		\leq C,
	\end{align*}
	where we used $\lambda_k\geq 0$ by Lemma~\ref{lem:stat} and 
	the boundedness of the dissipation by Lemma~\ref{irbeschr}.
\hfill$\square$


\section{Well Posedness of Algorithm~\ref{alg:nested}}\label{sec:nested}

To prove Proposition~\ref{prop:nested}, we have to verify Assumption~\ref{assu:nested} for 
Algorithm~\ref{alg:nested}.
First, we immediately observe that Assumption~\ref{assu:nested}\eqref{it:tolest} and \eqref{it:nested} are fulfilled 
by construction of the algorithm.

It thus remains to show that each inner iteration of Algorithm~\ref{alg:nested} 
(i.e., the iteration w.r.t.\ the index $k$) terminates after $N_n \in\N$ steps.
Too see this, we argue by induction.
For $n=1$, the inner iteration coincides with the one of Algorithm~\ref{alg:adaptlocmin} 
because of Steps~\ref{it:sigmagleichtau} and \ref{it:n1}.
Then, analogously to the proof of Proposition~\ref{endtime}, there holds
\begin{equation}\label{eq:tau1est}
    \tau^1_k \geq \min\Big\{\frac{\tol_1}{2\bar c}, \frac{\tau_1}{2}\Big\} := \tau^1_{\min}
    \quad \forall\, k \in \N
\end{equation}
and thus, an estimate as in \eqref{T} yields the existence of $N_1 \in \N$ such that $t^1_{N_1} > T$.
Now assume that, for some $n \geq 1$, there exists $N_n\in \N$ with $t^n_{N_n} > T$.
If we consider the iteration $n+1$, then an estimate of the form \eqref{eq:tau1est} (with $\tol_{n+1}$
instead of $\tol_1$) need not necessarily hold for all $k\in \N$ because of the modification of the step 
size in Step~\ref{it:taumod}. However, if, for some $k\in \N$, the step size 
$\tau^{n+1}_k$ arises from Step~\ref{it:taumod} and is less then 
\begin{equation*}
    \tau^{n+1}_{\min} := \min\Big\{\frac{\tol_{n+1}}{2\bar c}, \frac{\tau_1}{2}\Big\},
\end{equation*}
then, due to Lemma~\ref{lem:I1I2}, the conditions in Step~\ref{it:lesstol} will be met such that 
the iteration is accepted and we pass on to iteration $k+1$. The step size for this iteration is then
$\sigma^{n+1}_{k+1}$ 
(and \emph{not} $\tau^{n+1}_{k}$ as in the basic version of Algorithm~\ref{alg:adaptlocmin}).
By construction, $\sigma^{n+1}_j$, $j\in\N$, is  however only reduced, if $I^j_{1,n}$ and/or $I^j_{2,n}$ were above 
$\tol_{n+1}$, see Step~\ref{it:tauhalf}, and hence, $\sigma^{n+1}_j \geq \tau^{n+1}_{\min}$ for 
all $j\in \N$, again by Lemma~\ref{lem:I1I2}. 
These considerations show that a step size less than $\tau^{n+1}_{\min}$ can only 
appear at most $N_n$ times in the iteration $n+1$, which allows us to argue similarly to \eqref{T} as follows
\begin{equation}
\begin{aligned}
	t^{n+1}_k
	&= t_0^{n+1}+\sum_{i=1}^{k}\tau^{n+1}_i-\sum_{i=1}^{k}\|z^{n+1}_i-z^{n+1}_{i-1}\|_\V  \\
	&\geq (k - N_n)\, \tau^{n+1}_{\min}-\sum_{i=1}^{k}\|z^{n+1}_i-z^{n+1}_{i-1}\|_\V 
	\geq (k - N_n)\, \tau_{\min}^{n+1}-C ~\rightarrow \infty, \quad \text{as } k \to \infty,
\end{aligned}
\end{equation}
where we again used Lemma~\ref{lem:sumz}.
Note once more that the constant from Lemma~\ref{lem:sumz} does neither depend
on the tolerance nor on the step sizes.
This gives the existence of a constant $N_{n+1}$ such that $t^{n+1}_{N_{n+1}} > T$ as claimed.
\hfill$\square$


\section{A Lower Semicontinuity Result for the (Generalized) Distance}

\begin{lemma}\label{distuhs}
	Let $\{\eta_n\}_{n\in\N}\subset \ZZ^*$ be a sequence such that 
	$\eta_n\weak^* \eta$ in $\ZZ^*$ as $n\to\infty$ and $\dist_{\VV^*}\{-\eta_n,\partial\RR(0)\}\leq C$ for all $n\in\N$ 
	with a constant $C>0$. Then the distance is weakly lower semicontinuous, i.e., 
	\begin{equation*}
		\dist_{\VV^*}\{-\eta,\partial\RR(0)\} \leq \liminf_{n\to\infty}\dist_{\VV^*}\{-\eta_{n},\partial\RR(0)\}.
	\end{equation*}
\end{lemma}

\begin{proof}
    The proof is similar to \cite[Theorem~3.9]{FElocmin} or \cite[Lemma~3.2.18]{michael}.
 	First, we show by the direct method that the infimum in the definition of the distance is attained for every $n\in \N$. 
 	To this end, let $n\in \N$ be arbitrary and $\{\xi_m^n\}_{m\in \N} \subset \partial \RR(0)$ be a minimizing sequence, 
 	i.e., $\|\xi_m^n + \eta_n\|_{\V^{-1}} \to \dist_{\VV^*}\{-\eta_{n},\partial\RR(0)\}$ as $m\to \infty$.
 	Then, due to the boundedness of the distance by the assumption, there is a subsequence, 
 	denoted by the same symbol, such that $\xi_m^n + \eta_n \weak^* w$ in $\VV^*$ as $m\to\infty$ and therefore
 	$\xi_m^n \weak^* w - \eta_n =: \xi_n$ in $\ZZ^*$. As a subset of $\ZZ^*$ the convex subdifferential reads
 	\begin{equation*}
 	    \partial \RR(0) = \{ z^* \in \ZZ^* : \dual{z^*}{z}_{\ZZ^*, \ZZ} \leq \RR(z) \; \forall\, z\in \ZZ\}
 	\end{equation*}
    (where we used the positive homogeneity of $\RR$) such that $\partial \RR(0)$ is closed w.r.t.\ weak*
    convergence in $\ZZ^*$, which in turn implies $\xi_n\in \partial\RR(0)$. 
    Moreover, the weak lower semicontinuity of $\|\cdot \|_{\V^{-1}}$ gives
 	$\|\xi_n + \eta_n\|_{\V^{-1}} \leq \dist_{\VV^*}\{-\eta_{n},\partial\RR(0)\}$, which shows the optimality of $\xi_n$.
 	
 	To show the result, we argue similarly. Assume the assertion is wrong such that there exist a subsequence
 	$\{\eta_{n_\ell}\}_{\ell\in \N}$ and $\varepsilon > 0$ such that 
 	\begin{equation}\label{eq:distcontra}
 	    \dist_{\VV^*}\{-\eta_{n_\ell},\partial\RR(0)\} 
 	    \leq \dist_{\VV^*}\{-\eta,\partial\RR(0)\}  - \varepsilon.
 	\end{equation}
 	If we define $w_{n_\ell}:=\xi_{n_\ell}+\eta_{n_\ell}$, then, by assumption, 
 	$w_{n_\ell}$ is bounded in $\VV^*$ and there exists a subsequence, denoted by the same symbol for simplicity, 
 	such that $w_{n_\ell} \weak^* w$ in $\VV^*$ and therefore
 	\begin{equation*}
 	    \xi_{n_\ell} = w_{n_\ell} - \eta_{n_\ell} \weak^* w - \eta =: \xi \quad \text{in }\ZZ^*.
 	\end{equation*}
 	Again the weak closedness of $\partial\RR(0)$ yields $\xi \in \partial\RR(0)$, 
 	which, together with the weak lower semicontinuity of $\| \cdot \|_{\V^{-1}}$, implies 
 	\begin{equation*}
 	    \dist_{\VV^*}\{-\eta,\partial\RR(0)\}  
 	    \leq \|\xi + \eta\|_{\V^{-1}} 
 	    \leq \liminf_{\ell\to\infty} \|\xi_{n_\ell} + \eta_{n_\ell}\|_{\V^{-1}} 
 	    =  \liminf_{\ell\to\infty} \dist_{\VV^*}\{-\eta_{n_\ell},\partial\RR(0)\}, 
 	\end{equation*}
 	which contradicts \eqref{eq:distcontra}.
\end{proof}

\end{appendix}

\bibliographystyle{acm} \bibliography{literatur}

\end{document}